\documentclass{birkjour}

\usepackage[utf8]{inputenc}
\usepackage[T1]{fontenc}
\usepackage[english]{babel}
\usepackage{mathtools,amsmath,amssymb,amsfonts,amsthm,mathrsfs}
\mathtoolsset{showonlyrefs}

\usepackage{geometry}
\usepackage{graphicx,color}
\usepackage{tikz,pgfplots,pgf}
\usepackage[caption=false]{epsfig}
\usepackage{subcaption}
\usepackage[font=small]{caption}
\usepackage[pdftex]{hyperref}

\usepgfplotslibrary{polar}

\newtheorem{theorem}{Theorem}[section]
\newtheorem{lemma}{Lemma}[section]
\newtheorem{proposition}{Proposition}[section]

\theoremstyle{definition}
\newtheorem{remark}{Remark}[section]

\numberwithin{equation}{section}

\newcommand{\ee}{\varsigma}

\begin{document}

\title[Periodic perturbations of central force problems]{Periodic perturbations of central force \\ problems and an application to \\ a restricted $3$-body problem}

\author[A.~Boscaggin]{Alberto Boscaggin}

\address{
Department of Mathematics ``Giuseppe Peano'', University of Torino\\
Via Carlo Alberto 10, 10123 Torino, Italy}

\email{alberto.boscaggin@unito.it}

\author[W.~Dambrosio]{Walter Dambrosio}

\address{
Department of Mathematics ``Giuseppe Peano'', University of Torino\\
Via Carlo Alberto 10, 10123 Torino, Italy}

\email{walter.dambrosio@unito.it}

\author[G.~Feltrin]{Guglielmo Feltrin}

\address{
Department of Mathematics, Computer Science and Physics, University of Udine\\
Via delle Scienze 206, 33100 Udine, Italy}

\email{guglielmo.feltrin@uniud.it}

\thanks{Work written under the auspices of the Grup\-po Na\-zio\-na\-le per l'Anali\-si Ma\-te\-ma\-ti\-ca, la Pro\-ba\-bi\-li\-t\`{a} e le lo\-ro Appli\-ca\-zio\-ni (GNAMPA) of the Isti\-tu\-to Na\-zio\-na\-le di Al\-ta Ma\-te\-ma\-ti\-ca (INdAM). The first and the third author are supported by INdAM--GNAMPA project ``Problemi ai limiti per l'equazione della curvatura media prescritta''. The third author is grateful to the other authors and to the University of Torino for the kind hospitality during the period in which the work was written.
\\
\textbf{Preprint -- October 2021}} 



\date{}

\dedicatory{}

\begin{abstract}
We consider a perturbation of a central force problem of the form
\begin{equation*}
\ddot x = V'(|x|) \frac{x}{|x|} + \varepsilon \,\nabla_x U(t,x), \quad x \in \mathbb{R}^{2} \setminus \{0\},
\end{equation*}
where $\varepsilon \in \mathbb{R}$ is a small parameter, $V\colon (0,+\infty) \to \mathbb{R}$ and $U\colon \mathbb{R} \times (\mathbb{R}^{2} \setminus \{0\}) \to \mathbb{R}$ are smooth functions, and $U$ is $\tau$-periodic in the first variable.
Based on the introduction of suitable time-maps (the radial period and the apsidal angle) for the unperturbed problem ($\varepsilon=0$) and of an associated non-degeneracy condition, we apply an higher-dimensional version of the Poincar\'{e}--Birkhoff fixed point theorem to prove the existence of non-circular $\tau$-periodic solutions bifurcating from invariant tori at $\varepsilon=0$.
We then prove that this non-degeneracy condition is satisfied for some concrete examples of physical interest (including the homogeneous potential $V(r)=\kappa/r^{\alpha}$ for $\alpha\in(-\infty,2)\setminus\{-2,0,1\}$). Finally, an application is given to a restricted $3$-body problem with a non-Newtonian interaction.
\end{abstract}

\maketitle

\section{Introduction}\label{section-1}

In this paper, we investigate the existence of periodic solutions for systems of differential equations of the form
\begin{equation}\label{eq-intro}
\ddot x = V'(|x|) \frac{x}{|x|} + \varepsilon \,\nabla_x U(t,x), \quad x \in \mathbb{R}^{2} \setminus \{0\},
\end{equation}
where $\varepsilon \in \mathbb{R}$ is a parameter, $V\colon (0,+\infty) \to \mathbb{R}$ and $U\colon \mathbb{R} \times (\mathbb{R}^{2} \setminus \{0\}) \to \mathbb{R}$ are (smooth) potentials, and $U$ is $\tau$-periodic in the first variable. The above system appears as a perturbation of a central force problem and we are actually interested in the existence of solutions bifurcating, for $\varepsilon \to 0$, from the set of $\tau$-periodic solutions of the unperturbed problem ($\varepsilon = 0$). 

In this setting, a typical strategy consists in looking for periodic solutions bifurcating from the set of circular solutions of the unperturbed problem. For instance, in \cite{AmCo-89} the existence of nearly-circular periodic solutions is proved for the singular problem
\begin{equation}
\ddot x = -\kappa \dfrac{x}{|x|^{\alpha + 2}} + \varepsilon \,\nabla_x U(t,x), 
\end{equation}
where $\kappa > 0$ and $\alpha \in (0,2)$. The approach in \cite{AmCo-89} is variational, finding critical points of the associated action functional via an abstract perturbation theorem previously established in \cite{AmCoEk-87}, which requires as a crucial hypothesis that the manifold of critical points for the unperturbed action functional is non-degenerate, in a suitable sense. 
Notice that essentially no assumptions on the perturbation term are made in the non-Newtonian case (that is, $\alpha \neq 1$), while, due to the degeneracy of the Kepler problem, some symmetry conditions on $U$ are required when $\alpha = 1$. 
For further results in a similar spirit see \cite{FoGa-18,FoTo-08} and the references therein.

As it is well-known, however, the central force problem
\begin{equation}\label{eq-cenint}
\ddot x = V'(|x|) \frac{x}{|x|} 
\end{equation}
could have also a great variety of non-circular periodic solutions. This is the case, for instance, for the $-\alpha$-homogeneous central force problem 
\begin{equation}\label{eq-homintro}
\ddot x = -\kappa \dfrac{x}{|x|^{\alpha + 2}},
\end{equation}
where $\alpha \in (0,2)$ (see \cite[p.~7]{AmCo-93} and \cite[Section~2.8]{Ar-89}). Notice that, if it exists, a non-circular periodic solution to \eqref{eq-cenint} is actually part of a manifold which is at least two-dimensional, containing in fact all its time-translations and space-rotations. This is a consequence of the time-invariance and rotational-invariance of any central force problem in the plane, corresponding in turn to the existence of two first integrals: the energy and the angular momentum. Incidentally, we observe that, for the Kepler problem (that is, 
system \eqref{eq-homintro} for $\alpha = 1$) the manifold of periodic solutions of a fixed period is actually three-dimensional, containing all the Keplerian ellipses with fixed major semi-axis and, so, the circular solutions as a special case. On the contrary, this is not the case for system 
\eqref{eq-homintro} with $\alpha \in (0,2) \setminus \{1\}$, for which the set of non-circular periodic solutions with a given period is the union of two-dimensional tori (see again \cite[Section~2.8]{Ar-89}). To the best of our knowledge, the possibility of using tools of nonlinear analysis to bifurcate from non-circular periodic solutions of system \eqref{eq-homintro}, and more in general of system \eqref{eq-cenint}, has not yet been explored. Within this context, indeed, a serious difficulty consists in checking that a suitable non-degeneracy condition for the (at least two-dimensional) manifold of solutions holds true. 

Alternatively, the existence of periodic solutions of a system like \eqref{eq-intro} can be tackled by using dynamical system techniques. Indeed, it is well-known that, under mild assumptions, the central force problem \eqref{eq-cenint} can be regarded as a completely integrable Hamiltonian system with two degrees of freedom, so that problem \eqref{eq-intro} is interpreted as a perturbation of an integrable Hamiltonian system. Then, in a spirit which can be meant as a periodic counterpart of KAM theory, one could look for periodic solutions bifurcating from invariant tori of the unperturbed problem, see \cite{BeKa-87,Ch-92} as well as the more recent contribution \cite{FoGaGi-16} relying on a higher-dimensional version of the Poincar\'{e}--Birkhoff fixed point theorem \cite{FoUr-17} (see also \cite{FoSaZa-12} for similar results in the case of one degree of freedom).
Within this approach, the difficulty is now that action-angle coordinates $(I,\varphi)$ 
for the unperturbed problem must be constructed in order to check that the KAM non-degeneracy condition
\begin{equation}\label{eq-nondeg}
\mathrm{det\,} \nabla^{2} \mathcal{K}(I) \neq 0
\end{equation}
is satisfied, where $\mathcal{K}$ is the Hamiltonian of the unperturbed problem in action-angle coordinates. 
A successful attempt in this direction is recently given in \cite{BoDaFe-pp} for the perturbed relativistic Kepler problem
\begin{equation*}
\frac{\mathrm{d}}{\mathrm{d}t}\left( \frac{\dot x}{\sqrt{1 - |\dot x|^{2}/c^{2}}}\right) = -\kappa \, \frac{x}{|x|^{3}} + \varepsilon \,\nabla_x U(t,x),
\end{equation*}
where an explicit formula for the Hamiltonian in action-angle coordinates is provided. For the homogeneous central force problem \eqref{eq-homintro}, and more in general for problem \eqref{eq-cenint}, however, this does not seem to be possible.

The main aim of this paper is to provide, in the general setting of problem \eqref{eq-intro}, a more explicit non-degeneracy condition for the application of the above mentioned Hamiltonian perturbation approach.
The crucial ingredients for this are two suitable functions $T(H,L)$ and $\Theta(H,L)$, where $H$ and $L$ denote the energy and the angular momentum of the unperturbed problem, respectively. In more detail (see Section~\ref{section-2.1} for the precise definitions)
the function $T$ is nothing but the time-map for the radial component of a solution with energy $H$ and angular momentum $L$, while
the function $\Theta$ is the so-called apsidal angle, namely the angular variation of the solution in a radial period (in turn, such a function $\Theta$ can be meant as a time-map of a suitable oscillator, as well). The relation between the time-maps $T$ and $\Theta$ and the action-angle variables for the central force problem \eqref{eq-cenint} is independent on the specific choice of the potential $V$, as discussed in Section~\ref{section-2.2}. Based on these functions $T$ and $\Theta$, the non-degeneracy condition proposed in this paper reads as
\begin{equation}\label{eq-nondeg2}
\partial_{H} T(H,L) \cdot \partial_{L} \Theta(H,L) - \partial_{L} T(H,L) \cdot \partial_{H} \Theta(H,L) \neq 0.
\end{equation}
Actually, condition \eqref{eq-nondeg2} is nothing but an equivalent formulation of the usual non-degeneracy condition \eqref{eq-nondeg}, having however the advantage of depending only on the potential $V$ appearing in system \eqref{eq-cenint} and, instead, not requiring the explicit knowledge of the Hamiltonian in action-angle coordinates. With this in mind, our main result reads as follows (see Theorem~\ref{th-main} for a more precise statement).

\begin{theorem}\label{th-intro}
Let us assume that there exists a $\tau$-periodic non-circular solution $x^{*}$ to \eqref{eq-cenint} with energy $H^{*}$ and angular momentum $L^{*}$. 
If condition \eqref{eq-nondeg2} holds true at $(H^{*},L^{*})$, then for $\varepsilon$ small enough there exist at least three $\tau$-periodic solutions to system~\eqref{eq-intro}.
\end{theorem}

The details of the proof of the above theorem are given in Section~\ref{section-3}. Here, we just emphasize that these three $\tau$-periodic solutions to \eqref{eq-intro} are obtained, via the Poincar\'{e}--Birkhoff fixed point theorem, bifurcating from the invariant torus of the unperturbed problem containing the $\tau$-periodic non-circular solution $x^{*}$, as well as its time-translations and space-rotations. 
We also mention that a version of this result can be stated for perturbations of central force problems driven by the relativistic operator, namely
\begin{equation*}
\frac{\mathrm{d}}{\mathrm{d}t}\left( \frac{\dot x}{\sqrt{1 - |\dot x|^{2}/c^{2}}}\right) =  V'(|x|) \frac{x}{|x|} + \varepsilon \,\nabla_x U(t,x),
\end{equation*}
as well (cf.~Remark~\ref{rem-relativistic}).

In Section~\ref{section-4}, the effectiveness of the non-degeneracy condition \eqref{eq-nondeg2} is investigated in some concrete problems of physical interest. As a main example (see Section~\ref{section-4.2}), we deal with the homogeneous problem \eqref{eq-homintro}, in the full range $\alpha \in (-\infty,2)$ with $\alpha \neq 0$, thus considering both the case of Keplerian-like problems (that is, $\alpha > -1$) and the case of sublinear/superlinear oscillators (that is, $\alpha \leq -1$); incidentally, we recall that, for $\alpha \geq 2$, system \eqref{eq-homintro} does not possess non-circular periodic solutions. For a generic choice of $\alpha$, the explicit computation
of the time-map $T$ and $\Theta$ and of their derivatives is not possible. However, taking advantage of the homogeneity of the problem, it is possible to reduce the study of the sign of the left-hand side of \eqref{eq-nondeg2} to the one of $\partial_{H} \Theta$, that is, to the strict monotonicity of the apsidal angle as a function of the energy. This issue is investigated in \cite{Ca-15,Ro-18} and the results therein allow us to conclude that condition \eqref{eq-nondeg2} holds true provided $\alpha \neq -2$ (the harmonic oscillator) and $\alpha \neq 1$ (the Kepler problem). 
In both these excluded cases, the function $\Theta$ is in fact constant, corresponding to the degenerate picture in which all bounded orbits are periodic, as provided by the celebrated Bertrand's theorem (cf.~\cite{Be-1873}). For perturbations of the homogeneous problem \eqref{eq-homintro}, our result thus seems to be essentially sharp. Indeed, by the linear theory of ODEs, it is well known that periodic solutions to the harmonic oscillator cannot be provided for any perturbation term, due to resonance phenomena. On the other hand, perturbations of the Kepler problem have been investigated by assuming further condition on the perturbation term (as in \cite{AmCo-89}, see also \cite{FoGa-17, GiDe-89, SeTe-94}) or meaning the solutions in a generalized sense, allowing collisions with the singularity, see \cite{BaOrVe-21,BoDaPa-20, BoOrZh-19}. 

Other examples of applications are provided in Section~\ref{section-4.1} and Section~\ref{section-4.3}, dealing with the Levi-Civita equation 
(cf.~\cite{LeCi-28})
\begin{equation*}
\ddot x = -\kappa \dfrac{x}{|x|^{3}} - 2\lambda \dfrac{x}{|x|^4}, \quad \kappa,\lambda > 0,
\end{equation*}
and with the Lennard-Jones equation (cf.~\cite{LJ-31})
\begin{equation*}
\ddot x = -24 \ee \sigma^6 \frac{x}{|x|^8} + 48 \ee \sigma^{12} \frac{x}{|x|^{14}}, \quad \ee,\sigma > 0,
\end{equation*}
respectively, see the introduction of Section~\ref{section-4} for more details. In the first case, an explicit computation for the
left-hand side of \eqref{eq-nondeg2} is provided and the conclusion is easily obtained. In the latter, instead, we study the asymptotic behaviour of the left-hand side of \eqref{eq-nondeg2} as $H$ goes to its minimum admissible value,
eventually proving that condition \eqref{eq-nondeg2} is satisfied locally. It is worth noticing that for this proof we need to provide
a series of auxiliary results for the derivative of a time-map with respect to a parameter. This analysis is partially inspired by the well-known papers by Chicone and Schaaf \cite{Ch-87,Sc-85} and is given in a final Appendix, which hopefully may have some independent interest.

In Section~\ref{section-5}, we finally present an application of our approach to a restricted planar circular $3$-body problem
with a non-Newtonian interaction, precisely
\begin{equation*}
\ddot q = GM \dfrac{\xi_{M}(t) - q}{| \xi_{M}(t) - q |^{\alpha+2}} + Gm \dfrac{\xi_{m}(t) - q}{| \xi_{m}(t) - q |^{\alpha+2}}, \quad q \in \mathbb{R}^{2},
\end{equation*}
where $G$ is the gravitational constant, $\xi_M$ and $\xi_m$ are the heavy bodies with masses $M,m>0$, and $\alpha \in (0,2)$ with $\alpha \neq 1$. After a suitable change of variable, such a problem is interpreted as a perturbation, when $m \to 0$, of the homogeneous problem \eqref{eq-homintro}, so that Theorem~\ref{th-intro} can be applied. More precisely, by carefully analyzing the set of periodic solutions of the unperturbed problem so as to select the invariant tori for the bifurcation, we show that for any choice of the minimal period of the primaries a plethora of periodic solutions to the restricted problem can be provided when $m$ is small enough. 

To conclude this introduction, we briefly mention that from the non-degeneracy of the unperturbed problems proved in the paper other consequences can be derived. For instance, standard KAM theorems (see \cite[Chapter~10]{Ar-89} and \cite{Du-14}) can be applied to ensure the persistence of quasi-periodic invariant tori for small and sufficiently smooth perturbation terms.

\section{Preliminaries on central force problems }\label{section-2}

In this section, we summarize some known facts about the central force problem
\begin{equation}\label{eq-centr}
\ddot x = V'(|x|) \frac{x}{|x|}, \quad x \in \mathbb{R}^{2} \setminus \{0\},
\end{equation}
where $V\colon(0,+\infty)\to\mathbb{R}$ is a $\mathcal{C}^{2}$-function.

We recall that system \eqref{eq-centr} has two natural first integrals, namely the \textit{energy}
\begin{equation}\label{energy}
H = \dfrac{1}{2} |\dot x|^{2} - V(|x|)
\end{equation}
and the (scalar) \textit{angular momentum}
\begin{equation}\label{ang-mom}
L = \langle x, i \dot x \rangle.
\end{equation}
Restricting our analysis to pairs $(H,L)$ such that the corresponding solutions of \eqref{eq-centr} are bounded (periodic or quasi-periodic) and not rectilinear (that is, $L \neq 0$), in Section~\ref{section-2.1}, we analyze the dynamics in polar coordinates, so as to define two suitable functions $T$ and $\Theta$ of $(H,L)$ (see formulas \eqref{def-T-HL} and \eqref{def-Theta-HL}) which will play a crucial role in our main result. Then, in Section~\ref{section-2.2}, we adopt a Hamiltonian point of view and we present the construction of the action-angle coordinates in a neighbourhood of a fixed invariant torus with energy $H$ and angular momentum $L$.

Notice that, by reversing the time ($t \to -t$), it is not restrictive to assume that $L$ is positive. This will be tacitly assumed throughout the paper.

\subsection{Periodic solutions}\label{section-2.1}

Introducing the polar coordinates
\begin{equation*}
x(t) = r(t) \bigl{(} \cos(\vartheta(t)), \sin(\vartheta(t)) \bigr{)} = r(t) e^{i\vartheta(t)},
\end{equation*}
formula \eqref{ang-mom} reads as
\begin{equation*}
L = r^{2} \dot{\vartheta}.
\end{equation*}
Notice that, since $L > 0$, the solutions wind around the origin in the counter-clockwise sense.
In polar coordinates \eqref{energy} becomes
\begin{equation*}
H = \dfrac{1}{2}\dot{r}^{2} + W(r;L),
\end{equation*}
where
\begin{equation*}
W(r;L) = \dfrac{L^{2}}{2r^{2}} - V(r)
\end{equation*}
is the so-called \textit{effective potential}.

From now on, we assume that 
\begin{itemize}
\item [$(h_{W})$] \textit{there exists an open interval $J\subseteq(0,+\infty)$ such that for all $L\in J$ the function $W(\cdot;L)$ has a (strict) local minimum at $r=r_{0}(L)$ with $W'(\cdot;L)<0$ on $(r_{*}(L),r_{0}(L))$ and $W'(\cdot;L)>0$ 
on $(r_{0}(L),r^{*}(L))$, for some $0\leq r_{*}(L) < r_{0}(L) < r^{*}(L)\leq +\infty$.}
\end{itemize}
Setting
\begin{equation*}
w_{0}(L)=-W(r_{0}(L);L),
\end{equation*}
for all $L\in J$, by $(h_{W})$ there exists a value $H_{0}(L)>-w_{0}(L)$ such that for all $H\in(-w_{0}(L),H_{0}(L))$ the equation $W(r;L)=H$ has two solutions $r_{\pm}(H,L)$
such that
\begin{equation*}
r_{-}(H,L) \in (r_{*}(L),r_{0}(L)), \qquad
r_{+}(H,L) \in (r_{0}(L),r^{*}(L)).
\end{equation*}
For simplicity, we assume that $H_{0}$ depends continuously on $L$, so that the set
\begin{equation}\label{dom-HL}
\Lambda = \bigl{\{} (H,L)\in\mathbb{R}^{2} \colon L\in J, \; H\in (-w_{0}(L), H_{0}(L)) \bigr{\}}
\end{equation}
is open. For values $(H,L)\in\Lambda$, the corresponding orbits in the $(r,\dot r)$-plane are closed, winding around the point $(r_{0}(L),0)$ with minimal period given by
\begin{equation}\label{def-T-HL}
T(H,L) = \sqrt{2} \int_{r_{-}(H,L)}^{r_{+}(H,L)} \dfrac{\mathrm{d}r}{\sqrt{H-W(r;L)}}.
\end{equation}
See Figure~\ref{fig-01} for a graphical representation.

\begin{figure}[htb]
\begin{tikzpicture}
\begin{axis}[
  tick label style={font=\scriptsize},
  axis y line=left, 
  axis x line=middle,
  xtick={0},
  ytick={-0.304574,-0.2,0},
  xticklabels={},
  yticklabels={$-w_{0}(L)$,$H$,$0$},
  xlabel={\small $r$},
  ylabel={\small $W(r;L)$},
every axis x label/.style={
    at={(ticklabel* cs:1.0)},
    anchor=west,
},
every axis y label/.style={
    at={(ticklabel* cs:1.0)},
    anchor=south,
},
  width=6cm,
  height=6cm,
  xmin=0,
  xmax=8,
  ymin=-0.4,
  ymax=0.6]
\addplot [color=black,line width=0.9pt,smooth] coordinates {(0.1, 42.2567) (0.179, 11.1014) (0.258, 4.3267) (0.337, 1.94107) (0.416, 0.887245) (0.495, 0.357586) (0.574, 0.0694774) (0.653, -0.094917) (0.732, -0.191016) (0.811, -0.247283) (0.89, -0.279344) (0.969, -0.296244) (1.048, -0.303435) (1.127, -0.304323) (1.206, -0.301089) (1.285, -0.295166) (1.364, -0.287507) (1.443, -0.278761) (1.522, -0.269371) (1.601, -0.25964) (1.68, -0.24978) (1.759, -0.239936) (1.838, -0.230209) (1.917, -0.220666) (1.996, -0.211352) (2.075, -0.202295) (2.154, -0.193513) (2.233, -0.185014) (2.312, -0.1768) (2.391, -0.168869) (2.47, -0.161217) (2.549, -0.153838) (2.628, -0.146722) (2.707, -0.139861) (2.786, -0.133247) (2.865, -0.126868) (2.944, -0.120715) (3.023, -0.114779) (3.102, -0.109051) (3.181, -0.10352) (3.26, -0.0981792) (3.339, -0.0930192) (3.418, -0.088032) (3.497, -0.0832098) (3.576, -0.0785453) (3.655, -0.0740315) (3.734, -0.0696615) (3.813, -0.0654292) (3.892, -0.0613283) (3.971, -0.0573533) (4.05, -0.0534986) (4.129, -0.0497591) (4.208, -0.04613) (4.287, -0.0426065) (4.366, -0.0391842) (4.445, -0.0358591) (4.524, -0.0326271) (4.603, -0.0294844) (4.682, -0.0264275) (4.761, -0.023453) (4.84, -0.0205577) (4.919, -0.0177384) (4.998, -0.0149924) (5.077, -0.0123167) (5.156, -0.00970886) (5.235, -0.00716624) (5.314, -0.00468648) (5.393, -0.00226729) (5.472, 0.0000934928) (5.551, 0.00239795) (5.63, 0.00464806) (5.709, 0.00684572) (5.788, 0.00899272) (5.867, 0.0110908) (5.946, 0.0131416) (6.025, 0.0151467) (6.104, 0.0171076) (6.183, 0.0190258) (6.262, 0.0209026) (6.341, 0.0227393) (6.42, 0.0245373) (6.499, 0.0262977) (6.578, 0.0280218) (6.657, 0.0297105) (6.736, 0.0313651) (6.815, 0.0329865) (6.894, 0.0345757) (6.973, 0.0361337) (7.052, 0.0376614) (7.131, 0.0391597) (7.21, 0.0406294) (7.289, 0.0420713) (7.368, 0.0434863) (7.447, 0.044875) (7.526, 0.0462382) (7.605, 0.0475766) (7.684, 0.0488909) (7.763, 0.0501817) (7.842, 0.0514497) (7.921, 0.0526954) (8., 0.0539194)};
\draw [color=gray, dashed, line width=0.3pt] (axis cs:0,-0.304574)--(axis cs:1.1,-0.304574);
\draw [color=gray, dashed, line width=0.3pt] (axis cs:0,-0.2)--(axis cs:2.09541,-0.2);
\draw [color=gray, dashed, line width=0.3pt] (axis cs:0.742085,0.015)--(axis cs:0.742085,-0.2);
\draw [color=gray, dashed, line width=0.3pt] (axis cs:1.1,0.015)--(axis cs:1.1,-0.304574);
\draw [color=gray, dashed, line width=0.3pt] (axis cs:2.09541,0.015)--(axis cs:2.09541,-0.2);
\node at (axis cs: 1.4,0.05) {\scriptsize{$r_{0}(L)$}};
\draw [color=gray, line width=0.3pt] (axis cs: 0.742085,0)--(axis cs:4.8,-0.25);
\draw [color=gray, line width=0.3pt] (axis cs: 2.09541,0)--(axis cs:4.8,-0.18);
\node at (axis cs: 6,-0.18) {\scriptsize{$r_{+}(H,L)$}};
\node at (axis cs: 6,-0.25) {\scriptsize{$r_{-}(H,L)$}};
\fill (axis cs: 0.742085,0) circle (0.8pt);
\fill (axis cs: 2.09541,0) circle (0.8pt);
\end{axis}
\end{tikzpicture}
\quad\quad
\begin{tikzpicture}
\begin{axis}[
  tick label style={font=\scriptsize},
  axis y line=left, 
  axis x line=middle,
  xtick={1.2},
  ytick={0},
  xticklabels={},
  yticklabels={$0$},
  xlabel={\small $r$},
  ylabel={\small $\dot r$},
every axis x label/.style={
    at={(ticklabel* cs:1.0)},
    anchor=west,
},
every axis y label/.style={
    at={(ticklabel* cs:1.0)},
    anchor=south,
},
  width=6cm,
  height=6cm,
  xmin=0,
  xmax=4.5,
  ymin=-3.3,
  ymax=3.3]
\draw[->] [color=black,line width=0.9pt] (axis cs: 0.812, 1.94068)--(axis cs: 0.868, 2.27431);
\draw[->] [color=black,line width=0.9pt] (axis cs: 2.828, 1.86766)--(axis cs: 2.884, 1.79682);
\draw[->] [color=black,line width=0.9pt] (axis cs: 0.868, -2.27431)--(axis cs: 0.812, -1.94068);
\draw[->] [color=black,line width=0.9pt] (axis cs: 2.884, -1.79682)--(axis cs: 2.828, -1.86766);
\addplot [color=black,line width=0.9pt,smooth] coordinates {(0.7355, 0) (0.7368, 0.16485) (0.7384, 0.381708) (0.74, 0.511707) (0.7416, 0.612917) (0.7432, 0.698019) (0.7448, 0.772395) (0.7464, 0.838936) (0.748, 0.899413) (0.7496, 0.955007) (0.7512, 1.00655) (0.7528, 1.05467) (0.7544, 1.09983) (0.756, 1.1424) (0.7576, 1.18268) (0.7592, 1.22092) (0.7608, 1.25732) (0.7624, 1.29204) (0.764, 1.32525) (0.7656, 1.35706) (0.7672, 1.38759) (0.7688, 1.41692) (0.7704, 1.44515) (0.772, 1.47234) (0.7736, 1.49858) (0.7752, 1.52391) (0.7768, 1.54839) (0.7784, 1.57208) (0.78, 1.59501) (0.812, 1.94068) (0.868, 2.27431) (0.924, 2.45743) (0.98, 2.56894) (1.036, 2.64052) (1.092, 2.68754) (1.148, 2.71834) (1.204, 2.7378) (1.26, 2.74895) (1.316, 2.75377) (1.372, 2.75355) (1.428, 2.74922) (1.484, 2.74141) (1.54, 2.73059) (1.596, 2.71708) (1.652, 2.70114) (1.708, 2.68293) (1.764, 2.66258) (1.82, 2.6402) (1.876, 2.61583) (1.932, 2.58951) (1.988, 2.56126) (2.044, 2.53107) (2.1, 2.49894) (2.156, 2.46482) (2.212, 2.42868) (2.268, 2.39044) (2.324, 2.35005) (2.38, 2.30741) (2.436, 2.26241) (2.492, 2.21493) (2.548, 2.16483) (2.604, 2.11193) (2.66, 2.05602) (2.716, 1.99688) (2.772, 1.93421) (2.828, 1.86766) (2.884, 1.79682) (2.94, 1.72117) (2.996, 1.64004) (3.052, 1.55259) (3.108, 1.45768) (3.164, 1.35375) (3.22, 1.23855) (3.276, 1.10858) (3.3, 1.04701) (3.304, 1.03634) (3.308, 1.02554) (3.312, 1.01461) (3.316, 1.00354) (3.32, 0.992334) (3.324, 0.980983) (3.328, 0.969483) (3.332, 0.957828) (3.336, 0.946011) (3.34, 0.934028) (3.344, 0.921871) (3.348, 0.909534) (3.352, 0.897009) (3.356, 0.884288) (3.36, 0.871362) (3.364, 0.858223) (3.368, 0.84486) (3.372, 0.831263) (3.376, 0.817419) (3.38, 0.803317) (3.384, 0.788942) (3.388, 0.774279) (3.392, 0.759311) (3.396, 0.744021) (3.4, 0.728387) (3.404, 0.712387) (3.408, 0.695995) (3.412, 0.679185) (3.416, 0.661923) (3.42, 0.644173) (3.424, 0.625894) (3.428, 0.607038) (3.432, 0.587549) (3.436, 0.567362) (3.44, 0.546401) (3.444, 0.524571) (3.448, 0.50176) (3.452, 0.477827) (3.456, 0.452595) (3.46, 0.425832) (3.464, 0.397229) (3.468, 0.366355) (3.472, 0.332579) (3.476, 0.294905) (3.48, 0.251587) (3.484, 0.198975) (3.488, 0.12591) (3.49, 0)};
\addplot [color=black,line width=0.9pt,smooth] coordinates {(0.7355, 0) (0.7368, -0.16485) (0.7384, -0.381708) (0.74, -0.511707) (0.7416, -0.612917) (0.7432, -0.698019) (0.7448, -0.772395) (0.7464, -0.838936) (0.748, -0.899413) (0.7496, -0.955007) (0.7512, -1.00655) (0.7528, -1.05467) (0.7544, -1.09983) (0.756, -1.1424) (0.7576, -1.18268) (0.7592, -1.22092) (0.7608, -1.25732) (0.7624, -1.29204) (0.764, -1.32525) (0.7656, -1.35706) (0.7672, -1.38759) (0.7688, -1.41692) (0.7704, -1.44515) (0.772, -1.47234) (0.7736, -1.49858) (0.7752, -1.52391) (0.7768, -1.54839) (0.7784, -1.57208) (0.78, -1.59501) (0.812, -1.94068) (0.868, -2.27431) (0.924, -2.45743) (0.98, -2.56894) (1.036, -2.64052) (1.092, -2.68754) (1.148, -2.71834) (1.204, -2.7378) (1.26, -2.74895) (1.316, -2.75377) (1.372, -2.75355) (1.428, -2.74922) (1.484, -2.74141) (1.54, -2.73059) (1.596, -2.71708) (1.652, -2.70114) (1.708, -2.68293) (1.764, -2.66258) (1.82, -2.6402) (1.876, -2.61583) (1.932, -2.58951) (1.988, -2.56126) (2.044, -2.53107) (2.1, -2.49894) (2.156, -2.46482) (2.212, -2.42868) (2.268, -2.39044) (2.324, -2.35005) (2.38, -2.30741) (2.436, -2.26241) (2.492, -2.21493) (2.548, -2.16483) (2.604, -2.11193) (2.66, -2.05602) (2.716, -1.99688) (2.772, -1.93421) (2.828, -1.86766) (2.884, -1.79682) (2.94, -1.72117) (2.996, -1.64004) (3.052, -1.55259) (3.108, -1.45768) (3.164, -1.35375) (3.22, -1.23855) (3.276, -1.10858) (3.3, -1.04701) (3.304, -1.03634) (3.308, -1.02554) (3.312, -1.01461) (3.316, -1.00354) (3.32, -0.992334) (3.324, -0.980983) (3.328, -0.969483) (3.332, -0.957828) (3.336, -0.946011) (3.34, -0.934028) (3.344, -0.921871) (3.348, -0.909534) (3.352, -0.897009) (3.356, -0.884288) (3.36, -0.871362) (3.364, -0.858223) (3.368, -0.84486) (3.372, -0.831263) (3.376, -0.817419) (3.38, -0.803317) (3.384, -0.788942) (3.388, -0.774279) (3.392, -0.759311) (3.396, -0.744021) (3.4, -0.728387) (3.404, -0.712387) (3.408, -0.695995) (3.412, -0.679185) (3.416, -0.661923) (3.42, -0.644173) (3.424, -0.625894) (3.428, -0.607038) (3.432, -0.587549) (3.436, -0.567362) (3.44, -0.546401) (3.444, -0.524571) (3.448, -0.50176) (3.452, -0.477827) (3.456, -0.452595) (3.46, -0.425832) (3.464, -0.397229) (3.468, -0.366355) (3.472, -0.332579) (3.476, -0.294905) (3.48, -0.251587) (3.484, -0.198975) (3.488, -0.12591) (3.49, 0)};
\node at (axis cs: 1.3, -0.3) {\scriptsize{$r_{0}(L)$}};
\draw [color=gray, line width=0.3pt] (axis cs: 0.7355, 0)--(axis cs:2.85,3);
\draw [color=gray, line width=0.3pt] (axis cs: 3.49, 0)--(axis cs:3.49,2.2);
\node at (axis cs: 3.5, 3) {\scriptsize{$r_{-}(H,L)$}};
\node at (axis cs: 3.5, 2.5) {\scriptsize{$r_{+}(H,L)$}};
\fill (axis cs: 0.7355, 0) circle (0.8pt);
\fill (axis cs: 3.49, 0) circle (0.8pt);
\end{axis}
\end{tikzpicture}
\captionof{figure}{Qualitative graph of $W(\cdot;L)$ with angular momentum $L\in J$ (on the left) and representation of the bounded orbit in the $(r,\dot r)$-plane with $(H,L)\in\Lambda$ (on the right).}
\label{fig-01}
\end{figure}
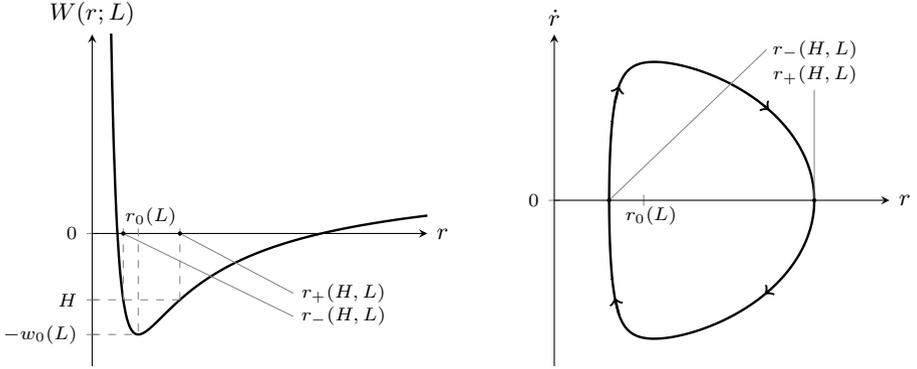

We then consider the function $\Theta \colon \Lambda \to \mathbb{R}$ defined by
\begin{equation}\label{def-Theta-HL}
\Theta(H,L)=\int_{0}^{T(H,L)} \dot{\vartheta}(t)\,\mathrm{d}t.
\end{equation}
For all $(H,L)\in\Lambda$, $\Theta(H,L)$ is the variation of the angular component $\vartheta$ of a solution $x$ of \eqref{eq-centr} in the radial period $T(H,L)$ and is called \emph{apsidal angle}.
We remark that the solution $x$ is periodic if and only if $\Theta(H,L)$ is commensurable with $2\pi$, that is there exist two coprime positive integers $n$ and $k$ such that
\begin{equation}\label{comm}
\Theta(H,L) = 2\pi \frac{k}{n} .
\end{equation}
More precisely, in such a case the solution $x$ has minimal period $nT(H,L)$ with winding number around the origin equals to $k$. 
Throughout the paper, we will call such solutions \emph{periodic solutions of type $(n,k)$}.

For further convenience, exploiting the relation $\dot\vartheta=L/r^{2}$, we notice that we can provide for $\Theta(H,L)$ the more explicit formula
\begin{equation*}
\Theta(H,L) 
= \int_{0}^{T(H,L)} \dfrac{L}{r^{2}(t)} \,\mathrm{d}t 
= \sqrt{2} L \int_{r_{-}(H,L)}^{r_{+}(H,L)} \dfrac{\mathrm{d}r}{r^{2} \sqrt{H-W(r;L)}}.
\end{equation*}
Moreover, via the Clairaut's change of variable $u=1/r$, we have
\begin{equation}\label{def-P}
\Theta(H,L) = L P(H,L),
\end{equation}
where
\begin{equation}\label{def-P-HL}
P(H,L) = \sqrt{2} \int_{\frac{1}{r_{+}(H,L)}}^{\frac{1}{r_{-}(H,L)}} \dfrac{\mathrm{d}u}{\sqrt{H-W(\frac{1}{u};L)}}.
\end{equation}
In such a way, the computation of $\Theta(H,L)$ is reduced to the computation of the period $P(H,L)$ of the abstract oscillator
\begin{equation*}
\dfrac{1}{2} \dot u^{2} + W\biggl{(}\dfrac{1}{u};L\biggr{)} = H
\end{equation*}
(compare with \cite{Ar-89, OrRo-19, Ro-18} for a similar trick).

\begin{remark}\label{rem-2.1}
We notice that when \eqref{comm} does not hold the solution $x$ is a \textit{quasi-periodic} function with two rationally linearly independent frequencies. Hence, the image of $x$ is a dense subset of the planar annulus $\{x\in\mathbb{R}^{2} \colon r_{-}(H,L)\leq |x| \leq r_{+}(H,L)\}$. Unless the function $\Theta$ is constant, the central force problem \eqref{eq-centr} thus possesses both periodic and quasi-periodic solutions. 

According to the celebrated Bertrand's theorem (cf.~\cite{Be-1873, OrRo-19}), there are only two central force problems for which all bounded solutions are periodic: the harmonic oscillator and the Kepler problem, corresponding to $V(|x|)=-\frac{\kappa}{2} |x|^{2}$ and $V(|x|)=\kappa /|x|$ with $\kappa>0$, respectively. In such cases the function $\Theta$ is in fact constant, precisely $\Theta\equiv\pi$ for the harmonic oscillator and $\Theta\equiv2\pi$ for the Kepler problem. See \cite[Section~2.8]{Ar-89} for more details.
\hfill$\lhd$
\end{remark}

\subsection{Invariant tori and action-angle coordinates}\label{section-2.2}

System \eqref{eq-centr} can be interpreted as the first order Hamiltonian system
\begin{equation}\label{hamilt-syst}
\dot x = \nabla_{p}\mathcal{H}(x,p), \qquad \dot p = -\nabla_{x}\mathcal{H}(x,p),
\end{equation}
where
\begin{equation*}
\mathcal{H}(x,p) = \dfrac{1}{2} |p|^{2} - V(|x|), \quad (x,p)\in(\mathbb{R}^{2}\setminus\{0\})\times\mathbb{R}^{2}.
\end{equation*}
Notice that, since $\dot x = p$ along a solution, $\mathcal{H}(x,p)=H$ (cf.~\eqref{energy}), according to the well-known fact that the Hamiltonian $\mathcal{H}$ is a first integral of system \eqref{hamilt-syst}. We also define
\begin{equation}\label{def-angh}
\mathcal{L}(x,p) = \langle x, i p \rangle
\end{equation}
and observe that $\mathcal{L}(x,p)=L$ (cf.~\eqref{ang-mom}), thus providing a second first integral of system \eqref{hamilt-syst}.

For a pair $(H,L)\in\Lambda$ as in Section~\ref{section-2.1} (giving bounded solutions), we consider the level set
\begin{equation*}
\mathcal{T}_{(H,L)} = \bigl{\{}(x,p)\in(\mathbb{R}^{2}\setminus\{0\})\times\mathbb{R}^{2} 
\colon \mathcal{H}(x,p)=H, \; \mathcal{L}(x,p)=L \bigr{\}}.
\end{equation*}
In principle, $\mathcal{T}_{(H,L)}$ can be disconnected, so that we focus on the compact connected component
\begin{equation}\label{def-t0}
\mathcal{T}_{(H,L)}^{0} = \bigl{\{}(x,p)\in\mathcal{T}_{(H,L)} \colon r_{-}(H,L) \leq |x| \leq r_{+}(H,L) \bigr{\}}.
\end{equation}
It is easy to check that the assumptions of Liouville--Arnol'd theorem hold true on $\mathcal{T}_{(H,L)}^{0}$, and thus $\mathcal{T}_{(H,L)}^{0}$ is diffeomorphic to a two-dimensional torus $\mathbb{T}^{2}$ (cf.~\cite{Ar-89, BeFa-notes}). This is in agreement with the fact that, given a solution $x$ of \eqref{eq-centr} with energy $H$ and angular momentum $L$, the function 
\begin{equation*}
x_{\lambda,\phi}(t) = x(t+\lambda) e^{i\phi}, \quad \lambda\in\mathbb{R}, \; \phi\in\mathbb{R},
\end{equation*}
is still a solution of \eqref{eq-centr} with same energy $H$ and angular momentum $L$.

Following \cite{Be-notes}, we now describe the construction of the action-angle variables $(I_{1},I_{2},\varphi_{1},\varphi_{2})$ on $\mathcal{T}_{(H,L)}^{0}$. In what follows it is convenient to visualize the torus $\mathcal{T}_{(H,L)}^{0}$ as in Figure~\ref{fig-02}. In particular, when it is useful, we can describe the points $(x,p)\in \mathcal{T}_{(H,L)}^{0}$ via the coordinates $(r, \dot r, \vartheta, L)$.

\begin{figure}[htb]
\centering
\begin{tikzpicture}
\begin{axis}[
  tick label style={font=\scriptsize},
  axis y line=left, 
  axis x line=middle,
  xtick={0.840896},
  ytick={0},
  xticklabels={},
  yticklabels={$0$},
  xlabel={\small $r$},
  ylabel={\small $\dot r$},
every axis x label/.style={
    at={(ticklabel* cs:1.0)},
    anchor=west,
},
every axis y label/.style={
    at={(ticklabel* cs:1.0)},
    anchor=south,
},
  width=6cm,
  height=6cm,
  xmin=0,
  xmax=4.9,
  ymin=-4.2,
  ymax=4.3]
\addplot [color=gray,line width=0.8pt,smooth] coordinates {(0.541, 0) (0.5412, 0.00451457) (0.542, 0.064726) (0.5428, 0.0912963) (0.5436, 0.111612) (0.5444, 0.12867) (0.5452, 0.143638) (0.546, 0.157113) (0.5468, 0.169453) (0.5476, 0.180891) (0.5484, 0.191588) (0.5492, 0.201663) (0.55, 0.211205) (0.5508, 0.220283) (0.5516, 0.228954) (0.5524, 0.237262) (0.5532, 0.245243) (0.554, 0.25293) (0.5548, 0.260349) (0.5556, 0.267522) (0.5564, 0.274468) (0.5572, 0.281205) (0.558, 0.287747) (0.5588, 0.294108) (0.5596, 0.3003) (0.5604, 0.306332) (0.5612, 0.312215) (0.562, 0.317956) (0.5628, 0.323564) (0.5636, 0.329045) (0.5644, 0.334406) (0.5652, 0.339653) (0.566, 0.34479) (0.5668, 0.349824) (0.5676, 0.354758) (0.5684, 0.359597) (0.5692, 0.364345) (0.57, 0.369005) (0.5768, 0.40551) (0.5924, 0.473612) (0.608, 0.527026) (0.6236, 0.570412) (0.6392, 0.606352) (0.6548, 0.636468) (0.6704, 0.66186) (0.686, 0.683317) (0.7016, 0.701426) (0.7172, 0.71664) (0.7328, 0.729314) (0.7484, 0.739732) (0.764, 0.748127) (0.7796, 0.754687) (0.7952, 0.75957) (0.8108, 0.762907) (0.8264, 0.764808) (0.842, 0.765364) (0.8576, 0.764652) (0.8732, 0.762735) (0.8888, 0.759668) (0.9044, 0.755492) (0.92, 0.750242) (0.9356, 0.743943) (0.9512, 0.736613) (0.9668, 0.728264) (0.9824, 0.718898) (0.998, 0.708513) (1.0136, 0.697096) (1.0292, 0.684629) (1.0448, 0.671083) (1.0604, 0.656422) (1.076, 0.640595) (1.0916, 0.62354) (1.1072, 0.605179) (1.1228, 0.585413) (1.1384, 0.564118) (1.154, 0.541136) (1.1696, 0.516265) (1.1852, 0.489237) (1.2008, 0.459695) (1.2164, 0.42714) (1.232, 0.390841) (1.2476, 0.34966) (1.2632, 0.301627) (1.2788, 0.242738) (1.29, 0.188248) (1.2904, 0.185988) (1.2908, 0.183698) (1.2912, 0.181378) (1.2916, 0.179027) (1.292, 0.176643) (1.2924, 0.174225) (1.2928, 0.171772) (1.2932, 0.169281) (1.2936, 0.166752) (1.294, 0.164183) (1.2944, 0.161571) (1.2948, 0.158915) (1.2952, 0.156211) (1.2956, 0.153459) (1.296, 0.150655) (1.2964, 0.147796) (1.2968, 0.144879) (1.2972, 0.1419) (1.2976, 0.138855) (1.298, 0.135741) (1.2984, 0.132551) (1.2988, 0.129281) (1.2992, 0.125924) (1.2996, 0.122473) (1.3, 0.11892) (1.3004, 0.115256) (1.3008, 0.111468) (1.3012, 0.107545) (1.3016, 0.103472) (1.302, 0.0992283) (1.3024, 0.0947926) (1.3028, 0.0901362) (1.3032, 0.0852229) (1.3036, 0.0800054) (1.304, 0.0744198) (1.3044, 0.0683759) (1.3048, 0.0617392) (1.3052, 0.0542928) (1.3056, 0.0456422) (1.306, 0.034903) (1.3064, 0.0187815) (1.3066, 0)};
\addplot [color=gray,line width=0.8pt,smooth] coordinates {(0.541, 0) (0.5412, -0.00451457) (0.542, -0.064726) (0.5428, -0.0912963) (0.5436, -0.111612) (0.5444, -0.12867) (0.5452, -0.143638) (0.546, -0.157113) (0.5468, -0.169453) (0.5476, -0.180891) (0.5484, -0.191588) (0.5492, -0.201663) (0.55, -0.211205) (0.5508, -0.220283) (0.5516, -0.228954) (0.5524, -0.237262) (0.5532, -0.245243) (0.554, -0.25293) (0.5548, -0.260349) (0.5556, -0.267522) (0.5564, -0.274468) (0.5572, -0.281205) (0.558, -0.287747) (0.5588, -0.294108) (0.5596, -0.3003) (0.5604, -0.306332) (0.5612, -0.312215) (0.562, -0.317956) (0.5628, -0.323564) (0.5636, -0.329045) (0.5644, -0.334406) (0.5652, -0.339653) (0.566, -0.34479) (0.5668, -0.349824) (0.5676, -0.354758) (0.5684, -0.359597) (0.5692, -0.364345) (0.57, -0.369005) (0.5768, -0.40551) (0.5924, -0.473612) (0.608, -0.527026) (0.6236, -0.570412) (0.6392, -0.606352) (0.6548, -0.636468) (0.6704, -0.66186) (0.686, -0.683317) (0.7016, -0.701426) (0.7172, -0.71664) (0.7328, -0.729314) (0.7484, -0.739732) (0.764, -0.748127) (0.7796, -0.754687) (0.7952, -0.75957) (0.8108, -0.762907) (0.8264, -0.764808) (0.842, -0.765364) (0.8576, -0.764652) (0.8732, -0.762735) (0.8888, -0.759668) (0.9044, -0.755492) (0.92, -0.750242) (0.9356, -0.743943) (0.9512, -0.736613) (0.9668, -0.728264) (0.9824, -0.718898) (0.998, -0.708513) (1.0136, -0.697096) (1.0292, -0.684629) (1.0448, -0.671083) (1.0604, -0.656422) (1.076, -0.640595) (1.0916, -0.62354) (1.1072, -0.605179) (1.1228, -0.585413) (1.1384, -0.564118) (1.154, -0.541136) (1.1696, -0.516265) (1.1852, -0.489237) (1.2008, -0.459695) (1.2164, -0.42714) (1.232, -0.390841) (1.2476, -0.34966) (1.2632, -0.301627) (1.2788, -0.242738) (1.29, -0.188248) (1.2904, -0.185988) (1.2908, -0.183698) (1.2912, -0.181378) (1.2916, -0.179027) (1.292, -0.176643) (1.2924, -0.174225) (1.2928, -0.171772) (1.2932, -0.169281) (1.2936, -0.166752) (1.294, -0.164183) (1.2944, -0.161571) (1.2948, -0.158915) (1.2952, -0.156211) (1.2956, -0.153459) (1.296, -0.150655) (1.2964, -0.147796) (1.2968, -0.144879) (1.2972, -0.1419) (1.2976, -0.138855) (1.298, -0.135741) (1.2984, -0.132551) (1.2988, -0.129281) (1.2992, -0.125924) (1.2996, -0.122473) (1.3, -0.11892) (1.3004, -0.115256) (1.3008, -0.111468) (1.3012, -0.107545) (1.3016, -0.103472) (1.302, -0.0992283) (1.3024, -0.0947926) (1.3028, -0.0901362) (1.3032, -0.0852229) (1.3036, -0.0800054) (1.304, -0.0744198) (1.3044, -0.0683759) (1.3048, -0.0617392) (1.3052, -0.0542928) (1.3056, -0.0456422) (1.306, -0.034903) (1.3064, -0.0187815) (1.3066, 0)};
\addplot [color=gray,line width=0.8pt,smooth] coordinates {(0.319, 0) (0.32, 0.121604) (0.321, 0.211001) (0.322, 0.271976) (0.323, 0.321154) (0.324, 0.36336) (0.325, 0.400808) (0.326, 0.434733) (0.327, 0.465909) (0.328, 0.494858) (0.329, 0.521955) (0.33, 0.547478) (0.331, 0.571639) (0.332, 0.594608) (0.333, 0.61652) (0.334, 0.637487) (0.335, 0.657602) (0.336, 0.676941) (0.337, 0.695573) (0.338, 0.713553) (0.339, 0.730934) (0.34, 0.747757) (0.341, 0.764062) (0.342, 0.779882) (0.343, 0.795249) (0.344, 0.81019) (0.345, 0.824728) (0.346, 0.838888) (0.347, 0.852688) (0.348, 0.866148) (0.349, 0.879285) (0.35, 0.892114) (0.351, 0.904649) (0.352, 0.916904) (0.353, 0.928892) (0.354, 0.940622) (0.355, 0.952107) (0.356, 0.963356) (0.357, 0.974378) (0.358, 0.985182) (0.359, 0.995777) (0.36, 1.00617) (0.38, 1.18025) (0.42, 1.41037) (0.46, 1.55739) (0.5, 1.65831) (0.54, 1.73024) (0.58, 1.78249) (0.62, 1.82068) (0.66, 1.84839) (0.7, 1.86804) (0.74, 1.88131) (0.78, 1.88938) (0.82, 1.89314) (0.86, 1.89324) (0.9, 1.89016) (0.94, 1.88429) (0.98, 1.8759) (1.02, 1.86521) (1.06, 1.8524) (1.1, 1.8376) (1.14, 1.8209) (1.18, 1.80236) (1.22, 1.78204) (1.26, 1.75996) (1.3, 1.73613) (1.34, 1.71054) (1.38, 1.68317) (1.42, 1.65398) (1.46, 1.62291) (1.5, 1.5899) (1.54, 1.55485) (1.58, 1.51767) (1.62, 1.4782) (1.66, 1.4363) (1.7, 1.39176) (1.74, 1.34434) (1.78, 1.29375) (1.82, 1.23962) (1.86, 1.18147) (1.9, 1.1187) (1.94, 1.0505) (1.98, 0.975737) (2.02, 0.892784) (2.06, 0.799109) (2.1, 0.690378) (2.14, 0.557871) (2.15, 0.518973) (2.153, 0.506681) (2.156, 0.494063) (2.159, 0.481095) (2.162, 0.467747) (2.165, 0.453985) (2.168, 0.43977) (2.171, 0.425059) (2.174, 0.409796) (2.177, 0.393917) (2.18, 0.377346) (2.183, 0.359986) (2.186, 0.341717) (2.189, 0.322385) (2.192, 0.301786) (2.195, 0.279639) (2.198, 0.255543) (2.201, 0.228882) (2.204, 0.198627) (2.207, 0.162787) (2.21, 0.116306) (2.213, 0.0231379) (2.215, 0)};
\addplot [color=gray,line width=0.8pt,smooth] coordinates {(0.319, 0) (0.32, -0.121604) (0.321, -0.211001) (0.322, -0.271976) (0.323, -0.321154) (0.324, -0.36336) (0.325, -0.400808) (0.326, -0.434733) (0.327, -0.465909) (0.328, -0.494858) (0.329, -0.521955) (0.33, -0.547478) (0.331, -0.571639) (0.332, -0.594608) (0.333, -0.61652) (0.334, -0.637487) (0.335, -0.657602) (0.336, -0.676941) (0.337, -0.695573) (0.338, -0.713553) (0.339, -0.730934) (0.34, -0.747757) (0.341, -0.764062) (0.342, -0.779882) (0.343, -0.795249) (0.344, -0.81019) (0.345, -0.824728) (0.346, -0.838888) (0.347, -0.852688) (0.348, -0.866148) (0.349, -0.879285) (0.35, -0.892114) (0.351, -0.904649) (0.352, -0.916904) (0.353, -0.928892) (0.354, -0.940622) (0.355, -0.952107) (0.356, -0.963356) (0.357, -0.974378) (0.358, -0.985182) (0.359, -0.995777) (0.36, -1.00617) (0.38, -1.18025) (0.42, -1.41037) (0.46, -1.55739) (0.5, -1.65831) (0.54, -1.73024) (0.58, -1.78249) (0.62, -1.82068) (0.66, -1.84839) (0.7, -1.86804) (0.74, -1.88131) (0.78, -1.88938) (0.82, -1.89314) (0.86, -1.89324) (0.9, -1.89016) (0.94, -1.88429) (0.98, -1.8759) (1.02, -1.86521) (1.06, -1.8524) (1.1, -1.8376) (1.14, -1.8209) (1.18, -1.80236) (1.22, -1.78204) (1.26, -1.75996) (1.3, -1.73613) (1.34, -1.71054) (1.38, -1.68317) (1.42, -1.65398) (1.46, -1.62291) (1.5, -1.5899) (1.54, -1.55485) (1.58, -1.51767) (1.62, -1.4782) (1.66, -1.4363) (1.7, -1.39176) (1.74, -1.34434) (1.78, -1.29375) (1.82, -1.23962) (1.86, -1.18147) (1.9, -1.1187) (1.94, -1.0505) (1.98, -0.975737) (2.02, -0.892784) (2.06, -0.799109) (2.1, -0.690378) (2.14, -0.557871) (2.15, -0.518973) (2.153, -0.506681) (2.156, -0.494063) (2.159, -0.481095) (2.162, -0.467747) (2.165, -0.453985) (2.168, -0.43977) (2.171, -0.425059) (2.174, -0.409796) (2.177, -0.393917) (2.18, -0.377346) (2.183, -0.359986) (2.186, -0.341717) (2.189, -0.322385) (2.192, -0.301786) (2.195, -0.279639) (2.198, -0.255543) (2.201, -0.228882) (2.204, -0.198627) (2.207, -0.162787) (2.21, -0.116306) (2.213, -0.0231379) (2.215, 0)};
\draw[->] [color=black,line width=0.9pt] (axis cs: 0.312, 1.94068)--(axis cs: 0.368, 2.27431);
\draw[->] [color=black,line width=0.9pt] (axis cs: 2.328, 1.86766)--(axis cs: 2.384, 1.79682);
\draw[->] [color=black,line width=0.9pt] (axis cs: 0.368, -2.27431)--(axis cs: 0.312, -1.94068);
\draw[->] [color=black,line width=0.9pt] (axis cs: 2.384, -1.79682)--(axis cs: 2.328, -1.86766);
\addplot [color=black, fill=gray, fill opacity=0.2, line width=0.9pt,smooth] coordinates {(0.2355, 0) (0.2368, 0.16485) (0.2384, 0.381708) (0.24, 0.511707) (0.2416, 0.612917) (0.2432, 0.698019) (0.2448, 0.772395) (0.2464, 0.838936) (0.248, 0.899413) (0.2496, 0.955007) (0.2512, 1.00655) (0.2528, 1.05467) (0.2544, 1.09983) (0.256, 1.1424) (0.2576, 1.18268) (0.2592, 1.22092) (0.2608, 1.25732) (0.2624, 1.29204) (0.264, 1.32525) (0.2656, 1.35706) (0.2672, 1.38759) (0.2688, 1.41692) (0.2704, 1.44515) (0.272, 1.47234) (0.2736, 1.49858) (0.2752, 1.52391) (0.2768, 1.54839) (0.2784, 1.57208) (0.28, 1.59501) (0.312, 1.94068) (0.368, 2.27431) (0.424, 2.45743) (0.48, 2.56894) (0.536, 2.64052) (0.592, 2.68754) (0.648, 2.71834) (0.704, 2.7378) (0.76, 2.74895) (0.816, 2.75377) (0.872, 2.75355) (0.928, 2.74922) (0.984, 2.74141) (1.04, 2.73059) (1.096, 2.71708) (1.152, 2.70114) (1.208, 2.68293) (1.264, 2.66258) (1.32, 2.6402) (1.376, 2.61583) (1.432, 2.58951) (1.488, 2.56126) (1.544, 2.53107) (1.6, 2.49894) (1.656, 2.46482) (1.712, 2.42868) (1.768, 2.39044) (1.824, 2.35005) (1.88, 2.30741) (1.936, 2.26241) (1.992, 2.21493) (2.048, 2.16483) (2.104, 2.11193) (2.16, 2.05602) (2.216, 1.99688) (2.272, 1.93421) (2.328, 1.86766) (2.384, 1.79682) (2.44, 1.72117) (2.496, 1.64004) (2.552, 1.55259) (2.608, 1.45768) (2.664, 1.35375) (2.72, 1.23855) (2.776, 1.10858) (2.8, 1.04701) (2.804, 1.03634) (2.808, 1.02554) (2.812, 1.01461) (2.816, 1.00354) (2.82, 0.992334) (2.824, 0.980983) (2.828, 0.969483) (2.832, 0.957828) (2.836, 0.946011) (2.84, 0.934028) (2.844, 0.921871) (2.848, 0.909534) (2.852, 0.897009) (2.856, 0.884288) (2.86, 0.871362) (2.864, 0.858223) (2.868, 0.84486) (2.872, 0.831263) (2.876, 0.817419) (2.88, 0.803317) (2.884, 0.788942) (2.888, 0.774279) (2.892, 0.759311) (2.896, 0.744021) (2.9, 0.728387) (2.904, 0.712387) (2.908, 0.695995) (2.912, 0.679185) (2.916, 0.661923) (2.92, 0.644173) (2.924, 0.625894) (2.928, 0.607038) (2.932, 0.587549) (2.936, 0.567362) (2.94, 0.546401) (2.944, 0.524571) (2.948, 0.50176) (2.952, 0.477827) (2.956, 0.452595) (2.96, 0.425832) (2.964, 0.397229) (2.968, 0.366355) (2.972, 0.332579) (2.976, 0.294905) (2.98, 0.251587) (2.984, 0.198975) (2.988, 0.12591) (2.99, 0)};
\addplot [color=black, fill=gray, fill opacity=0.2, line width=0.9pt,smooth] coordinates {(0.2355, 0) (0.2368, -0.16485) (0.2384, -0.381708) (0.24, -0.511707) (0.2416, -0.612917) (0.2432, -0.698019) (0.2448, -0.772395) (0.2464, -0.838936) (0.248, -0.899413) (0.2496, -0.955007) (0.2512, -1.00655) (0.2528, -1.05467) (0.2544, -1.09983) (0.256, -1.1424) (0.2576, -1.18268) (0.2592, -1.22092) (0.2608, -1.25732) (0.2624, -1.29204) (0.264, -1.32525) (0.2656, -1.35706) (0.2672, -1.38759) (0.2688, -1.41692) (0.2704, -1.44515) (0.272, -1.47234) (0.2736, -1.49858) (0.2752, -1.52391) (0.2768, -1.54839) (0.2784, -1.57208) (0.28, -1.59501) (0.312, -1.94068) (0.368, -2.27431) (0.424, -2.45743) (0.48, -2.56894) (0.536, -2.64052) (0.592, -2.68754) (0.648, -2.71834) (0.704, -2.7378) (0.76, -2.74895) (0.816, -2.75377) (0.872, -2.75355) (0.928, -2.74922) (0.984, -2.74141) (1.04, -2.73059) (1.096, -2.71708) (1.152, -2.70114) (1.208, -2.68293) (1.264, -2.66258) (1.32, -2.6402) (1.376, -2.61583) (1.432, -2.58951) (1.488, -2.56126) (1.544, -2.53107) (1.6, -2.49894) (1.656, -2.46482) (1.712, -2.42868) (1.768, -2.39044) (1.824, -2.35005) (1.88, -2.30741) (1.936, -2.26241) (1.992, -2.21493) (2.048, -2.16483) (2.104, -2.11193) (2.16, -2.05602) (2.216, -1.99688) (2.272, -1.93421) (2.328, -1.86766) (2.384, -1.79682) (2.44, -1.72117) (2.496, -1.64004) (2.552, -1.55259) (2.608, -1.45768) (2.664, -1.35375) (2.72, -1.23855) (2.776, -1.10858) (2.8, -1.04701) (2.804, -1.03634) (2.808, -1.02554) (2.812, -1.01461) (2.816, -1.00354) (2.82, -0.992334) (2.824, -0.980983) (2.828, -0.969483) (2.832, -0.957828) (2.836, -0.946011) (2.84, -0.934028) (2.844, -0.921871) (2.848, -0.909534) (2.852, -0.897009) (2.856, -0.884288) (2.86, -0.871362) (2.864, -0.858223) (2.868, -0.84486) (2.872, -0.831263) (2.876, -0.817419) (2.88, -0.803317) (2.884, -0.788942) (2.888, -0.774279) (2.892, -0.759311) (2.896, -0.744021) (2.9, -0.728387) (2.904, -0.712387) (2.908, -0.695995) (2.912, -0.679185) (2.916, -0.661923) (2.92, -0.644173) (2.924, -0.625894) (2.928, -0.607038) (2.932, -0.587549) (2.936, -0.567362) (2.94, -0.546401) (2.944, -0.524571) (2.948, -0.50176) (2.952, -0.477827) (2.956, -0.452595) (2.96, -0.425832) (2.964, -0.397229) (2.968, -0.366355) (2.972, -0.332579) (2.976, -0.294905) (2.98, -0.251587) (2.984, -0.198975) (2.988, -0.12591) (2.99, 0)};
\addplot [color=gray,line width=0.8pt,smooth] coordinates {(0.1895, 0) (0.1908, 0.478626) (0.1934, 0.771288) (0.196, 0.972715) (0.1986, 1.13301) (0.2012, 1.26814) (0.2038, 1.38574) (0.2064, 1.49016) (0.209, 1.5842) (0.2116, 1.66978) (0.2142, 1.74829) (0.2168, 1.82077) (0.2194, 1.88804) (0.222, 1.95075) (0.2246, 2.00943) (0.2272, 2.06451) (0.2298, 2.11635) (0.2324, 2.16527) (0.235, 2.21154) (0.2376, 2.25538) (0.2402, 2.297) (0.2428, 2.33657) (0.2454, 2.37425) (0.248, 2.41017) (0.2506, 2.44448) (0.2532, 2.47726) (0.2558, 2.50863) (0.2584, 2.53868) (0.261, 2.56749) (0.2636, 2.59513) (0.2662, 2.62168) (0.2688, 2.6472) (0.2714, 2.67174) (0.274, 2.69537) (0.2766, 2.71812) (0.2792, 2.74005) (0.2818, 2.7612) (0.2844, 2.78161) (0.287, 2.80132) (0.2896, 2.82035) (0.2922, 2.83875) (0.2948, 2.85654) (0.2974, 2.87375) (0.3, 2.89041) (0.3132, 2.96728) (0.3848, 3.23654) (0.4564, 3.3751) (0.528, 3.45365) (0.5996, 3.49996) (0.6712, 3.52699) (0.7428, 3.54147) (0.8144, 3.54724) (0.886, 3.54656) (0.9576, 3.54087) (1.0292, 3.53111) (1.1008, 3.5179) (1.1724, 3.50167) (1.244, 3.48272) (1.3156, 3.46126) (1.3872, 3.43742) (1.4588, 3.4113) (1.5304, 3.38296) (1.602, 3.35243) (1.6736, 3.31972) (1.7452, 3.28483) (1.8168, 3.24773) (1.8884, 3.20839) (1.96, 3.16674) (2.0316, 3.12273) (2.1032, 3.07628) (2.1748, 3.0273) (2.2464, 2.97567) (2.318, 2.92127) (2.3896, 2.86396) (2.4612, 2.80356) (2.5328, 2.73989) (2.6044, 2.67271) (2.676, 2.60177) (2.7476, 2.52675) (2.8192, 2.44728) (2.8908, 2.36293) (2.9624, 2.27315) (3.034, 2.17728) (3.1056, 2.07447) (3.1772, 1.96364) (3.2488, 1.84335) (3.3204, 1.71161) (3.392, 1.56553) (3.4636, 1.40064) (3.5352, 1.20928) (3.55, 1.16526) (3.554, 1.15304) (3.558, 1.14068) (3.562, 1.12816) (3.566, 1.11549) (3.57, 1.10266) (3.574, 1.08967) (3.578, 1.0765) (3.582, 1.06316) (3.586, 1.04963) (3.59, 1.03591) (3.594, 1.02199) (3.598, 1.00786) (3.602, 0.993508) (3.606, 0.978934) (3.61, 0.964123) (3.614, 0.949064) (3.618, 0.933744) (3.622, 0.918152) (3.626, 0.902272) (3.63, 0.88609) (3.634, 0.869587) (3.638, 0.852747) (3.642, 0.835548) (3.646, 0.817968) (3.65, 0.799981) (3.654, 0.78156) (3.658, 0.762673) (3.662, 0.743284) (3.666, 0.723354) (3.67, 0.702835) (3.674, 0.681676) (3.678, 0.659814) (3.682, 0.637177) (3.686, 0.61368) (3.69, 0.589219) (3.694, 0.563669) (3.698, 0.536874) (3.702, 0.508638) (3.706, 0.478706) (3.71, 0.446737) (3.714, 0.412257) (3.718, 0.374574) (3.722, 0.332601) (3.726, 0.284445) (3.73, 0.22619) (3.734, 0.14623) (3.737, 0)};
\addplot [color=gray,line width=0.8pt,smooth] coordinates {(0.1895, 0) (0.1908, -0.478626) (0.1934, -0.771288) (0.196, -0.972715) (0.1986, -1.13301) (0.2012, -1.26814) (0.2038, -1.38574) (0.2064, -1.49016) (0.209, -1.5842) (0.2116, -1.66978) (0.2142, -1.74829) (0.2168, -1.82077) (0.2194, -1.88804) (0.222, -1.95075) (0.2246, -2.00943) (0.2272, -2.06451) (0.2298, -2.11635) (0.2324, -2.16527) (0.235, -2.21154) (0.2376, -2.25538) (0.2402, -2.297) (0.2428, -2.33657) (0.2454, -2.37425) (0.248, -2.41017) (0.2506, -2.44448) (0.2532, -2.47726) (0.2558, -2.50863) (0.2584, -2.53868) (0.261, -2.56749) (0.2636, -2.59513) (0.2662, -2.62168) (0.2688, -2.6472) (0.2714, -2.67174) (0.274, -2.69537) (0.2766, -2.71812) (0.2792, -2.74005) (0.2818, -2.7612) (0.2844, -2.78161) (0.287, -2.80132) (0.2896, -2.82035) (0.2922, -2.83875) (0.2948, -2.85654) (0.2974, -2.87375) (0.3, -2.89041) (0.3132, -2.96728) (0.3848, -3.23654) (0.4564, -3.3751) (0.528, -3.45365) (0.5996, -3.49996) (0.6712, -3.52699) (0.7428, -3.54147) (0.8144, -3.54724) (0.886, -3.54656) (0.9576, -3.54087) (1.0292, -3.53111) (1.1008, -3.5179) (1.1724, -3.50167) (1.244, -3.48272) (1.3156, -3.46126) (1.3872, -3.43742) (1.4588, -3.4113) (1.5304, -3.38296) (1.602, -3.35243) (1.6736, -3.31972) (1.7452, -3.28483) (1.8168, -3.24773) (1.8884, -3.20839) (1.96, -3.16674) (2.0316, -3.12273) (2.1032, -3.07628) (2.1748, -3.0273) (2.2464, -2.97567) (2.318, -2.92127) (2.3896, -2.86396) (2.4612, -2.80356) (2.5328, -2.73989) (2.6044, -2.67271) (2.676, -2.60177) (2.7476, -2.52675) (2.8192, -2.44728) (2.8908, -2.36293) (2.9624, -2.27315) (3.034, -2.17728) (3.1056, -2.07447) (3.1772, -1.96364) (3.2488, -1.84335) (3.3204, -1.71161) (3.392, -1.56553) (3.4636, -1.40064) (3.5352, -1.20928) (3.55, -1.16526) (3.554, -1.15304) (3.558, -1.14068) (3.562, -1.12816) (3.566, -1.11549) (3.57, -1.10266) (3.574, -1.08967) (3.578, -1.0765) (3.582, -1.06316) (3.586, -1.04963) (3.59, -1.03591) (3.594, -1.02199) (3.598, -1.00786) (3.602, -0.993508) (3.606, -0.978934) (3.61, -0.964123) (3.614, -0.949064) (3.618, -0.933744) (3.622, -0.918152) (3.626, -0.902272) (3.63, -0.88609) (3.634, -0.869587) (3.638, -0.852747) (3.642, -0.835548) (3.646, -0.817968) (3.65, -0.799981) (3.654, -0.78156) (3.658, -0.762673) (3.662, -0.743284) (3.666, -0.723354) (3.67, -0.702835) (3.674, -0.681676) (3.678, -0.659814) (3.682, -0.637177) (3.686, -0.61368) (3.69, -0.589219) (3.694, -0.563669) (3.698, -0.536874) (3.702, -0.508638) (3.706, -0.478706) (3.71, -0.446737) (3.714, -0.412257) (3.718, -0.374574) (3.722, -0.332601) (3.726, -0.284445) (3.73, -0.22619) (3.734, -0.14623) (3.737, 0)};
\addplot [color=gray,line width=0.8pt,smooth] coordinates {(0.166, 0)  (0.1668, 0.0307122) (0.168, 0.506317) (0.1692, 0.711578) (0.1704, 0.866635) (0.1716, 0.99531) (0.1728, 1.10691) (0.174, 1.20623) (0.1752, 1.29616) (0.1764, 1.37858) (0.1776, 1.45481) (0.1788, 1.52581) (0.18, 1.59233) (0.1812, 1.65493) (0.1824, 1.71408) (0.1836, 1.77015) (0.1848, 1.82346) (0.186, 1.87426) (0.1872, 1.92279) (0.1884, 1.96922) (0.1896, 2.01374) (0.1908, 2.05647) (0.192, 2.09756) (0.1932, 2.13712) (0.1944, 2.17524) (0.1956, 2.21202) (0.1968, 2.24755) (0.198, 2.28188) (0.1992, 2.31511) (0.2004, 2.34727) (0.2016, 2.37844) (0.2028, 2.40866) (0.204, 2.43799) (0.2052, 2.46645) (0.2064, 2.49411) (0.2076, 2.52099) (0.2088, 2.54712) (0.21, 2.57255) (0.233, 2.95563) (0.316, 3.59067) (0.399, 3.83407) (0.482, 3.95165) (0.565, 4.01428) (0.648, 4.04838) (0.731, 4.0657) (0.814, 4.0722) (0.897, 4.07111) (0.98, 4.06436) (1.063, 4.05309) (1.146, 4.03806) (1.229, 4.01977) (1.312, 3.99852) (1.395, 3.97455) (1.478, 3.94799) (1.561, 3.91894) (1.644, 3.88745) (1.727, 3.85355) (1.81, 3.81723) (1.893, 3.77849) (1.976, 3.73729) (2.059, 3.69359) (2.142, 3.64731) (2.225, 3.59839) (2.308, 3.54673) (2.391, 3.49223) (2.474, 3.43477) (2.557, 3.37421) (2.64, 3.31039) (2.723, 3.24312) (2.806, 3.1722) (2.889, 3.09738) (2.972, 3.01838) (3.055, 2.93486) (3.138, 2.84643) (3.221, 2.75263) (3.304, 2.65288) (3.387, 2.5465) (3.47, 2.43261) (3.553, 2.3101) (3.636, 2.17754) (3.719, 2.03295) (3.802, 1.87355) (3.885, 1.69518) (3.968, 1.49105) (4., 1.40312) (4.006, 1.38593) (4.012, 1.3685) (4.018, 1.35082) (4.024, 1.33287) (4.03, 1.31465) (4.036, 1.29615) (4.042, 1.27735) (4.048, 1.25825) (4.054, 1.23881) (4.06, 1.21904) (4.066, 1.19892) (4.072, 1.17841) (4.078, 1.15752) (4.084, 1.13621) (4.09, 1.11446) (4.096, 1.09224) (4.102, 1.06952) (4.108, 1.04628) (4.114, 1.02248) (4.12, 0.99807) (4.126, 0.973013) (4.132, 0.947254) (4.138, 0.920736) (4.144, 0.893391) (4.15, 0.865141) (4.156, 0.835892) (4.162, 0.805538) (4.168, 0.773947) (4.174, 0.740962) (4.18, 0.706388) (4.186, 0.669977) (4.192, 0.631414) (4.198, 0.590275) (4.204, 0.545979) (4.21, 0.497684) (4.216, 0.444088) (4.222, 0.38297) (4.228, 0.309912) (4.234, 0.212962) (4.237, 0)};
\addplot [color=gray,line width=0.8pt,smooth] coordinates {(0.166, 0)  (0.1668, -0.0307122) (0.168, -0.506317) (0.1692, -0.711578) (0.1704, -0.866635) (0.1716, -0.99531) (0.1728, -1.10691) (0.174, -1.20623) (0.1752, -1.29616) (0.1764, -1.37858) (0.1776, -1.45481) (0.1788, -1.52581) (0.18, -1.59233) (0.1812, -1.65493) (0.1824, -1.71408) (0.1836, -1.77015) (0.1848, -1.82346) (0.186, -1.87426) (0.1872, -1.92279) (0.1884, -1.96922) (0.1896, -2.01374) (0.1908, -2.05647) (0.192, -2.09756) (0.1932, -2.13712) (0.1944, -2.17524) (0.1956, -2.21202) (0.1968, -2.24755) (0.198, -2.28188) (0.1992, -2.31511) (0.2004, -2.34727) (0.2016, -2.37844) (0.2028, -2.40866) (0.204, -2.43799) (0.2052, -2.46645) (0.2064, -2.49411) (0.2076, -2.52099) (0.2088, -2.54712) (0.21, -2.57255) (0.233, -2.95563) (0.316, -3.59067) (0.399, -3.83407) (0.482, -3.95165) (0.565, -4.01428) (0.648, -4.04838) (0.731, -4.0657) (0.814, -4.0722) (0.897, -4.07111) (0.98, -4.06436) (1.063, -4.05309) (1.146, -4.03806) (1.229, -4.01977) (1.312, -3.99852) (1.395, -3.97455) (1.478, -3.94799) (1.561, -3.91894) (1.644, -3.88745) (1.727, -3.85355) (1.81, -3.81723) (1.893, -3.77849) (1.976, -3.73729) (2.059, -3.69359) (2.142, -3.64731) (2.225, -3.59839) (2.308, -3.54673) (2.391, -3.49223) (2.474, -3.43477) (2.557, -3.37421) (2.64, -3.31039) (2.723, -3.24312) (2.806, -3.1722) (2.889, -3.09738) (2.972, -3.01838) (3.055, -2.93486) (3.138, -2.84643) (3.221, -2.75263) (3.304, -2.65288) (3.387, -2.5465) (3.47, -2.43261) (3.553, -2.3101) (3.636, -2.17754) (3.719, -2.03295) (3.802, -1.87355) (3.885, -1.69518) (3.968, -1.49105) (4., -1.40312) (4.006, -1.38593) (4.012, -1.3685) (4.018, -1.35082) (4.024, -1.33287) (4.03, -1.31465) (4.036, -1.29615) (4.042, -1.27735) (4.048, -1.25825) (4.054, -1.23881) (4.06, -1.21904) (4.066, -1.19892) (4.072, -1.17841) (4.078, -1.15752) (4.084, -1.13621) (4.09, -1.11446) (4.096, -1.09224) (4.102, -1.06952) (4.108, -1.04628) (4.114, -1.02248) (4.12, -0.99807) (4.126, -0.973013) (4.132, -0.947254) (4.138, -0.920736) (4.144, -0.893391) (4.15, -0.865141) (4.156, -0.835892) (4.162, -0.805538) (4.168, -0.773947) (4.174, -0.740962) (4.18, -0.706388) (4.186, -0.669977) (4.192, -0.631414) (4.198, -0.590275) (4.204, -0.545979) (4.21, -0.497684) (4.216, -0.444088) (4.222, -0.38297) (4.228, -0.309912) (4.234, -0.212962) (4.237, 0)};
\addplot [color=gray,line width=0.3pt] coordinates {(3.5,3.7) (2.5,1)};
\node at (axis cs: 3.5,4) {\scriptsize{$\mathcal{A}(\bar{H},\bar{L})$}};
\end{axis}
\end{tikzpicture}
\hspace{10pt}
\begin{tikzpicture}
\begin{axis}[
  tick label style={font=\scriptsize},
  axis y line=left, 
  axis x line=bottom,
  xtick={0,2.2,6},
  ytick={0,4},
  xticklabels={$0$,$\vartheta^{*}$,$2\pi$},
  yticklabels={,$\bar{L}$},
  xlabel={\small $\vartheta$},
  ylabel={\small $L$},
every axis x label/.style={
    at={(ticklabel* cs:1.0)},
    anchor=west,
},
every axis y label/.style={
    at={(ticklabel* cs:1.0)},
    anchor=south,
},
  width=6cm,
  height=4.5cm,
  xmin=0,
  xmax=7,
  ymin=0,
  ymax=6]
\draw[->] [color=black,line width=0.9pt] (axis cs: 1.45, 4)--(axis cs: 1.5, 4);
\draw[->] [color=black,line width=0.9pt] (axis cs: 4.45, 4)--(axis cs: 4.5, 4);
\addplot [color=gray,line width=0.8pt,smooth] coordinates {(0,1) (6,1)};
\addplot [color=gray,line width=0.8pt,smooth] coordinates {(0,2) (6,2)};
\addplot [color=gray,line width=0.8pt,smooth] coordinates {(0,3) (6,3)};
\addplot [color=black,line width=0.9pt,smooth] coordinates {(0,4) (6,4)};
\addplot [color=gray,line width=0.8pt,smooth] coordinates {(0,5) (6,5)};
\addplot [color=gray,line width=0.3pt,dashed] coordinates {(6,0) (6,5.5)};
\addplot [color=gray,line width=0.3pt,dashed] coordinates {(2.2,0) (2.2,5.5)};
\end{axis}
\end{tikzpicture}
\vspace{10pt}
\\
\begin{tikzpicture}
\begin{axis}[
  axis equal image,
  tick label style={font=\scriptsize},
  axis y line=middle, 
  axis x line=middle,
  xtick={0},
  ytick={0},
  xticklabels={},
  yticklabels={$0$},
  xlabel={},
  ylabel={},
every axis x label/.style={
    at={(ticklabel* cs:1.0)},
    anchor=west,
},
every axis y label/.style={
    at={(ticklabel* cs:1.0)},
    anchor=south,
},
  width=8cm,
  height=6cm,
  xmin=-8,
  xmax=8,
  ymin=-6,
  ymax=6]
\draw [color=gray, line width=0.3pt] (axis cs: 0,0)--(axis cs: -4.5,4.5);
\draw [color=gray, line width=0.3pt] (axis cs: 0,0)--(axis cs: -1,-1.73205);
\node at (axis cs: -5.7,-5) {\scriptsize{$r_{-}(\bar{H},\bar{L}) e^{i \psi}$}};
\addplot [color=gray,line width=0.3pt] coordinates {(-5.7,-4.5) (-0.59,-1.02191)};
\fill (axis cs: -0.59,-1.02191) circle (1pt);
\node at (axis cs: -5.7,5) {\scriptsize{$r^{*} e^{i \vartheta^{*}}$}};
\addplot [color=gray,line width=0.3pt] coordinates {(-5.7,4.5) (-3.2,3.2)};
\fill (axis cs: -3.2,3.2) circle (1pt);
\node at (axis cs: 6,-0.8) {\scriptsize{$r_{+}(\bar{H},\bar{L})$}};
\draw [color=black, densely dashed, line width=0.3pt] (axis cs: 0,0) circle (12.1pt);
\draw [color=black, densely dashed, line width=0.3pt] (axis cs: 0,0) circle (57pt);
\draw[->] [color=black,line width=0.8pt] (axis cs: -0.5,3.32)--(axis cs: -0.55, 3.34);
\addplot [data cs=polar, gray, dashed, domain=0:1800, samples=360, smooth] (x,{2.41/(0.801234*cos(3*x/5)+1.24444)});
\addplot [data cs=polar, black, line width=0.9pt, domain=550:880, samples=360, smooth] (x,{2.41/(0.801234*cos(3*x/5)+1.24444)});
\end{axis}
\end{tikzpicture}
\captionof{figure}{The two figures on the top provide a graphical representation of the torus $\mathcal{T}_{(\bar{H},\bar{L})}^{0}$; precisely $\mathcal{T}_{(\bar{H},\bar{L})}^{0}$ can be meant as the cartesian product between the closed orbit of energy $\bar{H}$ in the $(r, \dot r; \bar{L})$-plane (on the left) and the circle $L=\bar{L}$ in the $(L,\vartheta)$-cylinder (on the right). The figure at the bottom shows the behaviour of a solution $x$ in the plane. Notice that $\psi$ is the angle, measured in the counter-clockwise sense from $\vartheta=0$, of the greatest non-positive instant $t$ in which $x$ lies at the pericenter.
}
\label{fig-02}
\end{figure}
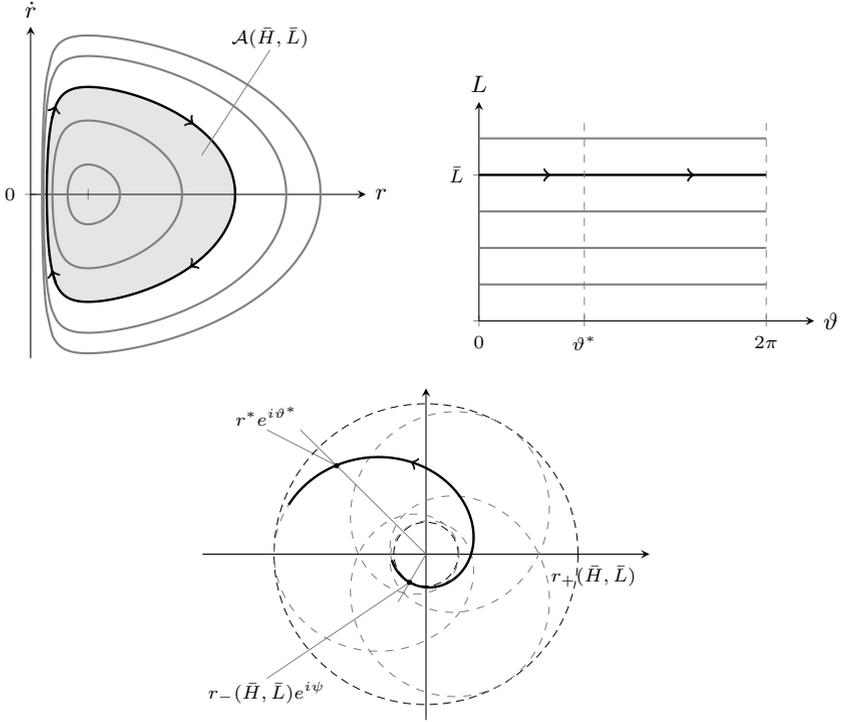
 
The action variables $(I_{1},I_{2})$ are defined as
\begin{equation}\label{def-actions}
I_{1} = \dfrac{1}{2\pi} \mathcal{A}(H,L) + L, \qquad I_{2} =L,
\end{equation}
where $\mathcal{A}(H,L)$ is the area of the bounded region enclosed by the orbit passing through $(r_{\pm}(H,L),0)$ in the $(r,\dot r)$-plane. 
Of course $(I_{1},I_{2})$ are constant on $\mathcal{T}_{(H,L)}^{0}$. 

As for the angles, we take
\begin{equation}\label{def-angles}
\varphi_{1} = 2\pi \dfrac{\mu}{T(H,L)}, \qquad \varphi_{2} = ( \Theta(H,L)-2\pi ) \dfrac{\mu}{T(H,L)} + \psi.
\end{equation}
In the above formula
\begin{itemize} 
\item $\mu \in [0,T(H,L))$ is the time needed to reach the point $(r,\dot r)$ moving in the $(r,\dot r)$-plane along the orbit starting at $(r_{-}(H,L),0)$;
\item $\psi \in [0,2\pi)$ is the angle, measured in the counter-clockwise sense from $\vartheta=0$, of the greatest non-positive instant $t$ in which $|x(t)|=r_{-}(H,L)$, where $x$ is the solution with initial condition $(x(0),p(0))$ determined by $(r, \dot r, \vartheta, L)$, see Figure~\ref{fig-02}.
\end{itemize}
Notice that, in the definition of $\varphi_2$, possible discontinuities of the angle $\psi$ are remedied by the angle defined by $\mu$; in this way, $\varphi_1$ and $\varphi_2$ are continuously well-defined as variables in $\mathbb{T} = \mathbb{R}/2\pi \mathbb{Z}$.

The above construction of the action-angle coordinates can be done starting from every torus $\mathcal{T}^{0}_{(\tilde{H},\tilde{L})}$ with $(\tilde{H},\tilde{L})$ in a neigbourhood of $(H,L)$. Therefore, the transformation
\begin{equation*}
\Sigma\colon (x,p) \mapsto (I_{1},I_{2},\varphi_{1},\varphi_{2})
\end{equation*}
is well-defined for $(x,p)$ in a neighbourhood of the fixed torus $\mathcal{T}_{(H,L)}^{0}$ and it is well-known that it provides a symplectic diffeomorphism from this set onto its image.

In the new variables, system \eqref{hamilt-syst} takes the simpler form
\begin{equation}\label{hamilt-act-ang}
\dot I_{1} = 0, \qquad \dot I_{2} = 0, 
\qquad \dot \varphi_{1} = \partial_{I_{2}} \mathcal{K}(I_{1},I_{2}), 
\qquad \dot \varphi_{2} = \partial_{I_{1}} \mathcal{K}(I_{1},I_{2}), 
\end{equation}
where 
\begin{equation}\label{def-k}
\mathcal{K}(I_{1},I_{2}) = \mathcal{H}(\Sigma^{-1}(I_{1},I_{2},\varphi_{1},\varphi_{2})),
\end{equation}
meaning that in the new coordinates the Hamiltonian $\mathcal{K}$ depends only on the action variables. 

\begin{remark}\label{rem-2.2}
A similar scenario holds for central force problems driven by the relativistic operator, namely
\begin{equation*}
\frac{\mathrm{d}}{\mathrm{d}t}\left( \frac{\dot x}{\sqrt{1 - |\dot x|^{2}/c^{2}}}\right) = V'(|x|) \frac{x}{|x|},
\end{equation*}
with $c > 0$ denoting the speed of light. Indeed, the above system can be written in the Hamiltonian form \eqref{hamilt-syst} with
\begin{equation*}
\mathcal{H}(x,p) = c^{2} \sqrt{1 + \dfrac{|p|^{2}}{c^{2}}} - V(|x|);
\end{equation*}
moreover, the angular momentum defined in \eqref{def-angh} is still a first integral.
Introducing the \emph{linear momentum}
\begin{equation*}
l = \Big{\langle} \dfrac{x}{|x|}, p \Big{\rangle},
\end{equation*}
for $x = r e^{i\vartheta}$, one obtains
\begin{equation*}
\begin{cases}
\, \dot r = \dfrac{l}{\sqrt{1+\dfrac{l^{2}+L^{2}/r^{2}}{c^{2}}}}, \vspace{7pt}\\
\, \dot l = \dfrac{L^{2}}{r^{3}} \dfrac{1}{\sqrt{1+\dfrac{l^{2}+L^{2}/r^{2}}{c^{2}}}} - V'(r), \vspace{7pt}\\
\, \dot \vartheta = \dfrac{L}{r^{2}}\dfrac{1}{\sqrt{1+\dfrac{l^{2}+L^{2}/r^{2}}{c^{2}}}}, 
\end{cases}
\end{equation*}
where $L$ is the angular momentum.
Then, one can first study the dynamics in the $(r,l)$-plane (with $L$ fixed) so as to define, for closed orbits, the radial 
period $T_{\mathrm{rel}}(H,L)$. Using the equation for $\dot \vartheta$, the apsidal angle $\Theta_{\mathrm{rel}}(H,L)$ can thus be defined as in \eqref{def-Theta-HL}. The construction of the action-angle variables is exactly the same.

The above analysis has been carried out in detail in \cite{BoDaFe-pp}, dealing with the relativistic Kepler problem
\begin{equation*}
\frac{\mathrm{d}}{\mathrm{d}t}\left( \frac{\dot x}{\sqrt{1 - |\dot x|^{2}/c^{2}}}\right) = -\kappa \, \frac{x}{|x|^{3}},
\end{equation*}
where $\kappa > 0$. In particular, explicit formulas for the radial period and the apsidal angle can be provided, precisely
\begin{equation}\label{formule-keprel}
T_{\mathrm{rel}}(H,L) = \dfrac{2\pi \kappa c^{3}}{( c^{4} - H^{2} )^{\frac{3}{2}}},
\qquad
\Theta_{\mathrm{rel}}(H,L) = \dfrac{2\pi}{\sqrt{1 - \dfrac{\kappa^{2}}{c^{2} L^{2}}}}.
\end{equation}
See also Remark~\ref{rem-relativistic}.
\hfill$\lhd$
\end{remark}

\section{The main result}\label{section-3}

In this section, we state and prove our main result for system
\begin{equation}\label{eq-main}
\ddot x = V'(|x|) \frac{x}{|x|} + \varepsilon \,\nabla_x U(t,x), \quad x \in \mathbb{R}^{2} \setminus \{0\},
\end{equation}
where $\varepsilon\in\mathbb{R}$, $V\colon (0,+\infty) \to \mathbb{R}$ is $\mathcal{C}^{2}$ and $U\colon \mathbb{R} \times (\mathbb{R}^{2} \setminus \{0\}) \to \mathbb{R}$ is continuous, $\tau$-periodic in the first variable, and differentiable with respect to the second variable with $\nabla_{x} U$ continuous in $(t,x)$.

Using the notation of Section~\ref{section-2}, for a energy-angular momentum pair $(H,L)\in\Lambda$ we set
\begin{equation*}
\mathcal{D}(H,L) = \partial_{H} T(H,L) \cdot \partial_{L} \Theta(H,L) - \partial_{L} T(H,L) \cdot \partial_{H} \Theta(H,L).
\end{equation*}
Notice that the functions $T$ and $\Theta$ are of class $\mathcal{C}^{1}$ (see Remark~\ref{rem-regTTheta}) and so $\mathcal{D}$ is a well-defined continuous function.

With the above notation and recalling the definition of the invariant torus $\mathcal{T}_{(H,L)}^{0}$ given in \eqref{def-t0}, the following theorem can be stated.

\begin{theorem}\label{th-main}
Let us assume that, for some $(H^{*},L^{*})\in\Lambda$, the torus $\mathcal{T}_{(H^{*},L^{*})}^{0}$ is filled by periodic solutions of minimal period $\tau$.
If, moreover
\begin{equation}\label{cond-d}
\mathcal{D}(H^{*},L^{*}) \neq 0,
\end{equation}
then there exists $\varepsilon^{*} > 0$ such that, for every $\varepsilon\neq 0$ with $| \varepsilon | < \varepsilon^{*}$, system \eqref{eq-main} has at least three $\tau$-periodic solutions bifurcating from the torus $\mathcal{T}_{(H^{*},L^{*})}^{0}$.
\end{theorem}

A couple of remarks about the statement are in order. 
First of all, we recall that, as discussed in Section~\ref{section-2}, a necessary and sufficient condition for the torus $\mathcal{T}_{(H^{*},L^{*})}^{0}$ to be filled by periodic solutions is
\begin{equation}\label{s-star}
\Theta(H^{*},L^{*}) = 2\pi \frac{k}{n} 
\end{equation}
for some coprime positive integers $n$ and $k$; in such a case, we called the corresponding solutions of type $(n,k)$, meaning that their winding number around the 
origin is $k$ and the winding number of $(r,\dot r)$ around the point $(r_{0}(L^{*}),0)$ is $n$. Hence, in Theorem~\ref{th-main} we are implicitly requiring that $\tau = n T(H^{*},L^{*})$.

Second, we comment on the expression \emph{bifurcating from the torus} $\mathcal{T}_{(H^{*},L^{*})}^{0}$. By this, we mean that each of the above solutions $x_\varepsilon$ remains arbitrarily close, for $\varepsilon \to 0$, to a solution $x_{0}$ of the unperturbed system with
$(x_{0}(0),\dot x_{0}(0)) \in \mathcal{T}_{(H^{*},L^{*})}^{0}$. In particular, for $\varepsilon$ small enough, the solutions $x_\varepsilon$ have winding number around the origin equal to $k$. We also remark that, by changing the sign of $L$, other three solutions (with negative winding number $-k$) can be provided.

\begin{proof}
The proof relies on a perturbation theorem for completely integrable Hamiltonian systems proved in \cite[Section~2]{FoGaGi-16}; more precisely, we are going to apply a simplified version of this result as discussed in \cite[Section~3.1]{BoDaFe-pp}. 
To this end, we first write system \eqref{eq-main} in Hamiltonian form as
\begin{equation}\label{eq-he}
\dot x = \nabla_{p}\mathcal{H}_\varepsilon(x,p), \qquad \dot p = -\nabla_{x}\mathcal{H}_\varepsilon(x,p),
\end{equation}
where
\begin{equation*}
\mathcal{H}_\varepsilon(x,p) = \frac{1}{2}|p|^{2} - V(|x|) - \varepsilon \,U(t,x),
\end{equation*}
and we then pass to action-angle variables $(I_1,I_2,\varphi_1,\varphi_2) = \Sigma(x,p)$ in a neighborhood of the torus $\mathcal{T}_{(H^{*},L^{*})}^{0}$, as described in Section~\ref{section-2.2}.
System \eqref{eq-he} thus takes the form
\begin{equation}\label{hamilt-pert}
\begin{cases}
\, \dot I_1 = -\varepsilon \, \partial_{\varphi_1}\mathcal{R}(t,I_1,I_2,\varphi_1,\varphi_2), \\
\, \dot I_2 = -\varepsilon \, \partial_{\varphi_2}\mathcal{R}(t,I_1,I_2,\varphi_1,\varphi_2), \\
\, \dot \varphi_1 = \partial_{I_1}\mathcal{K}(I_1,I_2) + \varepsilon \, \partial_{I_1}\mathcal{R}(t,I_1,I_2,\varphi_1,\varphi_2), \\
\, \dot \varphi_2 = \partial_{I_2}\mathcal{K}(I_1,I_2) + \varepsilon \, \partial_{I_2}\mathcal{R}(t,I_1,I_2,\varphi_1,\varphi_2),
\end{cases}
\end{equation}
where $\mathcal{K}$ is defined in \eqref{def-k} and 
\begin{equation*}
\mathcal{R}(t,I_1,I_2,\varphi_1,\varphi_2) = U(t,x(I_1,I_2,\varphi_1,\varphi_2)),
\end{equation*}
with $x(I_1,I_2,\varphi_1,\varphi_2)$ the two-dimensional vector made by the first two components of the four-dimensional vector $\Sigma^{-1}(I_1,I_2,\varphi_1,\varphi_2)$. 
System \eqref{hamilt-pert} looks like a perturbation, for $\varepsilon \neq 0$ and small, of system \eqref{hamilt-act-ang}. According to \cite[Theorem~3.1]{BoDaFe-pp}, to prove our result we thus need to check that
conditions
\begin{equation}\label{cond-1}
\tau \nabla \mathcal{K}(I_1^{*},I_2^{*}) \in 2\pi\mathbb{Z}^{2}
\end{equation}
and 
\begin{equation}\label{cond-2}
\mathrm{det\,} \nabla^{2} \mathcal{K}(I_1^{*},I_2^{*}) \neq 0
\end{equation}
are satisfied, where $(I_1^{*},I_2^{*})$ are the actions corresponding to $(H^{*},L^{*})$ via \eqref{def-actions}.

In order to simplify the computations, from now on system \eqref{hamilt-pert} is meant on the covering space, that is, the angles $\varphi_1, \varphi_2$ are allowed take values in the whole real line. With this in mind, recalling \eqref{hamilt-act-ang}, and in particular $\dot \varphi_i = \partial_{I_i} \mathcal{K}(I_{1},I_{2})$ for $i=1,2$, we deduce that
\begin{equation}\label{gradiente-angolo}
 (\varphi_{1}(\tau)-\varphi_{1}(0), \varphi_{2}(\tau)-\varphi_{2}(0)) = \tau \nabla \mathcal{K}(I_{1},I_{2}).
\end{equation}
On the other hand, from the explicit expression of the angle variables
in \eqref{def-angles} we have that
\begin{align*}
& \varphi_1(\tau) - \varphi_1(0) = 2\pi \dfrac{\tau}{T(H,L)},
\\
& \varphi_2(\tau) - \varphi_2(0) = ( \Theta(H,L) - 2\pi ) \dfrac{\tau}{T(H,L)}. 
\end{align*}

According to \eqref{s-star}, solutions on $\mathcal{T}_{(H^{*},L^{*})}^{0}$ are periodic of type $(n,k)$ with minimal period $\tau = n T(H^{*},L^{*})$. 
For such solutions, we thus have
\begin{equation*}
(\varphi_{1}(\tau)-\varphi_{1}(0), \varphi_{2}(\tau)-\varphi_{2}(0))= 2\pi (n, k-n).
\end{equation*}
Recalling \eqref{gradiente-angolo}, we infer
\begin{equation*}
\tau \nabla \mathcal{K}(I_{1}^{*},I_{2}^{*}) = 2\pi (n, k-n),
\end{equation*}
so that \eqref{cond-1} is verified.

As for condition \eqref{cond-2}, we first introduce the notation
\begin{equation*}
\Phi(H,L) = \biggl{(} 2\pi \dfrac{\tau}{T(H,L)}, ( \Theta(H,L)-2\pi )\dfrac{\tau}{T(H,L)} \biggr{)}
\end{equation*}
and
\begin{equation*}
\Psi (H,L) = \biggl{(} \dfrac{1}{2\pi} \mathcal{A}(H,L) + L, L \biggr{)}.
\end{equation*}
The map $\Phi$ is of class $\mathcal{C}^{1}$ since $T$ and $\Theta$ are so. Moreover, the regularity of $\Sigma^{-1}$ implies that map $\Psi$ is of class $\mathcal{C}^{1}$ as well. As proved in \cite[p.~282]{Ar-89}, $\partial_{H} \mathcal{A}(H,L)=T(H,L)$ and so
\begin{equation*}
\mathrm{det\,} \mathrm{D}\Psi(H,L) = 
\mathrm{det} 
\begin{pmatrix}
\dfrac{\partial_{H} \mathcal{A}(H,L)}{2\pi} & \dfrac{\partial_{L} \mathcal{A}(H,L)}{2\pi} + 1 \vspace{6pt} \\
0 & 1
\end{pmatrix} 
 \neq 0.
\end{equation*}
Since $\Psi$ is a one-to-one map, we have $\mathrm{det\,} \mathrm{D}\Psi^{-1} \neq 0$ as well.

According to this notation and the previous computations, we deduce
\begin{equation*}
\Phi(\Psi^{-1}(I_{1},I_{2})) = \tau \nabla \mathcal{K}(I_{1},I_{2}).
\end{equation*}
By differentiating, we find
\begin{equation*}
\nabla^{2} K(I_{1},I_{2}) 
= \dfrac{1}{\tau} \; \mathrm{D}(\Phi(\Psi^{-1}(I_{1},I_{2})))
= \dfrac{1}{\tau} \; \mathrm{D} \Phi\Big{\vert}_{(H,L) = \Psi^{-1}(I_1,I_2)} \; \mathrm{D} \Psi^{-1}(I_{1},I_{2})
\end{equation*}
and, consequently,
\begin{equation*}
\mathrm{det\,} \nabla^{2} K(I_{1},I_{2}) =
\dfrac{1}{\tau^{2}} \,
\mathrm{det\,} \mathrm{D} \Phi\Big{\vert}_{(H,L) = \Psi^{-1}(I_1,I_2)} \cdot 
\mathrm{det\,} \mathrm{D}\Psi^{-1}(I_{1},I_{2}).
\end{equation*}
Since $\mathrm{D}\Psi^{-1}(I_{1},I_{2})$ is invertible, to prove \eqref{cond-2} we just need to verify that 
\begin{equation*}
\mathrm{det\,} \mathrm{D} \Phi(H^{*},L^{*}) \neq 0.
\end{equation*}
By direct computations (and omitting the dependence on $(H,L)$ to simplify the notation), we have
\begin{equation*}
\mathrm{D} \Phi(H,L) = \dfrac{\tau}{T}
\begin{pmatrix}
-2\pi\dfrac{\partial_{H}T}{T} 
& -2\pi\dfrac{\partial_{L}T}{T}
\vspace{2pt}\\ 
\partial_{H} \Theta - (\Theta-2\pi )\dfrac{\partial_{H}T}{T} 
& \partial_{L} \Theta - (\Theta-2\pi )\dfrac{\partial_{L}T}{T}
\end{pmatrix},
\end{equation*}
so that
\begin{align*}
&\mathrm{det\,} \mathrm{D}\Phi(H,L) =
\\
&= -2\pi \dfrac{\tau^{2}}{(T(H,L))^{3}} 
\bigl{[} \partial_{H}T(H,L)\cdot\partial_{L}\Theta(H,L) - \partial_{L}T(H,L)\cdot\partial_{H}\Theta(H,L)\bigr{]} 		
\\
&= -2\pi \dfrac{\tau^{2}}{(T(H,L))^{3}} \mathcal{D}(H,L).
\end{align*}
Using assumption \eqref{cond-d}, the conclusion follows.
\end{proof}

\begin{remark}\label{rem-3.1}
We observe that, with no changes in the proof, the result is still valid if the perturbation term $U$ depends also on $\varepsilon$, that is 
$U = U(t,x,\varepsilon)$, provided that $U$ and $\nabla_x U$ remain bounded as $\varepsilon \to 0$. Moreover, by the local nature of the result, we could assume that $U$ (as well as $V$ itself) is defined only on a neighborhood of the planar annulus 
$\{ x \in \mathbb{R}^{2} \colon r_{-}(H^{*},L^{*}) \leq |x| \leq r_{+}(H^{*},L^{*}) \}$.
Finally, we mention that the fact that the periodic solutions filling the torus $\mathcal{T}_{(H^{*},L^{*})}^{0}$ have \emph{minimal} period equal to $\tau$ is unnecessary. In a similar way, one could bifurcate from tori made by solutions with smaller minimal period, namely $\tau = \ell n T(H^{*},L^{*})$ for some $\ell \geq 2$. 
All these observations will be crucial in Section~\ref{section-5}.
\hfill$\lhd$
\end{remark}

\begin{remark}\label{rem-relativistic}
A version of Theorem~\ref{th-main} for perturbation of relativistic central force fields, namely
\begin{equation*}
\frac{\mathrm{d}}{\mathrm{d}t}\left( \frac{\dot x}{\sqrt{1 - |\dot x|^{2}/c^{2}}}\right) = V'(|x|) \frac{x}{|x|} + \varepsilon \,\nabla_x U(t,x),
\end{equation*}
can be provided, as well. Indeed, as explained in Remark~\ref{rem-2.2} the construction of the action-angle variables for the unperturbed problem is exactly the same; accordingly, periodic solutions bifurcating from an invariant torus of energy $H^{*}$ and angular momentum $L^{*}$ can be found whenever condition \eqref{cond-d} is satisfied, with radial period $T_{\mathrm{rel}}$ and apsidal angle $\Theta_{\mathrm{rel}}$
as in Remark~\ref{rem-2.2}. 

In particular, for the Kepler potential $V(|x|) = \kappa/|x|$ (with $\kappa > 0$), from the explicit formulas \eqref{formule-keprel} one immediately obtains $\mathcal{D}(H,L) \neq 0$. Thus, the relativistic variant of Theorem~\ref{th-main} applies, in agreement with the analysis performed in \cite{BoDaFe-pp}, requiring instead the explicit computations of the Hamiltonian in action-angle coordinates.
\hfill$\lhd$
\end{remark}

\section{Some examples}\label{section-4}

In this section, we present some applications of Theorem~\ref{th-main} to various central force problems of physical interest.
More precisely, in Section~\ref{section-4.1} we deal with the (perturbed) Levi-Civita equation
\begin{equation}\label{eq-lc}
\ddot x = -\kappa \dfrac{x}{|x|^{3}} - 2\lambda \dfrac{x}{|x|^4} + \varepsilon \,\nabla_x U(t,x), \quad \kappa,\lambda > 0,
\end{equation}
introduced in \cite{LeCi-28} as a relativistic correction of the Kepler problem (see also \cite{AmBe-90,FoGa-17} and the references therein).
In Section~\ref{section-4.2} we analyze the (perturbed) homogeneous central force problem 
\begin{equation}\label{eq-hom}
\ddot x = -\kappa \dfrac{x}{|x|^{\alpha + 2}} + \varepsilon \,\nabla_x U(t,x), \quad \kappa > 0,\, \alpha < 2, \,\alpha \notin \{-2,1\},
\end{equation}
modeling sublinear/superlinear oscillators when $\alpha \leq -1$ (cf.~\cite{BaBe-84,MaWi-89}) and generalizing the Kepler problem when $\alpha \in (-1,2)$ (cf.~\cite{AmCo-89,McG-81}). 
Finally, in Section~\ref{section-4.3} we consider the (perturbed) Lennard-Jones equation
\begin{equation}\label{eq-lj}
\ddot x = -24 \ee \sigma^6 \frac{x}{|x|^8} + 48 \ee \sigma^{12} \frac{x}{|x|^{14}} + \varepsilon \,\nabla_x U(t,x), 
\quad \ee,\sigma > 0,
\end{equation}
introduced in \cite{LJ-31} to describe intermolecolar interactions (see also \cite{Bi-02,CLPC-04} and the references therein).

For all the above problems, we are going to show that Theorem~\ref{th-main} can be applied by checking that condition \eqref{cond-d}
holds true for the unperturbed problem. In Section~\ref{section-4.1} this is done by explicit computations for the functions $T$ and $\Theta$; in Section~\ref{section-4.2}, instead, explicit formula cannot be provided and we take advantage of the homogeneity of the problem to prove that the sign of $\mathcal{D}(H,L)$ is the same as the one of $\partial_{H} \Theta(H,L)$, which is studied in \cite{Ca-15,Ro-18}.
The application in Section~\ref{section-4.3} is more involved: in this case, we are able to prove the result only for values of the energy near the minimum, by developing careful asymptotic expansions for the maps $T$ and $\Theta$ and their derivatives (see Appendix~\ref{appendix-A}).

We stress that, in the following, we are just going to prove that condition~\eqref{cond-d} is satisfied, without caring about the minimal period of the invariant torus from which the bifurcation occurs. Thus, our next applications have to be meant in the following way: fixed an invariant torus filled by periodic solutions, with minimal period $\tau$, of the unperturbed problem, at least six $\tau$-periodic solutions (three with positive angular momentum and three with negative angular momentum) exist if a sufficiently small perturbation term, periodic in time with the same period $\tau$, is added.
The fact that tori filled by periodic solutions always exists is guaranteed since $\Theta$ is not constant.

Of course, one could also try to study, for a period $\tau$ fixed in advance, if $\tau$-periodic solutions of the unperturbed problem actually exist (and how many) and then bifurcate from the associated invariant torus. This requires however a study of the range of the functions $T$ and $\Theta$ which we will provide only for the homogeneous potential, see Section~\ref{section-5.1} and Section~\ref{section-5.2}.

\subsection{Levi-Civita potential}\label{section-4.1}

Let us consider the potential
\begin{equation*}
V(r)=\dfrac{\kappa}{r}+\dfrac{\lambda}{r^{2}},
\end{equation*}
where $\kappa,\lambda>0$, giving rise to \eqref{eq-lc}.

Recalling the notation introduced in Section~\ref{section-2}, for every $L > 0$ we have
\begin{align*}
&W(r;L) = \dfrac{L^{2}}{2r^{2}} - \dfrac{\kappa}{r} - \dfrac{\lambda}{r^{2}}
= \dfrac{1}{2r^{2}} \bigl{(} - 2 \kappa r + L^{2} - 2 \lambda \bigr{)},
\\
&W'(r;L) = - \dfrac{L^{2}}{r^{3}} + \dfrac{\kappa}{r^{2}} + \dfrac{2 \lambda}{r^{3}}
= \dfrac{1}{r^{3}} \bigl{(} \kappa r + 2 \lambda - L^{2} \bigr{)}.
\end{align*}
Therefore, $W'(\cdot;L)$ changes sign (once) in $(0,+\infty)$ if and only if
\begin{equation*}
L^{2} > 2 \lambda.
\end{equation*}
In this case $W(\cdot;L)$ has a unique strict global minimum at
\begin{equation*}
r_{0}(L) = \dfrac{L^{2}-2\lambda}{\kappa},
\end{equation*}
with
\begin{equation*}
w_{0}(L) = -W(r_{0}(L);L) = \dfrac{\kappa^{2}}{2(L^{2}-2\lambda)} > 0.
\end{equation*}
Moreover, $W'(\cdot;L)<0$ in $(0,r_{0}(L))$, and $W'(\cdot;L)>0$ in $(r_{0}(L),+\infty)$.
Therefore condition $(h_{W})$ is satisfied. Observing, moreover, that 
\begin{equation*}
\lim_{r\to0^{+}} W(r;L) = +\infty, \qquad \lim_{r\to+\infty} W(r;L) = 0,
\end{equation*}
we infer that the equation $W(r;L) = H$ has exactly two solutions $r_{\pm}(H,L)$ if and only if $-w_{0}(L) < H < 0$. Thus, recalling \eqref{dom-HL}, we have
\begin{equation*}
\Lambda = \biggl{\{}(H,L)\in\mathbb{R}^{2} \colon L\in (2\lambda,+\infty), \; H\in \left(-\dfrac{\kappa^{2}}{2(L^{2}-2\lambda)}, 0\right) \biggr{\}}.
\end{equation*}
For further convenience, we also observe that the two solutions $r_{\pm}(H,L)$ satisfy
\begin{equation}\label{eq-formular}
(r-r_{+}(H,L))(r-r_{-}(H,L)) = r^{2}+\dfrac{\kappa}{H} r + \dfrac{2\lambda-L^{2}}{2H} = 0.
\end{equation}

We are going to compute explicitly $T$ and $\Theta$ and, in consequence, their derivatives.
As for $T$, using \eqref{eq-formular}, we find
\begin{align*}
T(H,L) 
&= \sqrt{2} \int_{r_{-}(H,L)}^{r_{+}(H,L)} \dfrac{\mathrm{d}r}{\sqrt{H-W(r;L)}}
\\
&= \dfrac{\sqrt{2}}{\sqrt{-H}} \int_{r_{-}(H,L)}^{r_{+}(H,L)} \dfrac{r \,\mathrm{d}r}{\sqrt{(r_{+}(H,L)-r)(r-r_{-}(H,L))}}.
\end{align*}
By performing the change of variable
\begin{equation}\label{change-var}
r= \dfrac{r_{+}(H,L)+r_{-}(H,L) \, u^{2}}{u^{2}+1}, 
\qquad
u=\sqrt{\dfrac{r_{+}(H,L)-r}{r-r_{-}(H,L)}},
\end{equation}
we thus obtain
\begin{align*}
T(H,L) 
&= \dfrac{2\sqrt{2}}{\sqrt{-H}} \int_{0}^{+\infty} \dfrac{r_{+}(H,L)+r_{-}(H,L) u^{2}}{(u^{2}+1)^{2}} \,\mathrm{d}u
\\
&= \dfrac{2\sqrt{2}}{\sqrt{-H}} \biggl{[} \dfrac{r_{+}(H,L)-r_{-}(H,L)}{2}\dfrac{u}{u^{2}+1}
+ \dfrac{r_{+}(H,L)+r_{-}(H,L)}{2}\arctan(u) \biggr{]}_{0}^{+\infty}
\\
&= \dfrac{\pi \kappa}{\sqrt{2} (-H)^{\frac{3}{2}}}.
\end{align*}
Therefore, 
\begin{equation}\label{LC-T}
\partial_{H} T(H,L) = \dfrac{3\pi \kappa}{2\sqrt{2} (-H)^{\frac{5}{2}}} > 0,
\qquad
\partial_{L} T(H,L) = 0.
\end{equation}
Notice that $T$ is independent on $L$.

As for $\Theta$, by exploiting again the change of variable \eqref{change-var}, we get
\begin{align*}
\Theta(H,L) 
&= \sqrt{2} L \int_{r_{-}(H,L)}^{r_{+}(H,L)} \dfrac{\mathrm{d}r}{r^{2} \sqrt{H-W(r;L)}}
\\
&= \dfrac{\sqrt{2} L }{\sqrt{-H}} \int_{r_{-}(H,L)}^{r_{+}(H,L)} \dfrac{\mathrm{d}r}{r\sqrt{(r_{+}(H,L)-r)(r-r_{-}(H,L))}}
\\
&= \dfrac{2\sqrt{2} L}{\sqrt{-H}} \int_{0}^{+\infty} \dfrac{\mathrm{d}u}{r_{+}(H,L)+r_{-}(H,L) \, u^{2}}
\\
&= \dfrac{2\sqrt{2} L}{\sqrt{-H}} \Biggl{[}\dfrac{1}{\sqrt{r_{+}(H,L) \cdot r_{-}(H,L)} } \; \arctan \Biggl{(}\sqrt{\dfrac{r_{-}(H,L)}{r_{+}(H,L)}} \, u \Biggr{)}\, \Biggr{]}_{0}^{+\infty}
\\
&= \dfrac{2\pi L}{\sqrt{L^{2}-2 \lambda}}.
\end{align*}
This implies that
\begin{equation}\label{LC-Theta}
\partial_{H} \Theta(H,L) = 0,
\qquad
\partial_{L} \Theta(H,L) = -\dfrac{4 \pi \lambda}{(L^{2}-2 \lambda)^{\frac{3}{2}}} <0.
\end{equation}
Notice that $\Theta$ is independent on $H$.

Summing up, from \eqref{LC-T} and \eqref{LC-Theta}, we can compute
\begin{equation*}
\mathcal{D}(H,L) = \partial_{H} T(H,L) \cdot \partial_{L} \Theta (H,L)
= -\dfrac{3 \sqrt{2} \pi^{2} \kappa \lambda}{(-H)^{\frac{5}{2}} (L^{2}-2 \lambda)^{\frac{3}{2}}} <0,
\end{equation*}
so that condition \eqref{cond-d} holds true and Theorem~\ref{th-main} can be applied.

\begin{remark}\label{rem-4.1}
Notice that the previous computations for $T$ and $\Theta$ are valid for every $\lambda \in \mathbb{R}$. In particular, 
it is easily checked that condition \eqref{cond-d} is satisfied also when $\lambda < 0$. 
On the other hand, for $\lambda=0$ (corresponding to the Kepler potential $V(r) = \kappa/r$) we find $\mathcal{D}(H,L) = 0$ for all $(H,L)\in\Lambda$, and so Theorem~\ref{th-main} cannot be applied. Incidentally, notice that in this case $\Theta \equiv 2\pi$, as already observed in Remark~\ref{rem-2.1}.
\hfill$\lhd$
\end{remark}

\subsection{Homogeneous potentials}\label{section-4.2}

For $\alpha\in(-\infty,0)\cup(0,2)$ and $\kappa > 0$, let us consider the potential
\begin{equation}\label{hom-pot}
V(r)=\dfrac{\kappa}{\alpha r^{\alpha}}, 
\end{equation}
giving rise to \eqref{eq-hom}.

Recalling the notation introduced in Section~\ref{section-2}, for all $L > 0$ we have
\begin{align}
&W(r;L) = \dfrac{L^{2}}{2r^{2}}-\dfrac{\kappa}{\alpha r^{\alpha}} 
= -\dfrac{1}{2 \alpha r^{2}} \bigl{(} 2 \kappa r^{2-\alpha} - \alpha L^{2} \bigr{)},
\label{def-WL-hom}
\\
&W'(r;L) = - \dfrac{L^{2}}{r^{3}} + \dfrac{\kappa}{r^{\alpha+1}} 
= \dfrac{1}{r^{3}} \bigl{(} \kappa r^{2-\alpha} - L^{2} \bigr{)}.
\label{def-WLprime-hom}
\end{align}
Therefore, $W'(\cdot;L)$ changes sign (once) in $(0,+\infty)$. Hence, $W(\cdot;L)$ has a unique strict global minimum at 
\begin{equation}\label{def-r0L-hom}
r_{0}(L) = \biggl(\dfrac{L^{2}}{\kappa}\biggr)^{\!\frac{1}{2-\alpha}},
\end{equation}
with
\begin{equation*}
w_{0}(L) = -W(r_{0}(L);L) = \dfrac{2-\alpha}{2\alpha}\kappa^{\frac{2}{2-\alpha}} L^{-\frac{2\alpha}{2-\alpha}} > 0. 
\end{equation*}
Moreover, $W'(\cdot;L)<0$ in $(0,r_{0}(L))$, and $W'(\cdot;L)>0$ in $(r_{0}(L),+\infty)$, so that condition $(h_{W})$ holds.
Observing furthermore that
\begin{equation*}
\lim_{r\to0^{+}} W(r;L) = +\infty, \qquad \lim_{r\to+\infty} W(r;L) = 0,
\end{equation*}
we deduce that the equation $W(r;L) = H$ has exactly two solutions $r_{\pm}(H,L)$ if and only if $-w_{0}(L) < H < 0$.
Thus, 
\begin{equation}\label{def-lambda}
\Lambda = \biggl{\{}(H,L)\in\mathbb{R}^{2} \colon L\in (0,+\infty), \; H\in \left(-\dfrac{2-\alpha}{2\alpha}\kappa^{\frac{2}{2-\alpha}} L^{-\frac{2\alpha}{2-\alpha}}, 0\right) \biggr{\}}.
\end{equation}

We now present two lemmas, showing that the homogeneity of the problem allows us to reduce the computations
of $T$ and $\Theta$ to the case $L =1$.

\begin{lemma}\label{hom-THL}
For all $(H,L)\in\Lambda$, it holds that
\begin{equation}\label{eq-T_HL-riscal}
T(H,L) = L^{\frac{2+\alpha}{2-\alpha}} T(HL^{\frac{2\alpha}{2-\alpha}},1).
\end{equation}
\end{lemma}

\begin{proof}
We first observe (and we omit the straightforward proof) that, if $r$ is a periodic function satisfying the energy relation 
\begin{equation*}
\dfrac{1}{2}\dot{r}(t)^{2} + W(r(t);L) = H,
\end{equation*} 
then the function $\rho$ defined as 
\begin{equation*}
\rho(t)=L^{-\frac{2}{2-\alpha}} r(L^{\frac{2+\alpha}{2-\alpha}}t)
\end{equation*}
satisfies
\begin{equation*}
\dfrac{1}{2}\dot{\rho}(t)^{2} + W(\rho(t);1) = H L^{\frac{2\alpha}{2-\alpha}}.
\end{equation*} 
Denoting by $T_r$ and $T_\rho$ the minimal periods of $r$ and $\rho$, respectively, it obviously holds that $T_r = T(H,L)$ and $T_\rho = T(HL^{\frac{2\alpha}{2-\alpha}},1)$. Since $T_r$ and $T_\rho$ are related by
\begin{equation*}
T_r = L^{\frac{2+\alpha}{2-\alpha}} T_\rho,
\end{equation*}
the thesis follows.
\end{proof}

Recalling the definition of the function $P$ given in \eqref{def-P-HL}, in a similar way we have the following.

\begin{lemma}\label{hom-thetaHL}
For all $(H,L)\in\Lambda$, it holds that
\begin{equation}\label{eq-P_HL-riscal}
P(H,L) = L^{-1} P(HL^{\frac{2\alpha}{2-\alpha}},1).
\end{equation}
\end{lemma}

\begin{proof}
The proof is similar to the one of Lemma~\ref{hom-THL}. Indeed, if $u$ is a periodic function satisfying
\begin{equation*}
\frac{1}{2} \dot{u}^{2} + W\biggl{(}\frac{1}{u};L\biggr{)}=H,
\end{equation*}
then the function 
\begin{equation*}
\eta(t) =L^{\frac{2}{2-\alpha}} u(L^{-1} t)
\end{equation*}
satisfies
\begin{equation*}
\dfrac{1}{2}\dot{\eta}^{2} + W\biggl{(}\frac{1}{\eta};1\biggr{)} = H L^{\frac{2\alpha}{2-\alpha}}.
\end{equation*} 
Denoting by $P_u$ and $P_\eta$ the minimal periods of $u$ and $\eta$, respectively, it holds that $P_u = P(H,L)$ and $P_\eta = P(HL^{\frac{2\alpha}{2-\alpha}},1)$. Since $P_u$ and $P_\eta$ are related by
\begin{equation*}
P_u = L^{-1} P_\eta,
\end{equation*}
the thesis follows.
\end{proof}

From the above lemmas, we can obtain the following equalities involving $T$, $\Theta$, and their derivatives.

\begin{lemma}\label{lem-sigma-2}
For all $(H,L)\in\Lambda$, it holds that
\begin{align}
\partial_{L} T(H,L) &= \dfrac{L^{-1}}{2-\alpha}\bigl{(} (2+\alpha) T(H,L) + 2 \alpha H \partial_{H} T(H,L) \bigr{)},
\label{eq-pdeT}
\\
\partial_{L} \Theta(H,L) 
&= \dfrac{2\alpha}{2-\alpha} \dfrac{H}{L} \partial_{H} \Theta(H,L).
\label{eq-pdeTheha}
\end{align}
\end{lemma}

\begin{proof}
As for \eqref{eq-pdeT}, we differentiate \eqref{eq-T_HL-riscal} with respect to $H$ and $L$ to obtain
\begin{align*}
\partial_{H} T(H,L) &= L^{\frac{2+\alpha}{2-\alpha}+ \frac{2\alpha}{2-\alpha}} \partial_{H} T(HL^{\frac{2\alpha}{2-\alpha}},1),
\\
\partial_{L} T(H,L) 
&= \frac{2+\alpha}{2-\alpha} L^{\frac{2+\alpha}{2-\alpha}-1} T(HL^{\frac{2\alpha}{2-\alpha}},1)
\\
&\quad + L^{\frac{2+\alpha}{2-\alpha}} \partial_{H} T(HL^{\frac{2\alpha}{2-\alpha}},1) H \frac{2\alpha}{2-\alpha} L^{\frac{2\alpha}{2-\alpha}-1}
\\
&= \frac{2+\alpha}{2-\alpha} L^{-1} T(H,L) + \frac{2\alpha}{2-\alpha} H L^{-1} \partial_{H} T(H,L). 
\end{align*}
Hence, \eqref{eq-pdeT} follows.

As for \eqref{eq-pdeTheha}, we first differentiate \eqref{def-P} with respect to $L$ to obtain
\begin{equation}\label{eq-AAAL}
\partial_{L} \Theta(H,L) = P(H,L) + L \, \partial_{L} P(H,L).
\end{equation}
Then, exploiting \eqref{eq-P_HL-riscal}, we have
\begin{align*}
\partial_{H} P(H,L) &= L^{-1+\frac{2\alpha}{2-\alpha}} \partial_{H} P(HL^{\frac{2\alpha}{2-\alpha}},1),
\\
\partial_{L} P(H,L) &= -L^{-2} \partial_{L} P(HL^{\frac{2\alpha}{2-\alpha}},1)
+ L^{-1} \partial_{H} P(HL^{\frac{2\alpha}{2-\alpha}},1) H \dfrac{2\alpha}{2-\alpha}L^{\frac{2\alpha}{2-\alpha}-1}
\\
&= -L^{-1} P(H,L) + \dfrac{2\alpha}{2-\alpha} L^{-1} H \partial_{H} P(H,L).
\end{align*}
Thus, using \eqref{eq-AAAL} we get
\begin{equation*}
\partial_{L} \Theta(H,L) = \dfrac{2\alpha}{2-\alpha} H \partial_{H} P(H,L).
\end{equation*}
Since $\partial_{H} \Theta(H,L) = L \partial_{H} P(H,L)$, \eqref{eq-pdeTheha} follows.
\end{proof}

Using this lemma, we can finally obtain the following formula for $\mathcal{D}(H,L)$.

\begin{lemma}\label{lem-sigma-5}
For all $(H,L)\in\Lambda$, it holds that
\begin{equation*}\label{eq-dhom}
\mathcal{D}(H,L) = - \dfrac{2+\alpha}{2-\alpha} \dfrac{1}{L} T(H,L) \partial_{H} \Theta(H,L).
\end{equation*}
\end{lemma}

\begin{proof}
Using Lemma~\ref{lem-sigma-2} we have
\begin{align*}
\mathcal{D}(H,L)
&= \partial_{H} T(H,L) \cdot \partial_{L} \Theta (H,L) - \partial_{L} T(H,L) \cdot \partial_{H} \Theta (H,L) = 
\\
&= \partial_{H} T(H,L) \cdot \dfrac{2\alpha}{2-\alpha} \dfrac{H}{L} \partial_{H} \Theta(H,L)
\\
& \quad - \dfrac{L^{-1}}{2-\alpha}\bigl{(} (2+\alpha) T(H,L) + 2 \alpha H \partial_{H} T(H,L) \bigr{)}
\cdot \partial_{H} \Theta (H,L)
\\
&= - \dfrac{2+\alpha}{2-\alpha} \dfrac{1}{L} T(H,L) \partial_{H} \Theta(H,L)
\end{align*}
and the thesis follows.
\end{proof}

For $\alpha = -2$ (corresponding to the harmonic oscillator), we have thus found $\mathcal{D}(H,L) = 0$ for all $(H,L)\in\Lambda$ and so Theorem~\ref{th-main} does not apply. For $\alpha = 1$ (corresponding to the Kepler problem) we again have $\mathcal{D}(H,L) = 0$ since $\partial_{H} \Theta(H,L) = 0$ as shown in Remark~\ref{rem-4.1}. In the other cases (namely, $\alpha < 2$ and $\alpha \notin \{-2,0,1\}$), we know from \cite{Ca-15,Ro-18} that
\begin{equation}\label{sign-cast-rojas}
\partial_{H} \Theta(H,L)
\begin{cases}
>0, & \text{if $\alpha\in(-\infty,-2)\cup(1,2)$,}
\\
<0, & \text{if $\alpha\in(-2,1)\setminus \{0\}$,}
\end{cases}
\end{equation}
for all $(H,L)\in\Lambda$.
Accordingly
\begin{equation*}
\mathcal{D}(H,L)
\begin{cases}
>0, & \text{if $\alpha\in(-\infty,1)\setminus \{-2,0\}$,}
\\
<0, & \text{if $\alpha\in(1,2)$,}
\end{cases}
\end{equation*}
for all $(H,L)\in\Lambda$.
Therefore, condition \eqref{cond-d} holds true and Theorem~\ref{th-main} can be applied.

\begin{remark}[Monotonicity of $\Theta$ with respect to $L$]\label{rem-4.2}
We notice that, using \eqref{eq-pdeTheha}, we can compute the sign of $\partial_{L} \Theta(H,L)$ thanks to \eqref{sign-cast-rojas}. Indeed, recalling that $H<0$, we obtain
\begin{equation*}
\partial_{L} \Theta(H,L)
\begin{cases}
>0, & \text{if $\alpha\in(-\infty,-2)\cup(0,1)$,}
\\
<0, & \text{if $\alpha\in(-2,0)\cup(1,2)$,}
\end{cases}
\end{equation*}
for all $(H,L)\in\Lambda$. Notice that this was explicitly raised as an open problem in \cite[p.~25]{Ca-phd}. 
\hfill$\lhd$
\end{remark}

\begin{remark}[Logarithmic potential]\label{rem-4.3}
For $\alpha = 0$, the definition of the potential $V$ in \eqref{hom-pot} is clearly meaningless. However, the central force problem \eqref{eq-hom} still makes sense, coming from the logarithmic potential
\begin{equation*}
V(r) = -\kappa \log r.
\end{equation*}
with $\kappa>0$.
In this case, in spite of the missing homogeneity of $V$, one can still argue in a similar way as before. 
More precisely,
exploiting the scaling $\rho(t) = L^{-1}r(Lt)$ it is easily proved that
\begin{equation*}
T(H,L) = L T(H-\log L,1),
\end{equation*}
while, using the scaling $\eta(t) = L u(L^{-1}t)$, one can deduce that
\begin{equation*}
P(H,L) = L^{-1} P(H - \log L ,1).
\end{equation*}
Differentiating these formulas, we reach
\begin{equation*}
\partial_{L} T(H,L) = L^{-1}\bigl{(} T(H,L) - \partial_{H} T(H,L) \bigr{)}
\end{equation*}
and
\begin{equation*}
\partial_{L} \Theta(H,L) = -L^{-1}\partial_{H} \Theta(H,L).
\end{equation*}
Therefore
\begin{equation*}
\mathcal{D}(H,L) = - L^{-1} T(H,L) \partial_{H} \Theta(H,L).
\end{equation*}
Incidentally, notice that the above formula corresponds exactly to formula \eqref{eq-dhom} for $\alpha = 0$.
Now, recalling \cite{Ca-15}, we known that $\partial_{H} \Theta(H,L) < 0$ and so 
$\mathcal{D}(H,L) > 0$. Hence, also in the case of the logarithmic potential condition \eqref{cond-d} is satisfied. We omit the details for briefness.
\hfill$\lhd$
\end{remark}

\subsection{Lennard-Jones potentials}\label{section-4.3}

Let us consider the potential
\begin{equation*}
V(r) = 4\ee \dfrac{\sigma^{6}}{r^{6}} - 4\ee \dfrac{\sigma^{12}}{r^{12}},
\end{equation*}
where $\ee,\sigma>0$, giving rise to \eqref{eq-lj}.

Recalling the notation introduced in Section~\ref{section-2}, for every $L>0$, we have
\begin{equation*}
W(r;L) = \dfrac{L^{2}}{2r^{2}} - 4\ee \dfrac{\sigma^{6}}{r^{6}} + 4\ee \dfrac{\sigma^{12}}{r^{12}},
\end{equation*}
which satisfies
\begin{equation*}
\lim_{r\to0^{+}} W(r;L) = +\infty, \qquad \lim_{r\to+\infty} W(r;L) = 0.
\end{equation*}

We analyze the monotonicity of $W(\cdot;L)$. Accordingly, we compute
\begin{equation*}
W'(r;L) = -\dfrac{L^{2}}{r^{3}} +24 \ee \dfrac{\sigma^{6}}{r^{7}} - 48 \ee \dfrac{\sigma^{12}}{r^{13}}
= \dfrac{1}{r^{13}} \bigl{(} - L^{2}r^{10} + 24 \ee \sigma^{6} r^{6} - 48 \ee \sigma^{12} \bigr{)}.
\end{equation*}
Let us set
\begin{equation*}
n(r):= - L^{2}r^{10} + 24 \ee \sigma^{6} r^{6} - 48 \ee \sigma^{12}.
\end{equation*}
Then,
\begin{equation*}
n'(r) = -10 L^{2}r^{9} +144 \ee \sigma^{6} r^{5} 
= 2 r^{5} \bigl{(} -5 L^{2} r^{4} + 72 \ee \sigma^{6} \bigr{)}.
\end{equation*}
Thus, $n$ is strictly increasing in $(0,\bar{r}(L))$ and strictly decreasing in $(\bar{r}(L),+\infty)$, where
\begin{equation*}
\bar{r}(L) = \biggl{(} \dfrac{72}{5} \dfrac{\ee \sigma^{6}}{L^{2}} \biggr{)}^{\!\frac{1}{4}},
\qquad
n(\bar{r}(L)) = 48 \ee \sigma^{12} \biggl{[} \dfrac{1}{5} \biggl{(}\dfrac{72}{5}\biggr{)}^{\!\frac{3}{2}} \dfrac{\ee^{\frac{3}{2}} \sigma^{3}}{L^{3}} - 1\biggr{]},
\end{equation*}
are the maximum point and the maximum of $n$, respectively. 
Since
\begin{equation*}
\lim_{r\to 0^{+}} n(r) = - 48 \ee \sigma^{12}<0,
\qquad \lim_{r\to +\infty} n(r) = - \infty,
\end{equation*}
a necessary condition to have a zero of $ W'(r;L)$ is $n(\bar{r}(L))>0$, that is
\begin{equation}\label{LJ-nec}
L^{2} < \biggl{(}\dfrac{1}{5}\biggr{)}^{\!\frac{2}{3}} \dfrac{72}{5} \ee \sigma^{2}.
\end{equation}
If \eqref{LJ-nec} holds, then $n$ has exactly two zeros $r_{1}(L)$ and $r_{2}(L)$ with $0<r_{1}(L)<r_{2}(L)$. 
From this analysis, we deduce that
\begin{equation*}
W'(\cdot;L)<0 \quad \text{in $(0,r_{1}(L))\cup(r_{2}(L),+\infty)$,} 
\qquad
W'(\cdot;L)>0 \quad \text{in $(r_{1}(L),r_{2}(L))$.} 
\end{equation*}
Hence, condition $(h_{W})$ holds for $r_{0}(L)=r_{1}(L)$. See the representation of the graph of $W(\cdot;L)$ in Figure~\ref{fig-03}.

\begin{figure}[htb]
\begin{tikzpicture}
\begin{axis}[
  scaled ticks=false, 
  tick label style={font=\scriptsize},
  axis y line=left, axis x line=middle,
  xtick={0},
  ytick={-0.033,0},
  xticklabels={},
  yticklabels={$-w_{0}(L)$,$0$},
  xlabel={\small $r$},
  ylabel={\small $W(r;L)$},
every axis x label/.style={
    at={(ticklabel* cs:1.0)},
    anchor=west,
},
every axis y label/.style={
    at={(ticklabel* cs:1.0)},
    anchor=south,
},
  width=6cm,
  height=6cm,
  xmin=0,
  xmax=10,
  ymin=-0.04,
  ymax=0.07]
\addplot [color=black,line width=0.9pt,smooth] coordinates {(0.793, 0.0987088) (0.8425, 0.0369734) (0.892, 0.00106148) (0.9415, -0.0188788) (0.991, -0.0289165) (1.0405, -0.0328424) (1.09, -0.033048) (1.1395, -0.0310527) (1.189, -0.0278241) (1.2385, -0.0239774) (1.288, -0.0199005) (1.3375, -0.015834) (1.387, -0.011923) (1.4365, -0.00824977) (1.486, -0.00485672) (1.5355, -0.00176018) (1.585, 0.00103981) (1.6345, 0.00355313) (1.684, 0.00579552) (1.7335, 0.0077857) (1.783, 0.00954371) (1.8325, 0.0110897) (1.882, 0.0124431) (1.9315, 0.0136227) (1.981, 0.0146456) (2.0305, 0.0155278) (2.08, 0.016284) (2.1295, 0.0169275) (2.179, 0.0174704) (2.2285, 0.0179235) (2.278, 0.0182966) (2.3275, 0.0185985) (2.377, 0.0188371) (2.4265, 0.0190195) (2.476, 0.019152) (2.5255, 0.0192403) (2.575, 0.0192894) (2.6245, 0.0193039) (2.674, 0.0192878) (2.7235, 0.0192448) (2.773, 0.019178) (2.8225, 0.0190904) (2.872, 0.0189846) (2.9215, 0.0188628) (2.971, 0.0187271) (3.0205, 0.0185795) (3.07, 0.0184214) (3.1195, 0.0182544) (3.169, 0.0180799) (3.2185, 0.0178989) (3.268, 0.0177126) (3.3175, 0.0175218) (3.367, 0.0173275) (3.4165, 0.0171303) (3.466, 0.0169309) (3.5155, 0.0167299) (3.565, 0.0165279) (3.6145, 0.0163253) (3.664, 0.0161225) (3.7135, 0.0159199) (3.763, 0.0157179) (3.8125, 0.0155166) (3.862, 0.0153164) (3.9115, 0.0151175) (3.961, 0.01492) (4.0105, 0.0147242) (4.06, 0.0145302) (4.1095, 0.0143382) (4.159, 0.0141482) (4.2085, 0.0139603) (4.258, 0.0137746) (4.3075, 0.0135912) (4.357, 0.0134101) (4.4065, 0.0132314) (4.456, 0.0130551) (4.5055, 0.0128813) (4.555, 0.0127099) (4.6045, 0.012541) (4.654, 0.0123746) (4.7035, 0.0122106) (4.753, 0.0120491) (4.8025, 0.0118901) (4.852, 0.0117336) (4.9015, 0.0115795) (4.951, 0.0114278) (5.0005, 0.0112785) (5.05, 0.0111316) (5.0995, 0.0109871) (5.149, 0.0108449) (5.1985, 0.010705) (5.248, 0.0105674) (5.2975, 0.010432) (5.347, 0.0102988) (5.3965, 0.0101679) (5.446, 0.0100391) (5.4955, 0.00991238) (5.545, 0.00978778) (5.5945, 0.00966523) (5.644, 0.00954471) (5.6935, 0.00942618) (5.743, 0.00930961) (5.7925, 0.00919496) (5.842, 0.0090822) (5.8915, 0.0089713) (5.941, 0.00886223) (5.9905, 0.00875495) (6.04, 0.00864944) (6.0895, 0.00854565) (6.139, 0.00844357) (6.1885, 0.00834315) (6.238, 0.00824437) (6.2875, 0.0081472) (6.337, 0.00805161) (6.3865, 0.00795757) (6.436, 0.00786505) (6.4855, 0.00777402) (6.535, 0.00768445) (6.5845, 0.00759632) (6.634, 0.0075096) (6.6835, 0.00742426) (6.733, 0.00734028) (6.7825, 0.00725763) (6.832, 0.00717628) (6.8815, 0.00709622) (6.931, 0.00701742) (6.9805, 0.00693985) (7.03, 0.00686349) (7.0795, 0.00678831) (7.129, 0.0067143) (7.1785, 0.00664144) (7.228, 0.0065697) (7.2775, 0.00649905) (7.327, 0.00642949) (7.3765, 0.00636099) (7.426, 0.00629353) (7.4755, 0.00622708) (7.525, 0.00616164) (7.5745, 0.00609718) (7.624, 0.00603369) (7.6735, 0.00597114) (7.723, 0.00590952) (7.7725, 0.00584881) (7.822, 0.005789) (7.8715, 0.00573006) (7.921, 0.00567199) (7.9705, 0.00561476) (8.02, 0.00555837) (8.0695, 0.00550279) (8.119, 0.00544801) (8.1685, 0.00539402) (8.218, 0.0053408) (8.2675, 0.00528833) (8.317, 0.00523662) (8.3665, 0.00518563) (8.416, 0.00513536) (8.4655, 0.0050858) (8.515, 0.00503693) (8.5645, 0.00498874) (8.614, 0.00494121) (8.6635, 0.00489435) (8.713, 0.00484813) (8.7625, 0.00480254) (8.812, 0.00475758) (8.8615, 0.00471322) (8.911, 0.00466947) (8.9605, 0.00462631) (9.01, 0.00458373) (9.0595, 0.00454172) (9.109, 0.00450027) (9.1585, 0.00445937) (9.208, 0.00441902) (9.2575, 0.00437919) (9.307, 0.00433989) (9.3565, 0.00430111) (9.406, 0.00426283) (9.4555, 0.00422504) (9.505, 0.00418775) (9.5545, 0.00415094) (9.604, 0.0041146) (9.6535, 0.00407872) (9.703, 0.00404331) (9.7525, 0.00400834) (9.802, 0.00397382) (9.8515, 0.00393973) (9.901, 0.00390607) (9.9505, 0.00387283) (10., 0.00384)};
\draw [color=gray, dashed, line width=0.3pt] (axis cs: 1.0678,0)--(axis cs:1.0678,-0.033);
\draw [color=gray, dashed, line width=0.3pt] (axis cs: 0,-0.0333)--(axis cs:1.0678,-0.033);
\draw [color=gray, dashed, line width=0.3pt] (axis cs: 2.622,-0.002)--(axis cs:2.622,0.0193);
\draw [color=gray, dashed, line width=0.3pt] (axis cs: 0,0.0193)--(axis cs:2.622,0.0193);
\node at (axis cs: 7,0.04) {\scriptsize{$r_{1}(L)$}};
\node at (axis cs: 2.622,-0.008) {\scriptsize{$r_{2}(L)$}};
\draw [color=gray, line width=0.3pt] (axis cs: 1.0678,0)--(axis cs:6,0.04);
\fill (axis cs: 1.0678,0) circle (0.8pt);
\end{axis}
\end{tikzpicture}
\captionof{figure}{Qualitative graph of $W(\cdot;L)$ for the Lennard-Jones potential with angular momentum $L>0$ satisfying \eqref{LJ-nec}.}
\label{fig-03}
\end{figure}
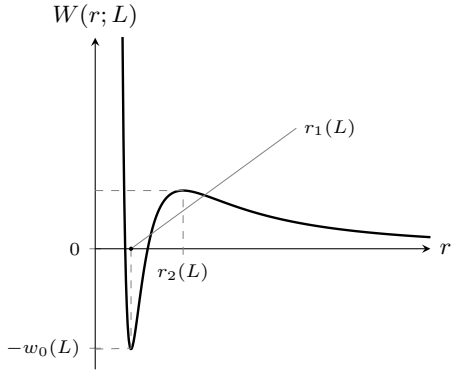

Since $r_{0}(L)$ solves $n(r)=0$, we have
\begin{equation}\label{eq-r0}
L^{2} r_{0}(L)^{10} = 24 \ee \sigma^{6} r_{0}(L)^{6} - 48 \ee \sigma^{12}
\end{equation}
and thus
\begin{equation*}
w_{0}(L) = -W(r_{0}(L);L) = 
- \dfrac{4\ee\sigma^{6}}{r_{0}(L)^{2}} 
\bigl{(} 2 r_{0}(L)^{6} - 5 \sigma^{6} \bigr{)}.
\end{equation*}

We consider the functions $T$ and $\Theta$ defined in the set
\begin{equation*}
\Lambda = \biggl{\{}(H,L)\in\mathbb{R}^{2} \colon L>0, \; L^{2} < \biggl{(}\dfrac{1}{5}\biggr{)}^{\!\frac{2}{3}} \dfrac{72}{5} \ee \sigma^{2}, \; H\in (-w_{0}(L),W(r_{2}(L))) \biggr{\}}.
\end{equation*}
A global study of the sign of $\mathcal{D}$ on the set $\Lambda$ seems to be a difficult task. Therefore, from now on our analysis will be local, in the sense that we are going to consider values of $H$ slightly above $-w_{0}(L)$.
More precisely, taking advantage of the results in Appendix~\ref{appendix-A}, we compute the limits as $H\to(-w_{0}(L))^{+}$ of $\partial_{H} T$, $\partial_{L} T$, $\partial_{H} \Theta$, $\partial_{L} \Theta$, in order to show that
\begin{equation}\label{eq-LJ-D_HL}
\lim_{H\to(-w_{0}(L))^{+}} \mathcal{D}(H,L) > 0,
\end{equation}
for every admissible $L$ sufficiently small.
Thus, for every admissible $L$ small enough, there exists a suitable $\tilde{H}(L)>-w_{0}(L)$ such that $\mathcal{D}(H,L)>0$ for every $H\in(-w_{0}(L), \tilde{H}(L))$. Hence, $\mathcal{D}(H,L)\neq0$ at least on a non-empty open subset of $\Lambda$.

First, we deal with the function $T$. To enter the framework of Appendix~\ref{appendix-A}, we set
$\mathcal{V}(r)=V(r)$ and $k=-1$. Hence,
\begin{equation*}
\Omega(r;L) = W(r;L) + w_{0}(L) = \dfrac{1}{2r^{12}} \bigl{(} L^{2} r^{10} - 8 \ee \sigma^{6} r^{6} + 8 \ee \sigma^{12} + 2 w_{0}(L) r^{12} \bigr{)}.
\end{equation*}
The higher order derivatives are the following
\begin{align*}
&W''(r;L) = \dfrac{3}{r^{14}} \bigl{(} L^{2} r^{10} - 56 \ee \sigma^{6} r^{6} + 208 \ee \sigma^{12} \bigr{)},
\\
&W'''(r;L) = -\dfrac{12}{r^{15}} \bigl{(} L^{2} r^{10} - 112 \ee \sigma^{6} r^{6} + 728 \ee \sigma^{12} \bigr{)},
\\
&W^{(4)}(r;L) = \dfrac{12}{r^{16}} \bigl{(} 5 L^{2} r^{10} - 1008 \ee \sigma^{6} r^{6} + 10920 \ee \sigma^{12} \bigr{)}.
\end{align*}
Then, recalling \eqref{eq-r0} and the definition of $\Omega_{0}^{(i)}(L)$ given in \eqref{eq-deromegazero}, we find
\begin{align*}
\Omega_{0}^{(2)}(L) 
&= \dfrac{3}{r_{0}(L)^{14}} \bigl{(} (24 - 56)  \ee \sigma^{6} r_{0}(L)^{6} + (-48+208) \ee \sigma^{12} \bigr{)} 
\\
&= \dfrac{2^{5}\cdot 3 \, \ee \sigma^{6}}{r_{0}(L)^{14}} \bigl{(} - r_{0}(L)^{6} + 5 \sigma^{6} \bigr{)},
\\
\Omega_{0}^{(3)} (L) 
&= \dfrac{2^{5}\cdot 3 \, \ee \sigma^{6}}{r_{0}(L)^{15}} \bigl{(} 11 r_{0}(L)^{6} - 85 \sigma^{6} \bigr{)},
\\
\Omega_{0}^{(4)}(L) 
&= \dfrac{2^{5}\cdot 3^{2} \, \ee \sigma^{6}}{r_{0}(L)^{16}} \bigl{(} -37 r_{0}(L)^{6} + 445 \sigma^{6} \bigr{)}.
\end{align*}
Using the fact that $n'(r_{0}(L))>0$ together with \eqref{eq-r0}, we easily deduce that
\begin{equation}\label{positive-omega-2}
\Omega_{0}^{(2)}(L)>0.
\end{equation}
Moreover, we compute
\begin{equation*}
\biggl{(} \dfrac{\mathrm{d}}{\mathrm{d}L} \Omega^{(2)}_{0} \biggr{)}(L) = - \dfrac{2^{5}\cdot 3 \, \ee \sigma^{6}}{r_{0}(L)^{15}} r_{0}'(L) \Bigl{(} 14 \bigl{(} - r_{0}(L)^{6} + 5 \sigma^{6} \bigr{)} + 6 r_{0}(L)^{6} \Bigr{)} < 0,
\end{equation*}
where the last inequality follows from \eqref{positive-omega-2} and the fact that $r_{0}'(L)>0$ by \eqref{eq-dim11} (with $k=-1$).
Now, omitting the dependence on $L$, according to definition \eqref{def-sigma-zero} of $\Sigma_{0}$, we thus have
\begin{align*}
\Sigma_{0} &= 5 (\Omega_{0}^{(3)})^{2} - 3 \Omega_{0}^{(2)} \Omega_{0}^{(4)} 
\\
&= \dfrac{2^{10}\cdot 3^{2} \, \ee^{2} \sigma^{12}}{r_{0}^{30}} 
\Bigl{[} 5 \bigl{(} 11 r_{0}^{6} - 85 \sigma^{6} \bigr{)}^{2} 
- 3^{2} \bigl{(} - r_{0}^{6} + 5 \sigma^{6} \bigr{)} \bigl{(} -37 r_{0}^{6} + 445 \sigma^{6} \bigr{)} \Bigr{]}
\\
&= \dfrac{2^{12}\cdot 3^{2} \, \ee^{2} \sigma^{12}}{r_{0}^{30}} \bigr{(} 68 r_{0}^{12} - 920 \sigma^{6} r_{0}^{6} + 4025 \sigma^{12} \bigl{)} >0.
\end{align*}
Then, by Proposition~\ref{prop-stime-tempo}, for every admissible $L$ it holds that 
\begin{equation}\label{lim-LJ-1}
\lim_{H\to(-w_{0}(L))^{+}} \partial_{H} T(H,L) >0.
\end{equation}
Moreover, we have
\begin{align*}
&r_{0}^{2} \Sigma_{0}+24 \Omega_{0}^{(2)}\Omega_{0}^{(3)} r_{0}+72 (\Omega_{0}^{(2)})^{2}=
\\
&= \dfrac{2^{12}\cdot 3^{2} \, \ee^{2} \sigma^{12}}{r_{0}^{28}} 
\Bigl{[} \bigl{(} 68 r_{0}^{12} - 920 \sigma^{6} r_{0}^{6} + 4025 \sigma^{12} \bigr{)}
+ 6 \bigl{(} - r_{0}^{6} + 5 \sigma^{6} \bigr{)}\bigl{(} 11 r_{0}^{6} - 85 \sigma^{6} \bigr{)}\\
& \hspace{75pt} + 18 \bigl{(} - r_{0}^{6} + 5 \sigma^{6} \bigr{)}^{2}
\Bigl{]}
\\
&= \dfrac{2^{12}\cdot 3^{2}\cdot 5 \, \ee^{2} \sigma^{12}}{r_{0}^{28}} 
\bigl{(} 4 r_{0}^{12} - 52 \sigma^{6} r_{0}^{6} + 385 \sigma^{12}\bigr{)}>0.
\end{align*}
Then, by Proposition~\ref{prop-stime-tempo} (with $k=-1$ and $s_{0}(L)=r_{0}(L)$), for every admissible $L$ it holds that
\begin{equation}\label{lim-LJ-2}
\lim_{H\to(-w_{0}(L))^{+}} \partial_{L} T(H,L) <0.
\end{equation}

Second, we deal with the function $\Theta$. To enter the framework of Appendix~\ref{appendix-A}, we set
$\mathcal{V}(u)=V(1/u)$ and $k=1$. Hence,
\begin{equation*}
\Omega(u;L) = W\left(\dfrac{1}{u};L\right) + W\left(\dfrac{1}{r_{0}(L)};L\right) = \dfrac{1}{2} L^{2} u^{2} - 4 \ee \sigma^{6} u^{6} + 4 \ee \sigma^{12} u^{12} + W\left(\dfrac{1}{r_{0}(L)};L\right).
\end{equation*}
Notice that the minimum point of $\Omega(\cdot;L)$ is $u_{0}(L)=1/ r_{0}(L)$, with $r_{0}(L)$ defined as above, so that hypothesis $(h_{\mathcal{W}})$ is satisfied.
The derivatives of $\Omega$ are the following
\begin{align*}
&\Omega'(u;L) = L^{2} u - 24 \ee \sigma^{6} u^{5}+48 \ee \sigma^{12} u^{11},
\\
&\Omega''(u;L) = L^{2} - 120 \ee \sigma^{6} u^{4}+528 \ee \sigma^{12} u^{10},
\\
&\Omega'''(u;L) = - 480 \ee \sigma^{6} u^{3}+5280 \ee \sigma^{12} u^{9},
\\
&\Omega^{(4)}(u;L) = - 1440 \ee \sigma^{6} u^{2} + 47520 \ee \sigma^{12} u^{8}.
\end{align*}
From \eqref{eq-r0}, we deduce that
\begin{equation}\label{eq-u0}
L^{2} = 24 \ee \sigma^{6} (u_{0}(L))^{4} - 48 \ee \sigma^{12} (u_{0}(L))^{10},
\end{equation}
which corresponds to the fact that $\Omega'(u_{0}(L);L)=0$.
Then, using \eqref{eq-u0}, we find
\begin{align*}
&\Omega_{0}^{(2)}(L)
= 2^{5}\cdot 3 \ee \sigma^{6} u_{0}(L)^{4} \bigl{(} 5 \sigma^{6} u_{0}(L)^{6} -1 \bigr{)},
\\
&\Omega_{0}^{(3)}(L)
= 2^{5}\cdot 3\cdot 5 \ee \sigma^{6} u_{0}(L)^{3} \bigl{(} 11 \sigma^{6} u_{0}(L)^{6} -1 \bigr{)},
\\
&\Omega_{0}^{(4)}(L) 
= 2^{5}\cdot 3^{2} \cdot 5 \ee \sigma^{6} u_{0}(L)^{2} \bigl{(} 33 \sigma^{6} u_{0}(L)^{6} -1 \bigr{)}.
\end{align*}
As above, one can easily check that
\begin{equation*}
\Omega_{0}^{(2)}(L)>0
\end{equation*}
and that
\begin{equation*}
\biggl{(} \dfrac{\mathrm{d}}{\mathrm{d}L} \Omega^{(2)}_{0} \biggr{)}(L) = 2^{5}\cdot 3 \ee \sigma^{6} u_{0}'(L) \bigl{(} 4u_{0}(L)^{3} \bigl{(} 5 \sigma^{6} u_{0}(L)^{6} -1 \bigr{)} + 30 \sigma^{6} u_{0}(L)^{9} \bigr{)} <0.
\end{equation*}
Omitting again the dependence on $L$, we have
\begin{align*}
\Sigma_{0} &= 5 (\Omega_{0}^{(3)})^{2} - 3 \Omega_{0}^{(2)} \Omega_{0}^{(4)} 
\\
&= 2^{10} \cdot 3^{2} \cdot 5 \ee^{2} \sigma^{12} u_{0}^{6} \Bigl{[} 
5^{2} \bigl{(} 11 \sigma^{6} u_{0}^{6} -1 \bigr{)}^{2}
- 3^{2} \bigl{(} 5 \sigma^{6} u_{0}^{6} -1 \bigr{)}  \bigl{(} 33 \sigma^{6} u_{0}^{6} -1 \bigr{)}
\Bigr{]}
\\
&= 2^{12} \cdot 3^{2} \cdot 5 \ee^{2} \sigma^{12} u_{0}^{6} 
\bigl{(} 385 \sigma^{12} u_{0}^{12} - 52 \sigma^{6} u_{0}^{6} + 4 \bigr{)}>0.
\end{align*}
Then, by Proposition~\ref{prop-stime-tempo} (with $k=1$ and $s_{0}(L)=u_{0}(L)$), for every admissible $L$ it holds that
\begin{equation*}
\lim_{H\to(-w_{0}(L))^{+}} \partial_{H} P(H,L) >0.
\end{equation*}
Therefore, recalling \eqref{def-P}, we have that
\begin{equation}\label{lim-LJ-3}
\lim_{H\to(-w_{0}(L))^{+}} \partial_{H} \Theta(H,L) >0.
\end{equation}
Next, in order to analyze the sign of  $\partial_{L} \Theta(H,L)$, first we compute
\begin{align*}
&u_{0}^{2} \Sigma_{0} - 24 \Omega_{0}^{(2)}\Omega_{0}^{(3)} u_{0} + 24 (\Omega_{0}^{(2)} )^{2}=
\\
&= 2^{12} \cdot 3^{2} \ee^{2} \sigma^{12} u_{0}^{8} \Bigl{[} 
5 \bigl{(} 385 \sigma^{12} u_{0}^{12} - 52 \sigma^{6} u_{0}^{6} + 4 \bigr{)}
\\
&\hspace{80pt}
- 30 \bigl{(} 5 \sigma^{6} u_{0}^{6} -1 \bigr{)}\bigl{(} 11 \sigma^{6} u_{0}^{6} -1 \bigr{)}
+ 6 \bigl{(} 5 \sigma^{6} u_{0}^{6} -1 \bigr{)}^{2}
\Bigr{]}
\\
&= 2^{12} \cdot 3^{2} \ee^{2} \sigma^{12} u_{0}^{8} \bigl{(} 425 \sigma^{12} u_{0}^{12} +160 \sigma^{6} u_{0}^{6} -4 \bigr{)}.
\end{align*}
Then, recalling again \eqref{def-P} and \eqref{eq-AAAL}, by Proposition~\ref{prop-stime-tempo} (with $k=1$ and $s_{0}(L)=u_{0}(L)$), for every admissible $L$ it holds that
\begin{align*}
&\lim_{H\to(-w_{0}(L))^{+}} \partial_{L} \Theta(H,L) =
\\
&= \lim_{H\to(-w_{0}(L))^{+}} P(H,L) + \lim_{H\to(-w_{0}(L))^{+}} L \partial_{L} P(H,L)
\\
&= \dfrac{2\pi}{(\Omega_{0}^{(2)})^{\frac{1}{2}}}
- \dfrac{\pi L^{2}}{12 (\Omega_{0}^{(2)})^{\frac{7}{2}}} \cdot 2^{12} \cdot 3^{2} \ee^{2} \sigma^{12} u_{0}^{8} \bigl{(} 425 \sigma^{12} u_{0}^{12} +160 \sigma^{6} u_{0}^{6} -4 \bigr{)}
\\
&= \dfrac{2\pi}{(\Omega_{0}^{(2)})^{\frac{7}{2}}} \Bigl{[} (\Omega_{0}^{(2)})^{3} - L^{2} \cdot 2^{9} \cdot 3 \ee^{2} \sigma^{12} u_{0}^{8} \bigl{(} 425 \sigma^{12} u_{0}^{12} +160 \sigma^{6} u_{0}^{6} -4 \bigr{)}\Bigr{]}.
\end{align*}
Recalling \eqref{eq-u0}, we compute the term into the square bracket
\begin{align*}
&2^{15}\cdot 3^{3} \ee^{3} \sigma^{18} u_{0}^{12} \bigl{(} 5 \sigma^{6} u_{0}^{6} -1 \bigr{)}^{3} 
\\
&\qquad - 2^{3}\cdot 3 \ee \sigma^{6} u_{0}^{4} \bigl{(} 1-2 \sigma^{6} u_{0}^{6} \bigr{)} \cdot 2^{9} \cdot 3 \ee^{2} \sigma^{12} u_{0}^{8} \bigl{(} 425 \sigma^{12} u_{0}^{12} +160 \sigma^{6} u_{0}^{6} -4 \bigr{)} =
\\
&= 2^{12}\cdot 3^{2} \ee^{3} \sigma^{18} u_{0}^{12} \Bigl{[}
3850 \sigma^{18} u_{0}^{18} - 1905 \sigma^{12} u_{0}^{12} + 192 \sigma^{6} u_{0}^{6} - 20 \Bigr{]}.
\end{align*}
Regarding the expression inside the brackets as a third order polynomial function of $\sigma^{6} u_{0}^{6}(L)$, one can prove that it change sign once in $(0,+\infty)$, say $k_{0}$ this point, being negative in $(0,k_{0})$ and positive in $(k_{0},+\infty)$.
Since $L\mapsto u_{0}(L)$ is a strictly decreasing function (by differentiating \eqref{eq-u0} and using the fact that $\Omega_{0}^{(2)}(L)>0$) and $\lim_{L\to 0^{+}} u_{0}(L) > k_{0}$, we conclude that there exists $L_{0}>0$ such that
\begin{equation}\label{lim-LJ-4}
\lim_{H\to(-w_{0}(L))^{+}} \partial_{L} \Theta(H,L) \geq 0,
\end{equation}
for every admissible $L$ with $L\leq L_{0}$ (and strictly positive otherwise).

Summing up, from \eqref{lim-LJ-1}, \eqref{lim-LJ-2}, \eqref{lim-LJ-3}, \eqref{lim-LJ-4} we have that \eqref{eq-LJ-D_HL} holds true.

\section{An application to a restricted $3$-body problem with non-Newtonian interaction}\label{section-5}

In this section, we investigate the existence of periodic solutions for a restricted planar circular $3$-body problem, with the interaction between the primaries driven by a $-\alpha$-homogeneous potential with $\alpha\in(0,2)$. Our strategy consists in interpreting this restricted problem as a perturbation, when the mass of one of the primaries tends to zero, of a $-\alpha$-homogeneous Kepler problem, so as to make an application of Theorem~\ref{th-main} possible, as discussed in Section~\ref{section-4.2}. Notice that this forces us to assume $\alpha \neq 1$, thus excluding the case of a Newtonian attraction. For the existence of periodic solutions in the Newtonian case see \cite{AnLi-18,CoPiSo-05,OrZh-21,PYFN-06} and the references therein.

This section is organized as follows. In Section~\ref{section-5.1}, we first prove some technical results about the functions $T$ and $\Theta$ for the potential \eqref{hom-pot}. Based on this, in Section~\ref{section-5.2}, we present a complete analysis of the periodic solutions of the $-\alpha$-homogeneous Kepler problem, finding a necessary and sufficient condition for the existence of $\tau$-periodic solutions of type $(n,k)$. In Section~\ref{section-5.3}, we then give the details of the application to the restricted planar circular $3$-body problem.

\subsection{Some results about the functions $T$ and $\Theta$ for the $-\alpha$-homogeneous Kepler problem}\label{section-5.1}

Throughout this section, let $V$ be as in \eqref{hom-pot}, with $\alpha\in(0,2)$. For future convenience, we observe that the minimum point $r_{0}(L)$ defined in \eqref{def-r0L-hom} is non-degenerate, indeed by computing
\begin{equation}\label{def-WLsecond-hom}
W''(r;L) = \dfrac{1}{r^{4}} \bigl{(} 3L^{2}-(\alpha+1)\kappa r^{2-\alpha} \bigr{)},
\end{equation}
it follows that
\begin{equation}\label{min-non-deg}
W''(r_{0}(L);L) = \dfrac{L^{2} (2-\alpha)}{r_{0}(L)^{4}}  >0.
\end{equation}
We consider the functions $T$ and $\Theta$ defined for $(H,L)\in\Lambda$, with $\Lambda$ as in \eqref{def-lambda}.

The first result ensures the strict monotonicity with respect to $H$ of the function $T$ defined in \eqref{def-T-HL}.

\begin{lemma}\label{lem-dH-T}
Let $\alpha\in(0,2)$. For every $(H,L) \in \Lambda$, it holds that $\partial_{H} T(H,L) > 0$.
\end{lemma}

\begin{proof}
We are going to exploit the Schaaf's monotonicity criterion (cf.~\cite[Theorem~1]{Sc-85} and Remark~\ref{rem-Schaaf}), thus, in addition to \eqref{min-non-deg}, we need to prove that
\begin{itemize}
\item[$(i)$] $5(W'''(r;L))^{2} - 3 W''(r;L) W^{(4)}(r;L) >0$ for all $r>r_{0}(L)$,
\item[$(ii)$] $W'(r_{0}(L);L)  W'''(r_{0}(L);L) <0$.
\end{itemize}

Preliminarily, recalling the definition of $W(\cdot;L)$ given in \eqref{def-WL-hom} and the expression of $W'(\cdot;L)$ in \eqref{def-WLprime-hom} and $W'(\cdot;L)$ in \eqref{def-WLsecond-hom}, we compute the higher order derivatives of $W(\cdot;L)$:
\begin{align*}
&W'''(r;L) = \dfrac{1}{r^{5}} \bigl{(} -12L^{2}+(\alpha+1)(\alpha+2)\kappa r^{2-\alpha} \bigr{)},
\\
&W^{(4)}(r;L) = \dfrac{1}{r^{6}} \bigl{(} 60L^{2}+(\alpha+1)(\alpha+2)(\alpha+3)\kappa r^{2-\alpha} \bigr{)}.
\end{align*}
Then, we get
\begin{align*}
&5(W'''(r;L))^{2} - 3 W''(r;L) W^{(4)}(r;L) =
\\
&= \dfrac{5}{r^{10}} \bigl{(} -12L^{2}+(\alpha+1)(\alpha+2)\kappa r^{2-\alpha} \bigr{)}^{2} 
\\
& \quad - \dfrac{5}{r^{10}} \bigl{(} 3L^{2}-(\alpha+1)\kappa r^{2-\alpha} \bigr{)} \bigl{(} 60L^{2}+(\alpha+1)(\alpha+2)(\alpha+3)\kappa r^{2-\alpha} \bigr{)}
\\
&= \dfrac{1}{r^{10}} \Bigl{[} 180 L^{4} + 3(\alpha+1)(3\alpha^{2}-25 \alpha - 2) L^{2} \kappa r^{2-\alpha}) 
\\
&\hspace{35pt} +(\alpha+1)^{2}(\alpha+2)(2\alpha+1) \kappa^{2} r^{2(2-\alpha)} \Bigr{]}.
\end{align*}
Regarding the expression inside the brackets as a second order polynomial function of $r^{2-\alpha}$, its discriminant is given by
\begin{equation*}
\Delta = 27 (\alpha+1)^{2}(2-\alpha)^{2}(\alpha-13)(3\alpha+1) \kappa^{2} L^{4}.
\end{equation*}
Hence, $\Delta<0$, and thus $(i)$ is satisfied.

Recalling the expression of $r_{0}(L)$ given in \eqref{def-r0L-hom}, we can compute
\begin{align*}
&W'(r_{0}(L);L) W'''(r_{0}(L);L) =
\\
&= \dfrac{1}{(r_{0}(L))^{8}} \biggl{(} -L^{2} + \dfrac{3L^{2}}{\alpha+1} \biggl{)} \biggl{(} -12L^{2} + (\alpha+1)(\alpha+2) \dfrac{3L^{2}}{\alpha+1} \biggl{)}
\\
&= -\dfrac{1}{(r_{0}(L))^{8}} \dfrac{3L^{4}}{\alpha+1}(2-\alpha)^{2} <0.
\end{align*}
Therefore, condition $(ii)$ is satisfied. The proof is complete.
\end{proof}

The next lemma provides, for a fixed $L>0$, the behavior of $T(H,L)$ and $\Theta(H,L)$ when $H\to 0^{-}$. 

\begin{lemma}\label{lem-5.2}
For every $L>0$, it holds that
\begin{align}
&\lim_{H \to 0^{-}} T(H,L) = +\infty,
\label{lim-T0}
\\
&\lim_{H \to 0^{-}} \Theta(H,L) = \dfrac{2\pi}{2-\alpha}.
\label{lim-Theta0}
\end{align}
\end{lemma}

\begin{proof}
Preliminarily, we notice that, by \eqref{eq-T_HL-riscal} and \eqref{eq-P_HL-riscal}, we can reduce to the case $L=1$.
We start by proving \eqref{lim-T0}. Recalling the definition \eqref{def-T-HL}, we have
\begin{equation*}
T(H,1) = \sqrt{2} \int_{r_{-}(H,1)}^{r_{+}(H,1)} \dfrac{\mathrm{d}r}{\sqrt{H-\frac{1}{2 r^{2}} + \frac{\kappa}{\alpha r^{\alpha}}}},
\end{equation*}
where $r_{\pm}(H,1)$ are the two solutions of the equation $H-\frac{1}{2 r^{2}} + \frac{\kappa}{\alpha r^{\alpha}}= 0$ with $r_{-}(H,1)<r_{+}(H,1)$.
Passing to the limit as $H\to 0^{-}$ in the equation, we obtain that
\begin{equation*}
\lim_{H\to 0^{-}} r_{-}(H,1) = \hat{r}_{-} = \biggl{(} \dfrac{\alpha}{2\kappa} \biggr{)}^{\!\frac{1}{2-\alpha}},
\qquad
\lim_{H\to 0^{-}} r_{+}(H,1) = +\infty.
\end{equation*}
Therefore, by Fatou's lemma, we have
\begin{equation*}
+\infty = \sqrt{2} \int_{\hat{r}_{-}}^{+\infty} \dfrac{\mathrm{d}r}{\sqrt{-\frac{1}{2 r^{2}} + \frac{\kappa}{\alpha r^{\alpha}}}} \leq
\liminf_{H\to 0^{-}} T(H,1).
\end{equation*}
From this, the conclusion follows.

Now we deal with \eqref{lim-Theta0}. At first, we recall that $\Theta \equiv 2\pi$ for $\alpha = 1$, so that the limit relation clearly holds. Thus, we assume $\alpha \neq 1$ henceforth. Recalling the definitions \eqref{def-Theta-HL} and \eqref{def-P-HL}, we have
\begin{equation*}
\Theta(H,1) = P(H,1) 
= \sqrt{2} \int_{u_{-}(H)}^{u_{+}(H)} \dfrac{\mathrm{d}u}{\sqrt{H-\frac{1}{2} u^{2} + \frac{\kappa}{\alpha} u^{\alpha}}},
\end{equation*}
where $u_{\pm}(H)$ are the two solutions of the equation $H-\frac{1}{2} u^{2} + \frac{\kappa}{\alpha} u^{\alpha}= 0$ with $u_{-}(H)<u_{+}(H)$.
Passing to the limit as $H\to 0^{-}$ in the equation, we obtain that
\begin{equation*}
\lim_{H\to 0^{-}} u_{-}(H) = 0,
\qquad
\lim_{H\to 0^{-}} u_{+}(H) = \dfrac{1}{\hat{r}_{-}}.
\end{equation*}
Therefore, passing to the limit in the integral (see the next paragraph for the details), we obtain
\begin{equation}\label{eq-limfin}
\lim_{H\to 0^{+}} \Theta(H,1) = \sqrt{2} \int_{0}^{\frac{1}{\hat{r}_{-}}} \dfrac{\mathrm{d}u}{\sqrt{-\frac{1}{2} u^{2} + \frac{\kappa}{\alpha}u^{\alpha}}}.
\end{equation}
Next, via the change of variable
\begin{equation*}
\xi = \sqrt{\dfrac{\alpha}{2 \kappa}} u^{\frac{2-\alpha}{2}}, 
\end{equation*}
we have
\begin{align*}
\sqrt{2} \int_{0}^{\frac{1}{\hat{r}_{-}}} \dfrac{\mathrm{d}u}{\sqrt{-\frac{1}{2} u^{2} + \frac{\kappa}{\alpha} u^{\alpha}}} = 
\dfrac{4}{2-\alpha} \int_{0}^{1} \dfrac{\mathrm{d}\xi}{\sqrt{1-\xi^{2}}} 
= \dfrac{4}{2-\alpha} \Bigl{[} \arcsin \xi \Bigr{]}_{0}^{1} = \dfrac{2\pi}{2-\alpha},
\end{align*}
from which the conclusion follows.

We now give the rigorous justification for passing to the limit under the integral sign, which is here a rather delicate issue since the direct use of standard theorems is not possible (cf. \cite[Section~1.6]{Sc-90}). To overcome this difficulty,
it is convenient to use the change of variable described in Appendix~\ref{appendix-A} in order to rewrite
$\Theta(H,1)$ as  
\begin{equation*}
\Theta(H,1) = \sqrt{2} \int_{-\frac{\pi}{2}}^{\frac{\pi}{2}} \dfrac{1}{h'(h^{-1}(\sqrt{H+w_{0}(1)} \sin \vartheta))}\,\mathrm{d}\vartheta,
\end{equation*}
where 
\begin{equation*}
h(u) = \textrm{sgn}\left(u-\frac{1}{r_{0}(1)}\right)\sqrt{\frac{1}{2}u^{2} - \frac{\kappa}{\alpha}u^\alpha + w_{0}(1)},
\end{equation*}
see Lemma~\ref{lem-chicone} (together with Remark~\ref{rem-cond-extra-dL}). Fixing $\delta > 0$ small, 
one has that $h^{-1}(\sqrt{H+w_{0}(1)} \sin \vartheta) \geq \gamma > 0$ for every $\vartheta \in [-\pi/2+\delta,\pi/2]$ and $H \in (-w_{0}(1),0]$, for some $\gamma>0$. Since $h'$ is bounded away from zero on every compact subinterval of
$(0,+\infty)$, by uniform convergence we infer that
\begin{equation}\label{eq-lim1}
\lim_{H \to 0^{-}}\int_{-\frac{\pi}{2}+\delta}^{\frac{\pi}{2}} \dfrac{1}{h'(h^{-1}(\sqrt{H+w_{0}(1)} \sin \vartheta))}\,\mathrm{d}\vartheta =
 \int_{-\frac{\pi}{2}+\delta}^{\frac{\pi}{2}} \dfrac{1}{h'(h^{-1}(\sqrt{w_{0}(1)} \sin \vartheta))}\,\mathrm{d}\vartheta.
\end{equation}
We claim that
\begin{equation}\label{eq-lim2}
\lim_{H \to 0^{-}}\int_{-\frac{\pi}{2}}^{-\frac{\pi}{2}+\delta} \dfrac{1}{h'(h^{-1}(\sqrt{H+w_{0}(1)} \sin \vartheta))}\,\mathrm{d}\vartheta =
 \int_{-\frac{\pi}{2}}^{-\frac{\pi}{2}+\delta} \dfrac{1}{h'(h^{-1}(\sqrt{w_{0}(1)} \sin \vartheta))}\,\mathrm{d}\vartheta,
\end{equation}
the difficulty now being that $h'(u)$ may vanish as $u \to 0^{+}$. Simple computations show that
\begin{equation*}
h''(u) = \frac{\kappa}{2 w_{0}(1)^{\frac{1}{4}}}(\alpha-1) u^{\alpha - 2} + o (u^{\alpha - 2}), \quad \text{as $u \to 0^{+}$.}
\end{equation*}
Thus, in a right neighborhood of $u = 0$ the function $1/h'$ is increasing if $\alpha \in (0,1)$ and decreasing if $\alpha \in (1,2)$.
As a consequence, for $\delta > 0$ small enough the integrand functions
\begin{equation*}
\dfrac{1}{h'(h^{-1}(\sqrt{H+w_{0}(1)} \sin \vartheta))}, \quad \vartheta \in \biggl{[}-\frac{\pi}{2}, -\frac{\pi}{2}+\delta\biggr{]},
\end{equation*}
are decreasing with respect to $H$ if $\alpha \in (0,1)$ and increasing if $\alpha \in (1,2)$. Passing to the limit under the integral sign is then justified using the Lebesgue theorem in the former case and the monotone convergence theorem in the latter one. We thus find that 
\eqref{eq-lim2} holds true. Recalling \eqref{eq-lim1}, we finally conclude  
\begin{equation*}
\lim_{H \to 0^{-}}\Theta(H,1) =
\sqrt{2}\int_{-\frac{\pi}{2}}^{\frac{\pi}{2}} \dfrac{1}{h'(h^{-1}(\sqrt{w_{0}(1)} \sin \vartheta))}\,\mathrm{d}\vartheta.
\end{equation*}
Undoing the change of variable described above, it is easily checked that the integral on the right hand-side of the above equality coincides with the integral in \eqref{eq-limfin}. The proof is thus complete. 
\end{proof}

\subsection{Periodic solutions to the $-\alpha$-homogeneous Kepler problem}\label{section-5.2}

Let us consider the system
\begin{equation}\label{eq-akep}
\ddot x = - \kappa \, \dfrac{x}{|x|^{\alpha+2}}, \quad x\in\mathbb{R}^{2} \setminus \{0\},
\end{equation}
where $\alpha\in(0,2)$.
Our existence result for periodic solutions of~\eqref{eq-akep} is the following (as in Section~\ref{section-2}, we limit ourselves to solutions with positive angular momentum).

\begin{theorem}\label{th-rosette}
Let $\alpha\in(0,2)$ and $\alpha \neq 1$.
For all $\tau>0$ and for all coprime positive integers $n,k$, there exists a periodic solution of \eqref{eq-akep} of type $(n,k)$ and minimal period $\tau$ if and only if
\begin{equation}\label{cond-kn}
\dfrac{k}{n} \in 
\begin{cases}
\left(\dfrac{1}{2-\alpha},\dfrac{1}{\sqrt{2-\alpha}}\right), & \text{if $\alpha \in (0,1)$,}
\vspace{3pt} \\
\left(\dfrac{1}{\sqrt{2-\alpha}},\dfrac{1}{2-\alpha}\right), & \text{if $\alpha \in (1,2)$.}
\end{cases}
\end{equation}
Moreover, the periodic solution is unique, up to time-translations and rotations.
\end{theorem}

\begin{proof}
According to the analysis of Section~\ref{section-2.1}, periodic solutions of type $(n,k)$ and minimal period $\tau$ exist if and only if there exists a pair $(H,L) \in \Lambda$, with $\Lambda$ as in \eqref{def-lambda}, such that
\begin{equation}\label{eq-unique}
T(H,L) = \dfrac{\tau}{n} \quad \text{ and } \quad 
\Theta(H,L) = 2\pi \frac{k}{n}.
\end{equation}
Moreover, the uniqueness of the periodic solution (up to time-translations and rotations) corresponds to the uniqueness of such a pair $(H,L)$.
Hence, we are going to study the unique solvability of \eqref{eq-unique}. From now on, $\tau>0$ and the coprime positive integers $n,k$ are fixed.

Since $T$ is strictly increasing with respect to $H$ by Lemma~\ref{lem-dH-T}, the range of the function $T(\cdot,L)$ is the open interval
\begin{equation*}
\left( \lim_{H\to (-w_{0}(L))^{+}} T(H,L), \lim_{H\to 0^{-}} T(H,L) \right) = \left(\dfrac{2\pi}{\sqrt{2-\alpha}} L^{\frac{2+\alpha}{2-\alpha}}, +\infty \right),
\end{equation*}
where the last equality comes from Proposition~\ref{prop-stime-tempo} (see also Remark~\ref{rem-cond-extra-dL}) and \eqref{lim-T0}.
Hence, letting 
\begin{equation*}
\hat{L} = \left( \dfrac{\tau \sqrt{2-\alpha}}{2\pi n} \right)^{\!\frac{2-\alpha}{2+\alpha}},
\end{equation*}
so that
\begin{equation*}
\lim_{H\to (-w_{0}(\hat{L}))^{+}} T(H,\hat{L}) = \dfrac{\tau}{n},
\end{equation*}
we conclude that the first equation of \eqref{eq-unique} is (uniquely) solvable as a function of $H$ if and only if $L\in(0,\hat{L})$.
Then, for all $L\in(0,\hat{L})$, let us define $\hat{H}(L)$ as the unique value such that
\begin{equation}\label{def-HnL}
T(\hat{H}(L),L) = \dfrac{\tau}{n}.
\end{equation}
Notice that the function $\hat{H}$ is of class $\mathcal{C}^{1}$ by the implicit function theorem. Moreover, by differentiating \eqref{def-HnL} with respect to $L$ and exploiting \eqref{eq-pdeT}, we obtain
\begin{equation}\label{eq-Hprime}
\begin{aligned}
\hat{H}'(L) 
&= - \dfrac{\partial_{L} T(\hat{H}(L),L)}{\partial_{H} T(\hat{H}(L),L)} 
\\
&= -\dfrac{2+\alpha}{2-\alpha} L^{-1}\dfrac{T(\hat{H}(L),L)}{\partial_{H} T(\hat{H}(L),L)} - \dfrac{2\alpha}{2-\alpha}\dfrac{\hat{H}(L)}{L}.
\end{aligned}
\end{equation}

For further convenience, we also observe that
\begin{equation}
\lim_{L\to 0^{+}} \hat{H}(L) L^{\frac{2+\alpha}{2-\alpha}} = 0,
\qquad
\lim_{L\to \hat{L}^{-}} \hat{H}(L) = - w_{0}(\hat{L}).
\label{lim-1}
\end{equation}
The second equality is obvious. As for the first one, if not, up to a subsequence, $T(\hat{H}(L) L^{\frac{2+\alpha}{2-\alpha}}, 1)$ is upper-bounded, and so
\begin{equation*}
T(\hat{H}(L),L) = L T(\hat{H}(L) L^{\frac{2+\alpha}{2-\alpha}}, 1) \to 0, \quad \text{as $L\to0^{+}$,}
\end{equation*}
in contrast with \eqref{def-HnL}.

Now, for $L\in (0, \hat{L})$, we define
\begin{equation*}
\hat{\vartheta}(L) = \Theta( \hat{H}(L) , L).
\end{equation*}
By exploiting \eqref{eq-pdeTheha} and \eqref{eq-Hprime}, we find
\begin{align*}
\hat{\vartheta}'(L)
&= \partial_{H} \Theta(\hat{H}(L),L) \hat{H}'(L) + \partial_{L} \Theta(\hat{H}(L),L)
\\
&= \partial_{H} \Theta(\hat{H}(L),L) \biggl{(} \hat{H}'(L) + \dfrac{2\alpha}{2-\alpha} \dfrac{\hat{H}(L)}{L} \biggr{)}
\\
&= - \dfrac{2+\alpha}{2-\alpha} L^{-1} \dfrac{T(\hat{H}(L),L)}{\partial_{H} T(\hat{H}(L),L)} \partial_{H} \Theta(\hat{H}(L),L).  
\end{align*}
Finally, recalling \eqref{sign-cast-rojas}, we have
\begin{equation*}
\hat{\vartheta}'(L) 
\begin{cases}
>0, &\text{if $\alpha\in(0,1)$,}
\\
<0, &\text{if $\alpha\in(1,2)$.}
\end{cases}
\end{equation*}
Hence, the equation
\begin{equation*}
\hat{\vartheta}(L) = 2\pi \dfrac{k}{n}
\end{equation*}
is uniquely solvable if and only if $2\pi \frac{k}{n}$ belongs to the open interval of extrema
\begin{equation*}
\lim_{L\to 0^{+}} \hat{\vartheta}(L) \quad \text{ and } \quad \lim_{L\to \hat{L}^{-}} \hat{\vartheta}(L).
\end{equation*}
We now compute these limits. 
First, recalling the scaling formula \eqref{eq-P_HL-riscal}, we have
\begin{equation}\label{eq-scal}
 \Theta(\hat{H}(L),L) = \Theta(\hat{H}(L) L^{\frac{2\alpha}{2-\alpha}},1)
\end{equation}
and, by the first relation in \eqref{lim-1} and \eqref{lim-Theta0}, we deduce that
\begin{equation*}
\lim_{L\to 0^{+}} \hat{\vartheta}(L) =
\lim_{H\to 0} \Theta(H,1) = \dfrac{2\pi}{2-\alpha}.
\end{equation*}
Second, using again \eqref{eq-scal} and the second relation in \eqref{lim-1}, we have
\begin{equation*}
\lim_{L\to \hat{L}^{-}} \hat{\vartheta}(L) 
= \lim_{H\to-w_{0}(\hat{L})\hat{L}^{\frac{2\alpha}{2-\alpha}}} \Theta(H,1).
\end{equation*}
Since $w_{0}(\hat{L})\hat{L}^{\frac{2\alpha}{2-\alpha}} =w_{0}(1)$, from Proposition~\ref{prop-stime-tempo} (see also Remark~\ref{rem-cond-extra-dL}), we finally find
\begin{equation*}
\lim_{L\to \hat{L}^{-}} \hat{\vartheta}(L)  = \lim_{H\to -w_{0}(1)} \Theta(H,1) = \dfrac{2\pi}{\sqrt{\Omega_{0}^{(2)}(1)}} = \dfrac{2\pi}{\sqrt{2-\alpha}},
\end{equation*}
where the last equality comes from the fact that, with the notation of Appendix~\ref{appendix-A} $\Omega(u;1)=\frac{1}{2}u^{2} - \frac{\kappa}{\alpha} u^{\alpha} + w_{0}(1)$. The proof is complete.
\end{proof}

Notice that condition \eqref{cond-kn} is independent on $\tau$: hence, if a pair $(n,k)$ satisfies it, then there exists a $\tau$-periodic solution of type $(n,k)$ for any $\tau > 0$. Actually, by the homogeneity of the problem, these solutions are of the form
\begin{equation*}
t \mapsto \tau^{\frac{2}{2+\alpha }}x\left( \dfrac{t}{\tau}\right),
\end{equation*}
where $x$ is the solution of type $(n,k)$ and minimal period equal to $1$. From this, it immediately follows the next proposition.

\begin{proposition}\label{prop-loc}
For every $\tau > 0$, for every coprime positive integers $n,k$ satisfying \eqref{cond-kn}, and for every integer $\ell \geq 1$, let $x_\ell$ be the unique periodic solution of \eqref{eq-akep} of type $(n,k)$ and minimal period $\tau/\ell$. Then
\begin{equation*}
\lim_{\ell \to +\infty} \| x_\ell \|_{\infty} = 0,
\end{equation*}
where $\| x_\ell \|_{\infty} = \max_{t \in \mathbb{R}} | x_\ell(t) |$.
\end{proposition}

\subsection{Periodic solutions galore to a restricted $3$-body problem}\label{section-5.3}

Let us consider the system
\begin{equation}\label{eq-res}
\ddot q = GM \dfrac{\xi_{M}(t) - q}{| \xi_{M}(t) - q |^{\alpha+2}} + Gm \dfrac{\xi_{m}(t) - q}{| \xi_{m}(t) - q |^{\alpha+2}}, \quad q \in \mathbb{R}^{2},
\end{equation}
where $\alpha \in (0,2)$, $G,M, m > 0$, and $\xi_{M}, \xi_{m}$ are given functions. The above equation can be meant as a model for the planar motion of a massless particle $q$ under the attraction of two heavy bodies (the so-called primaries) with positions $\xi_{M}$ and $\xi_{m}$ and masses $M$ and $m$, respectively. The classical case of a Newtonian attraction ($\alpha = 1$) gives rise to the usual restricted $3$-body problem; here, we deal with the non-Newtonian case $\alpha \neq 1$. For some results about the $-\alpha$-homogeneous $N$-body problem, see \cite{BaTeVe-14,McG-81} and the references therein.

From now on on, we assume that the primaries move along circular orbits of the $-\alpha$-homogeneous two-body problem
\begin{equation*}
\ddot \xi_M = Gm \dfrac{\xi_m - \xi_M}{\vert \xi_m - \xi_M \vert^{\alpha+2}}, \qquad \ddot \xi_m = GM \dfrac{\xi_M - \xi_m}{\vert \xi_M - \xi_m \vert^{\alpha+2}}.
\end{equation*}
Thus, $\xi_{M}$ and $\xi_{m}$ are periodic functions with a minimal period $\tau > 0$, depending on the radii of the circles on which the primaries move. By normalization, it is possible to set 
\begin{equation*}
\tau = 2\pi, \qquad  m + M = 1, \qquad \vert \xi_M - \xi_m \vert = 1,
\end{equation*}
implying $G = 1$. Finally, we suppose that the center of mass of the primaries is placed at the origin, namely
\begin{equation*}
m \xi_{m}(t) + M \xi_{M}(t) = 0, \quad \text{for every $t \in \mathbb{R}$}.   
\end{equation*}
According to these assumptions, 
\begin{equation*}
\xi_{M}(t) =m e^{ i t}, 
\qquad \xi_{m}(t) = - (1-m) e^{ i t},
\end{equation*}
and therefore from now we deal with
\begin{equation}\label{eq-resnorm}
\ddot q = (1-m) \dfrac{m e^{ i t} - q}{| m e^{ i t} - q |^{\alpha+2}} - m \dfrac{ (1-m) e^{ i t} + q}{|(1-m) e^{ i t} + q |^{\alpha+2}}, \quad q \in \mathbb{R}^{2}.
\end{equation}
In this setting, our result is the following.

\begin{theorem}\label{th-restricted}
Let $\alpha\in(0,2)$ and $\alpha \neq 1$. 
Then, for every $N \geq 1$ there exists $m_{*} > 0$ such that, for every $m \in (0,m_{*})$, there exist $N$ periodic solutions of~\eqref{eq-resnorm} having period $2\pi$.
\end{theorem}

As it will be clear from the proof, the above solutions are obtained, after some suitable change of variables, as perturbations of periodic solutions of a $-\alpha$-homogeneous Kepler problem of type $(n,k)$ and minimal period 
\begin{equation*}
\dfrac{2\pi}{\ell_{*}(n,k) + 1}, \dfrac{2\pi}{\ell_{*}(n,k) + 2}, \cdots, \dfrac{2\pi}{\ell_{*}(n,k) + N},
\end{equation*}
where $\ell_{*}(n,k)$ is a sufficiently large integer. The pair $(n,k)$ can be arbitrarily chosen, as long as $n,k$ are coprime positive integers satisfying condition \eqref{cond-kn}. Hence, by varying the pair $(n,k)$ other periodic solutions can be provided: all of them remain distinct for $m>0$ sufficiently small, since they bifurcate from distinct tori of the unperturbed problem. System \eqref{eq-res} thus possesses periodic solutions galore.

\begin{proof}
Let us perform the change of variable
\begin{equation*}
x(t) = (1-m)^{-\frac{1}{\alpha + 2}} (q(t) - m e^{ i t}),
\end{equation*}
so as to place the primary with larger mass ($\xi_M(t) = m e^{ i t}$) at the origin. Therefore, system \eqref{eq-resnorm} becomes
\begin{equation*}
\ddot x = - \dfrac{x}{\vert x \vert^{\alpha + 2}} -   
m \left( (1-m)^{-1}\dfrac{(1-m)^{-\frac{1}{\alpha + 2}} e^{i t} + x}{\vert (1-m)^{-\frac{1}{\alpha + 2}} e^{i t} + x \vert^{\alpha + 2}} - (1-m)^{-\frac{1}{\alpha + 2}} e^{ i t} \right),
\end{equation*}
which is of the form
\begin{equation}\label{eq-res3}
\ddot x = - \dfrac{x}{\vert x \vert^{\alpha + 2}} + m \nabla_x U(t,x,m),
\end{equation}
with
\begin{equation*}
U(t,x,m) =  \dfrac{(1-m)^{-1}}{\alpha \vert  (1-m)^{-\frac{1}{\alpha + 2}} e^{ i t} +x \vert^{\alpha} } + (1-m)^{-\frac{1}{\alpha + 2}} \langle e^{i t}, x \rangle .
\end{equation*}
Note that the perturbation term $U$ is singular on the circle of radius $(1-m)^{-\frac{1}{\alpha + 2}}$, due to possible collisions with the primary with smaller mass. This radius tends to $1$ as $m \to 0^{+}$.

We now regard $m$ as a parameter, so as to interpret \eqref{eq-res3} as a perturbation of the $-\alpha$-homogeneous Kepler problem \eqref{eq-akep} (with $\kappa = 1$). Fixing arbitrarily $n,k$ coprime positive integers satisfying condition \eqref{cond-kn}, by Theorem~\ref{th-rosette} the unperturbed problem ($m = 0$) admits a periodic solution $x_\ell$ of type $(n,k)$ and minimal period $2\pi/\ell$, for every integer $\ell \geq 1$. Moreover, by Proposition~\ref{prop-loc} there exists $\ell^{*}(n,k)$ such that $\Vert x_\ell \Vert_\infty < 1/2$ if $\ell > \ell^{*}(n,k)$. On this open ball, the perturbation term $U$ is well-defined for $m$ small.  Thus, fixed an integer $N \geq 1$, we can apply Theorem~\ref{th-main}, in the generalized version discussed in Remark~\ref{rem-3.1}, to bifurcate from the tori associated with $x_{\ell^{*}+i}$, for every $i=1,\ldots,N$, as long as $m$ is small enough. The proof is thus completed.
\end{proof}

\appendix 
\section{Some auxiliary results on time-maps and their derivatives}\label{appendix-A}

In this appendix, we provide some technical results for the time-maps $T$ and $P$ introduced in Section~\ref{section-2}. More precisely, we are going to provide explicit formulas for both the derivatives with respect to $H$ and $L$ (see Lemma~\ref{lem-chicone}) and, from this, we compute their limit values as $H$ goes to the minimal value allowed (see Proposition~\ref{prop-stime-tempo}).
The results on the dependence on the energy parameter $H$ are essentially known from \cite{Ch-87, Sc-85, Sc-90}, while the ones for the parameter $L$ seem to be new. 

As observed in Section~\ref{section-2}, the function $T$ and $P$ defined in \eqref{def-T-HL} and \eqref{def-P-HL} are the period maps of the nonlinear oscillators
\begin{equation*}
\dfrac{1}{2} \dot{r}^{2} +\dfrac{1}{2} \dfrac{L^{2}}{r^{2}}-V(r)=H
\end{equation*}
and
\begin{equation*}
\dfrac{1}{2} \dot{u}^{2} +\dfrac{1}{2} L^{2}\, u^{2} - V\biggl{(}\dfrac{1}{u} \biggr{)} =H,
\end{equation*}
respectively. We aim to introduce a unifying approach to these two oscillators.

Let $I,J\subseteq(0,+\infty)$ be open intervals and $\mathcal{V}\in\mathcal{C}^{2}(I,\mathbb{R})$. For $k\in\{-1,1\}$, we define $\mathcal{W} \colon I\times J \to \mathbb{R}$ as
\begin{equation*}
\mathcal{W}(s;L) = \dfrac{1}{2} L^{2} s^{2k} - \mathcal{V}(s),\quad (s,L)\in I\times J.
\end{equation*}
In the following, the derivatives with respect to $s$ are simply denoted as $\mathcal{W}^{(i)}$. Observe that $\mathcal{W}^{(i)}$, $i=0,1,2,$ is differentiable with respect to $L$ and
\begin{equation}\label{eq-derWL}
\begin{aligned}
\partial_{L} \mathcal{W}(s;L) &= L s^{2k},
\\
\partial_{L} \mathcal{W}'(s;L) &= 2kL s^{2k-1},
\\
\partial_{L} \mathcal{W}''(s;L) &= 2k(2k-1)L s^{2k-2},	
\end{aligned}
\end{equation}
for every $(s,L)\in I\times J$.

We assume that
\begin{itemize}
\item [$(h_{\mathcal{W}})$] for all $L\in J$ the function $\mathcal{W}(\cdot;L)$ has a (strict) local minimum at $s=s_{0}(L)$ with $\mathcal{W}'(\cdot;L)<0$ on $(s_{*}(L),s_{0}(L))$ and $\mathcal{W}'(\cdot;L)>0$ on $(s_{0}(L),s^{*}(L))$, for some $s_{*}(L) < s_{0}(L) < s^{*}(L)$.
\end{itemize}
We set 
\begin{equation*}
I_{L}=(s_{*}(L),s^{*}(L)), \qquad
\Gamma=\{(s,L)\in\mathbb{R}^{2} \colon L\in J, \; s\in I_{L}\}.
\end{equation*}
From now on, we implicitly assume that $(s,L)\in \Gamma$.

Let us define
\begin{equation}\label{def-omega0}
\omega_{0}(L) = - \mathcal{W}(s_{0}(L);L).
\end{equation}
The regularity of $\mathcal{W}$ implies that $\omega_{0}$ is differentiable and, from \eqref{eq-derWL} and \eqref{def-omega0}, we have that
\begin{equation}\label{eq-derw0}
\begin{aligned}
\omega_{0}'(L) &= - \mathcal{W}'(s_{0}(L);L) s_{0}'(L) - \partial_{L} \mathcal{W}(s_{0}(L);L) 
\\
&= - \partial_{L} \mathcal{W}(s_{0}(L);L)=-L s_{0}(L)^{2k},
\end{aligned}
\end{equation}
for every $L\in J$.

Let us observe that in general the minimum level $-\omega_{0}(L)$ of the potential $\mathcal{W}$ is not zero; in the next computations it will be useful to have a zero minimum, so we define
\begin{equation*}
\Omega(s;L) = \mathcal{W}(s;L)+\omega_{0}(L), 
\quad (s,L)\in \Gamma.
\end{equation*}
From this definition and from hypothesis $(h_{\mathcal{W}})$ we plainly deduce that
\begin{equation} \label{eq-nuovaomega}
\Omega(s_0(L);L)=\Omega'(s_0(L);L)=0, 
\end{equation}
for every $L\in J$, and that
\begin{equation} \label{eq-segno2}
(s-s_{0}(L))\Omega'(s;L)>0, 
\end{equation}
for every $L\in J$ and $s\in I_{L} \setminus\{ s_{0}(L)\}$.
As a consequence, for all $L\in J$ there exists $H_{0}(L)>-\omega_{0}(L)$ such that for all $H\in(-\omega_{0}(L), H_{0}(L))$ the equation $\Omega(s;L)=H+\omega_{0}(L)$ has two solutions $s_{\pm}(H,L)$ such that
\begin{equation*}
s_{-}(H,L)\in (s_{*}(L),s_{0}(L)), \qquad
s_{+}(H,L)\in (s_{0}(L), s^{*}(L)).
\end{equation*}
For simplicity, we assume that $H_{0}$ depends continuously on $L$, so that the set
\begin{equation*}
\Lambda=\bigl{\{}(H,L)\in I\times J \colon \omega_{0}(L)<H<H_{0}(L) \bigr{\}},
\end{equation*}
is open.
We now define
\begin{equation*}
T_{\Omega}(H,L) = \sqrt{2} \int_{s_{-}(H,L)}^{s_{+}(H,L)} \dfrac{\mathrm{d}s}{\sqrt{H+\omega_{0}(L)-\Omega(s;L)}},\quad (H,L)\in \Lambda.
\end{equation*}
The function $T_\Omega$ can obviously be written as 
\begin{equation*}
T_{\Omega}(H,L) = \sqrt{2} \int_{s_{-}(H,L)}^{s_{+}(H,L)} \dfrac{\mathrm{d}s}{\sqrt{H-\mathcal{W}(s;L)}}.
\end{equation*}
In the above formula we recognize the expressions of the time-maps $T$ and $P$ in \eqref{def-T-HL} and \eqref{def-P-HL}, corresponding to $k=-1$ and $\mathcal{V}(s)=V(s)$, and $k=1$ and $\mathcal{V}(s)=V(1/s)$, respectively, with $V$ defined as in Section~\ref{section-2}.

\begin{remark} \label{rem-regTTheta}
Via a standard implicit function argument, it is possible to prove that $T_\Omega\in \mathcal{C}^1(\Lambda)$. Indeed, let us denote by $s(\cdot; s_{0},L)$ the solution of the Cauchy problem
\begin{equation*}
\begin{cases}
\,\ddot{s}+\mathcal{W}'(s;L)=0,
\\
\, s(0)=s_{0},\quad \dot{s}(0)=0.
\end{cases}
\end{equation*}
Then, $T_\Omega$ is implicitly defined by
\begin{equation*}
\dot{s}(T_\Omega (H,L);s_{+}(H,L),L)=0,
\end{equation*}
for every $(H,L)\in \Lambda$.
Since the vector solution $(s,\dot{s})$ and $s_{+}$ depend in a $\mathcal{C}^{1}$-way from all their variables, the conclusion follows by hypothesis $(h_{\mathcal{W}})$.
\hfill$\lhd$
\end{remark}

In the following, we are going to prove asymptotic estimates on $T_\Omega$, $\partial_{H} T_\Omega$ and $\partial_{L} T_\Omega$ when the energy $H$ goes to the minimum value allowed $-\omega_{0}(L)$. 
To this aim, we assume that
\begin{equation}\label{cond-extra}
\mathcal{W}''(s_{0}(L);L)>0, \quad \text{for every $L\in J$,}
\end{equation}
and that $\mathcal{V}\in\mathcal{C}^{4}(I)$.
From \eqref{cond-extra} and the regularity of $\mathcal{W}$ it follows that $s_{0}$ is differentiable in $J$ and, from $\mathcal{W}'(s_{0}(L);L)=0$ and \eqref{eq-derWL} we have that
\begin{equation}\label{eq-ders0}
s'_{0}(L)=-\dfrac{2kLs_{0}(L)^{2k-1}}{\mathcal{W}''(s_{0}(L);L)},
\end{equation}
for every $L\in J$.

From the regularity of $\mathcal{V}$ and $\omega_{0}$, it is immediate to see that $\Omega (\cdot;L)\in \mathcal{C}^{4}(I_{L})$ and that $\partial_{L} \Omega$,  $\partial^{2}_{L} \Omega$, $\partial_{L} \Omega'$ and $\partial_{L} \Omega''$ exist; moreover, setting 
\begin{equation*}
\Omega_{0}^{(i)} (L) = \Omega^{(i)}(s_{0}(L);L), \quad i\in\{1,2,3,4\},
\end{equation*}
we have 
\begin{equation} \label{eq-deromegazero}
\Omega_{0}^{(i)} (L) = \mathcal{W}^{(i)}(s_{0}(L);L), \quad i\in\{1,2,3,4\},
\end{equation}
for every $L\in J$. In particular, condition \eqref{cond-extra} writes equivalently as
\begin{equation}\label{cond-extra-omega02}
\Omega_{0}^{(2)} (L)>0, \quad \text{for every $L\in J$.}
\end{equation}

It is now convenient to write $T_{\Omega}$ in a different way, by a suitable change of variable in the integral; hence, recalling and inspired by \cite{Ch-87}, we first define
\begin{equation} \label{eq-defh}
h(s;L) = \mathrm{sgn} (s-s_{0}(L)) \sqrt{\Omega(s;L)}, \quad (s,L)\in \Gamma.
\end{equation}
The regularity of $\Omega$ and $s_{0}$ directly implies that $h\in \mathcal{C}(\Gamma)$ and $h', h'', h''', \partial_{L} h, \partial_{L} h'\in \mathcal{C}(\Gamma\setminus \gamma)$, where $\gamma$ is the set
\begin{equation*}
\gamma=\bigl{\{} (s,L)\in \Gamma \colon s=s_{0}(L),\, L\in J \bigr{\}}.
\end{equation*}
Under the supplementary condition \eqref{cond-extra-dL} below, it is possible to show that the regularity properties of the above functions can be extended to whole domain $\Gamma$, as stated in the next preliminary lemma.

\begin{lemma}\label{lem-h}
Let us assume \eqref{cond-extra-omega02} and that 
\begin{equation}\label{cond-extra-dL}
\biggl{(} \dfrac{\mathrm{d}}{\mathrm{d}L} \Omega^{(2)}_{0} \biggr{)}(L) \neq 0, \quad \text{for every $L\in J$.} 
\end{equation}
Then, it holds that
\begin{itemize}
\item[$(i)$] $h', h'', h'''\in\mathcal{C}(\Gamma)$ and for every $L\in J$ we have 
\begin{equation}\label{eq-hsugamma1}
\begin{aligned}
h'(s_{0}(L);L) &= \dfrac{(\Omega^{(2)}_{0}(L))^{\frac{1}{2}}}{\sqrt{2}},
\\
h''(s_{0}(L);L) &= \dfrac{\Omega^{(3)}_{0}(L)}{3\sqrt{2} \, ( \Omega^{(2)}_{0}(L))^{\frac{1}{2}}},
\\
h'''(s_{0}(L);L) &= \displaystyle \dfrac{3\Omega^{(2)}_{0}(L)\, \Omega^{(4)}_{0}(L)-(\Omega^{(3)}_{0}(L))^{2}}{ 12\sqrt{2} \, (\Omega^{(2)}_{0}(L))^{\frac{3}{2}}}.
\end{aligned}
\end{equation}
\item[$(ii)$]  $\partial_{L} h,\partial_{L} h'\in\mathcal{C}(\Gamma)$ and for every $L\in J$ we have
\begin{equation}\label{eq-hsugamma2}
\begin{aligned}
\partial_{L} h (s_{0}(L);L) &= \dfrac{\sqrt{2} kLs_{0}(L)^{2k-1}}{(\Omega^{(2)}_{0}(L))^{\frac{1}{2}}},
\\
\partial_{L} h' (s_{0}(L);L) &= \dfrac{3 \Omega^{(2)}_{0}(L) \biggl{(}\dfrac{\mathrm{d}}{\mathrm{d}L} \Omega_{0}^{(2)} \biggr{)}(L) + 4 k L s_{0}(L)^{2k-1} \Omega^{(3)}_{0}(L)}{6\sqrt{2} \, (\Omega^{(2)}_{0}(L))^{\frac{3}{2}}}.
\end{aligned}
\end{equation}
\end{itemize}
\end{lemma}

\begin{proof} 
Let us first define the set 
\begin{equation*}
\widetilde{\Gamma}=\bigl{\{} (\xi,L)\in \mathbb{R}\times J \colon (\xi+s_0(L),L)\in \Gamma \bigr{\}}
\end{equation*}
and observe that $(0,L)\in \widetilde{\Gamma}$, for every $L\in J$. Let us also introduce the functions $\widetilde{\Omega}, \tilde{h}\colon \widetilde{\Gamma}\to \mathbb{R}$ defined by
\begin{equation*}
\widetilde{\Omega}(\xi;L) = \Omega(\xi+s_{0}(L);L),\quad (\xi,L)\in \widetilde{\Gamma},
\end{equation*}
and
\begin{equation*}
\tilde{h}(\xi;L) = \mathrm{sgn}(\xi) \sqrt{\widetilde{\Omega}(\xi;L)},\quad (\xi,L)\in \widetilde{\Gamma},
\end{equation*}
respectively.

We will prove that $\tilde{h}',\tilde{h}'',\tilde{h}''',\partial_{L} \tilde{h}, \partial_{L} \tilde{h}'$ are continuous in $\widetilde{\Gamma}$. Since 
\begin{equation} \label{eq-legameh}
\tilde{h}(\xi;L) = h(\xi+s_{0}(L);L), \quad \text{for every $(\xi,L)\in \widetilde{\Gamma}$,}
\end{equation}
and $s_{0}\in\mathcal{C}^{1}(J)$, this implies the required regularity properties of $h$ and its derivatives.

According to \eqref{eq-nuovaomega} and \eqref{eq-segno2}, we plainly deduce that $\widetilde{\Omega}$ satisfies 
\begin{equation}\label{eq-der-omega}
\widetilde{\Omega}(0,L) = \widetilde{\Omega}'(0,L) = 0, \quad \text{for every $L\in J$,}
\end{equation}
and $\widetilde{\Omega}(\xi;L)>0$, for every $(\xi,L)\in \widetilde{\Gamma}$ with $\xi \neq 0$. As a consequence,
for every $(\xi,L)\in \widetilde{\Gamma}$ with $\xi\neq 0$ we have
\begin{equation}\label{eq-derhtilde}
\begin{aligned}
\tilde{h}'(\xi;L) &= \mathrm{sgn} (\xi) \, \dfrac{\widetilde{\Omega}'(\xi;L)}{2(\widetilde{\Omega}(\xi;L))^{\frac{1}{2}}}
\\
\tilde{h}''(\xi;L) &= \mathrm{sgn} (\xi)\, \dfrac{2 \widetilde{\Omega}(\xi;L)\widetilde{\Omega}''(\xi;L)-(\widetilde{\Omega}'(\xi;L))^{2}}{4 (\widetilde{\Omega}(\xi;L))^{\frac{3}{2}}}
\\
\tilde{h}'''(\xi;L) &= \mathrm{sgn} (\xi)\, \dfrac{4(\widetilde{\Omega}(\xi;L))^2\widetilde{\Omega}'''(\xi;L)-6 \widetilde{\Omega}(\xi;L)\widetilde{\Omega}'(\xi;L)\widetilde{\Omega}''(\xi;L)+3(\widetilde{\Omega}'(\xi;L))^{3}}{8 (\widetilde{\Omega}(\xi;L))^{\frac{5}{2}}}
\\
\partial_{L} \tilde{h}(\xi;L) &= \mathrm{sgn}\, (\xi) \dfrac{\partial_{L} \widetilde{\Omega}(\xi;L)}{2 (\widetilde{\Omega}(\xi;L))^{\frac{1}{2}}}
\\
\partial_{L} \tilde{h}'(\xi;L) &= \mathrm{sgn} (\xi) \, \dfrac{2 \widetilde{\Omega}(\xi;L) \partial_{L} \widetilde{\Omega}'(\xi;L) - \widetilde{\Omega}'(\xi;L) \partial_{L} \widetilde{\Omega}(\xi;L)}{4 (\widetilde{\Omega}(\xi;L))^{\frac{3}{2}}},
\end{aligned}
\end{equation}
showing the continuity of these functions in $\widetilde{\Gamma}\setminus \{(0,L)\colon L\in J\}$.

Now, let us focus on the continuity at $(0,\bar{L})$, with $\bar{L}\in J$. From \eqref{eq-der-omega} we deduce that
\begin{equation} 
\partial_{L}\widetilde{\Omega}(0,L) 
= \partial^{2}_{LL}\widetilde{\Omega}(0,L)
= \partial^{3}_{LLL}\widetilde{\Omega}(0,L)
= \partial^{4}_{LLLL}\widetilde{\Omega}(0,L)
=0
\end{equation}
and that
\begin{equation}
\partial_{L}\widetilde{\Omega}'(0,L) 
= \partial^{2}_{LL}\widetilde{\Omega}'(0,L)
= \partial^{3}_{LLL}\widetilde{\Omega}'(0,L)
=0,
\end{equation}
for every $L\in J$.
Observing that
\begin{equation*}
\widetilde{\Omega}^{(i)}(0,L) = \Omega_{0}^{(i)}(L),
\end{equation*}
for every $L\in J$, it is immediate to prove the validity of the following Taylor expansions in $(0,\bar{L})$ with $\bar{L}\in J$:
\begin{align*}
\widetilde{\Omega}(\xi,L) &=
\dfrac{1}{2} \Omega_{0}^{(2)}(\bar{L}) \xi^{2} + 
\dfrac{1}{6} \biggl{(} \Omega_{0}^{(3)}(\bar{L}) \xi^{3} 
+ 3 \biggl{(}\dfrac{\mathrm{d}}{\mathrm{d}L} \Omega_{0}^{(2)} \biggr{)}(\bar{L}) \xi^2 (L-\bar{L})
\biggr{)}
\\
&+ \dfrac{1}{24} \biggl{(} \Omega_{0}^{(4)}(\bar{L}) \xi^{4} 
+ 4 \biggl{(}\dfrac{\mathrm{d}}{\mathrm{d}L} \Omega_{0}^{(3)} \biggr{)}(\bar{L}) \xi^3 (L-\bar{L})
+ 6 \biggl{(}\dfrac{\mathrm{d}^{2}}{\mathrm{d}L^{2}} \Omega_{0}^{(2)} \biggr{)}(\bar{L}) \xi^2 (L-\bar{L})^2
\biggr{)}
\\
&+ O(|(\xi,L-\bar{L})|^{5}), 
\quad (\xi,L)\to(0,\bar{L}),
\\
\widetilde{\Omega}'(\xi,L) &=
\Omega_{0}^{(2)}(\bar{L}) \xi + 
\dfrac{1}{2} \biggl{(} \Omega_{0}^{(3)}(\bar{L}) \xi^{2} 
+ 2 \biggl{(}\dfrac{\mathrm{d}}{\mathrm{d}L} \Omega_{0}^{(2)} \biggr{)}(\bar{L}) \xi (L-\bar{L})
\biggr{)}
\\
&+ \dfrac{1}{6} \biggl{(} \Omega_{0}^{(4)}(\bar{L}) \xi^{3} 
+ 3 \biggl{(}\dfrac{\mathrm{d}}{\mathrm{d}L} \Omega_{0}^{(3)} \biggr{)}(\bar{L}) \xi^2 (L-\bar{L})
+ 3 \biggl{(}\dfrac{\mathrm{d}^{2}}{\mathrm{d}L^{2}} \Omega_{0}^{(2)} \biggr{)}(\bar{L}) \xi (L-\bar{L})^2
\biggr{)}
\\
&+ O(|(\xi,L-\bar{L})|^{4}), 
\quad (\xi,L)\to(0,\bar{L}),
\\
\widetilde{\Omega}''(\xi,L) &=
\Omega_{0}^{(2)}(\bar{L}) + 
 \Omega_{0}^{(3)}(\bar{L}) \xi
+ \biggl{(}\dfrac{\mathrm{d}}{\mathrm{d}L} \Omega_{0}^{(2)} \biggr{)}(\bar{L}) (L-\bar{L})
\\
&+ \dfrac{1}{2} \biggl{(} \Omega_{0}^{(4)}(\bar{L}) \xi^{2} 
+ 2 \biggl{(}\dfrac{\mathrm{d}}{\mathrm{d}L} \Omega_{0}^{(3)} \biggr{)}(\bar{L}) \xi (L-\bar{L})
+ \biggl{(}\dfrac{\mathrm{d}^{2}}{\mathrm{d}L^{2}} \Omega_{0}^{(2)} \biggr{)}(\bar{L}) (L-\bar{L})^2
\biggr{)}
\\
&+ O(|(\xi,L-\bar{L})|^{3}), 
\quad (\xi,L)\to(0,\bar{L}),
\\
\widetilde{\Omega}'''(\xi,L) &=
\Omega_{0}^{(3)}(\bar{L})
+ \Omega_{0}^{(4)}(\bar{L}) \xi 
+ \biggl{(}\dfrac{\mathrm{d}}{\mathrm{d}L} \Omega_{0}^{(3)} \biggr{)}(\bar{L}) (L-\bar{L})
\\
&+ O(|(\xi,L-\bar{L})|^{2}), 
\quad (\xi,L)\to(0,\bar{L}),
\\
\partial_{L} \widetilde{\Omega}(\xi,L) &=
\dfrac{1}{2} \biggl{(}\dfrac{\mathrm{d}}{\mathrm{d}L} \Omega_{0}^{(2)} \biggr{)}(\bar{L}) \xi^2
+ \dfrac{1}{6} \biggl{(}  \biggl{(}\dfrac{\mathrm{d}}{\mathrm{d}L} \Omega_{0}^{(3)} \biggr{)}(\bar{L}) \xi^3 
+ 3 \biggl{(}\dfrac{\mathrm{d}^{2}}{\mathrm{d}L^{2}} \Omega_{0}^{(2)} \biggr{)}(\bar{L}) \xi^2 (L-\bar{L})
\biggr{)}
\\
&+ O(|(\xi,L-\bar{L})|^{4}), 
\quad (\xi,L)\to(0,\bar{L}),
\\
\partial_{L} \widetilde{\Omega}'(\xi,L) &= 
\biggl{(}\dfrac{\mathrm{d}}{\mathrm{d}L} \Omega_{0}^{(2)} \biggr{)}(\bar{L}) \xi
+ \dfrac{1}{2} \biggl{(}\dfrac{\mathrm{d}}{\mathrm{d}L} \Omega_{0}^{(3)} \biggr{)}(\bar{L}) \xi^2
+ \biggl{(}\dfrac{\mathrm{d}^{2}}{\mathrm{d}L^{2}} \Omega_{0}^{(2)} \biggr{)}(\bar{L}) \xi (L-\bar{L})
\\
&+ O(|(\xi,L-\bar{L})|^{3}), 
\quad (\xi,L)\to(0,\bar{L}).
\end{align*}
Recalling \eqref{eq-derhtilde}, quite long but simple computations prove that
\begin{equation}\label{eq-limitihtilde}
\begin{aligned}
&\lim_{(\xi,L)\to (0,\bar{L})} \tilde{h}'(\xi;L)
=\dfrac{(\Omega^{(2)}_{0}(\bar{L}))^{\frac{1}{2}}}{\sqrt{2}},
\\
&\lim_{(\xi,L)\to (0,\bar{L})} \tilde{h}''(\xi;L) 
=
\dfrac{\Omega^{(3)}_{0}(\bar{L})}{3\sqrt{2} \, ( \Omega^{(2)}_{0}(\bar{L}))^{\frac{1}{2}}},
\\
&\lim_{(\xi,L)\to (0,\bar{L})} \tilde{h}'''(\xi;L)
=\dfrac{3\Omega^{(2)}_{0}(\bar{L})\, \Omega^{(4)}_{0}(\bar{L})-(\Omega^{(3)}_{0}(\bar{L}))^{2}}{ 12\sqrt{2} \, (\Omega^{(2)}_{0}(\bar{L}))^{\frac{3}{2}}},
\\
&\lim_{(\xi,L)\to (0,\bar{L})} \partial_{L} \tilde{h}(\xi;L) = 0,
\\
&\lim_{(\xi,L)\to (0,\bar{L})} \partial_{L} \tilde{h}'(\xi;L)
=\dfrac{\biggl{(}\dfrac{\mathrm{d}}{\mathrm{d}L} \Omega_{0}^{(2)} \biggr{)}(\bar{L})}{2\sqrt{2} (\Omega^{(2)}_{0}(\bar{L}))^{\frac{1}{2}}},
\end{aligned}
\end{equation}
thus showing that the functions $\tilde{h}',\tilde{h}'', \tilde{h}''', \partial_{L} \tilde{h}, \partial_{L} \tilde{h}'$ can be extended in a continuous way to the whole set $\widetilde{\Gamma}$.

In order to conclude, we observe that the validity of the relations in \eqref{eq-hsugamma1} and \eqref{eq-hsugamma2} is a consequence of \eqref{eq-ders0}, \eqref{eq-legameh} and \eqref{eq-limitihtilde}. Indeed, taking into account \eqref{eq-deromegazero}, let us first notice that \eqref{eq-ders0} can be written as
\begin{equation}\label{eq-dim11}
s'_0(L)=-\dfrac{2kLs_0(L)^{2k-1}}{\Omega^{(2)}_0(L)},
\end{equation}
for every $L\in J$.
Now, \eqref{eq-hsugamma1} follows from the first three relations in \eqref{eq-limitihtilde}, observing that
\begin{equation*}
h^{(j)}(s_0(L);L)=\tilde{h}^{(j)}(0;L), \quad j\in\{1,2,3\},
\end{equation*}
for every $L\in J$.
On the other hand, from \eqref{eq-legameh} we infer that
\begin{align*}
\partial_{L} h(s_0(L);L)
&=\partial_{L} \tilde{h}(0;L)-h'(s_0(L);L)s'_0(L),
\\
\partial_{L} h'(s_0(L);L)
&=\partial_{L} \tilde{h}'(0;L)-h''(s_0(L);L)s'_0(L),
\end{align*}
for every $L\in J$. The last relations in \eqref{eq-limitihtilde}, \eqref{eq-dim11} and \eqref{eq-hsugamma1} allow then to conclude  that \eqref{eq-hsugamma2} holds true. The proof is complete.
\end{proof}

From the above lemma, by means of straightforward computations, it is possible to prove the following result.

\begin{lemma}\label{lem-h1}
Let us assume \eqref{cond-extra-omega02} and \eqref{cond-extra-dL}. Then, it holds that
\begin{align*}
&3(h''(s_0(L);L))^{2} - h'(s_0(L);L)h'''(s_0(L);L) = 
\\
&= \dfrac{1}{24 \Omega_{0}^{(2)}(L)} \Bigl{(} 5 (\Omega_{0}^{(3)}(L))^{2} - 3 \Omega_{0}^{(2)}(L) \Omega_{0}^{(4)}(L) \Bigr{)},
\end{align*}
and
\begin{equation} \label{eq-h3}
h''(s_0(L);L) \partial_{L} h(s_0(L);L) - h'(s_0(L);L) \partial_{L} h'(s_0(L);L) =  
-\dfrac{1}{4} \dfrac{\mathrm{d}}{\mathrm{d}L} \Omega_{0}^{(2)}(L),
\end{equation}
for every $L\in J$.
\end{lemma}

According to \cite{Ch-87}, the function $h$ given in \eqref{eq-defh} can be used to introduce a change of variable in the integral defining $T_{\Omega}$ which transforms it in an integral which is not improper. To this aim, let us first observe that
\begin{equation} \label{eq-derh}
h'(s;L) = \mbox{sgn} (s-s_{0}(L)) \dfrac{\Omega'(s;L)}{2\sqrt{\Omega(s;L)}},\quad (s,L)\in \Gamma\setminus \gamma,
\end{equation}
and that $h'(s;L)$ can be extended by continuity to $\gamma$, as in Lemma~\ref{lem-h}.

Moreover, from \eqref{eq-segno2}, \eqref{eq-derh} and the expression of $h'(s;L)$ when $(s,L)\in \gamma$ we infer that
\begin{equation}\label{eq-hprimonozero}
h'(s;L)>0, \quad \text{for every $(s,L)\in \Gamma$,}
\end{equation}
thus proving that $h(\cdot;L)$ in invertible in $I_L$, for every $L\in J$.

As a consequence, in the integral defining $T_\Omega$ we can introduce the change of variable
\begin{equation} \label{eq-cambiovar}
s=s(\theta;H,L):=h^{-1}(\sqrt{H+\omega_{0}(L)}\sin \theta),
\quad \theta \in \biggl{[}-\frac{\pi}{2},\frac{\pi}{2}\biggr{]}.
\end{equation}
Via this transformation, we can prove the following result.

\begin{lemma}\label{lem-chicone}
Let us assume \eqref{cond-extra-omega02} and \eqref{cond-extra-dL}. Then, for every $(H,L)\in \Lambda$ it holds that
\begin{align*}
T_{\Omega}(H,L) &= \sqrt{2} \int_{-\frac{\pi}{2}}^{\frac{\pi}{2}} 
\dfrac{1}{h'(s;L)}\Bigg{|}_{s=s(\vartheta;H,L)}\,\mathrm{d}\vartheta,
\\
\partial_{H} T_{\Omega}(H,L) &= \dfrac{1}{\sqrt{2}} \int_{-\frac{\pi}{2}}^{\frac{\pi}{2}} 
\dfrac{3 (h''(s;L))^{2} - h'(s;L) h'''(s;L)}{(h'(s;L))^{5}}  \Bigg{|}_{s=s(\vartheta;H,L)} \, \cos^{2} \vartheta  \, \mathrm{d}\vartheta,
\\
\partial_{L} T_{\Omega}(H,L) &= \omega_{0}'(L) \partial_{H} T_{\Omega}(H,L) 
\\
&\quad + \sqrt{2} \int_{-\frac{\pi}{2}}^{\frac{\pi}{2}} 
\dfrac{h''(s;L) \partial_{L}h (s;L) - h'(s;L)\partial_{L} h'(s;L)}{(h'(s;L))^{3}} \Bigg{|}_{s=s(\vartheta;H,L)}\, \mathrm{d}\vartheta,
\end{align*}
where $s=s(\theta;H,L)$ is given in \eqref{eq-cambiovar}.
\end{lemma}

\begin{proof}
The formulas for $T_\Omega$ and $\partial_{H} T_\Omega$ given in Lemma~\ref{lem-chicone} have been proved in~\cite[Section~2]{Ch-87} in the case of functions not depending on the parameter $L$; the proofs given in~\cite{Ch-87} are valid also in our situation.

As far as the formula for $\partial_{L} T_\Omega$ is concerned, let us first write $T_\Omega$ as 
\begin{equation*}
T_\Omega (H,L)=\sqrt{2} \int_{-\frac{\pi}{2}}^{\frac{\pi}{2}} \varphi(\vartheta;H,L)\,\mathrm{d}\vartheta,
\end{equation*}
for every $(H,L)\in \Lambda$,
where the function $\varphi$ is defined by 
\begin{equation*}
\varphi (\theta, H, L)=\dfrac{1}{h'(s(\vartheta;H,L);L)}, \quad (\vartheta,H,L)\in \left[-\dfrac{\pi}{2},\dfrac{\pi}{2}\right]\times \Lambda.
\end{equation*}
A simple computation shows that
\begin{align*}
\partial_{L} \varphi (\theta, H, L) &=
- \biggl( \dfrac{\omega'_0(L)}{2\sqrt{H+\omega_0(L)}}\sin \vartheta-\partial_{L} h (s(\vartheta;H,L);L)\biggr) \dfrac{h''(s(\vartheta;H,L);L)}{h'(s(\vartheta;H,L);L)^3}
\\
& \quad -\dfrac{\partial_{L} h'(s(\vartheta;H,L);L)}{h'(s(\vartheta;H,L);L)^2},
\end{align*}
for every $(\vartheta,H,L)\in [-\pi/2,\pi/2]\times \Lambda$. From Lemma~\ref{lem-h} and the regularity of $s$ and $\omega_0$ we deduce that $\partial_{L} \varphi$ is a continuous function in $[-\pi/2,\pi/2]\times \Lambda$. Hence, we can differentiate $T_\Omega$ with respect to $L$ and obtain
\begin{equation*}
\partial_{L} T_\Omega (H,L)=\sqrt{2} \int_{-\frac{\pi}{2}}^{\frac{\pi}{2}} 
\partial_{L} \varphi(\vartheta;H,L)\,\mathrm{d}\vartheta,
\end{equation*}
for every $(H,L)\in \Lambda$.
Integrating by parts and taking into account the expression of $\partial_H T_\Omega$, we obtain the thesis.	
\end{proof}

\begin{remark}\label{rem-Schaaf}
The well known Chicone's criterion \cite[Theorem A]{Ch-87} for the monotonicity of $T_{\Omega}$ with respect to $H$ follows directly from the integral formula for $\partial_H T_\Omega$ given in Lemma~\ref{lem-chicone}. Indeed, it is straightforward to check that
\begin{align*}
&3 (h''(s;L))^{2} - h'(s;L) h'''(s;L) = 
\\
&= 6 \Omega(s;L)(\Omega''(s;L))^2 - 3 (\Omega'(s;L))^2 \Omega''(s;L) 
- 2 \Omega(s;L)\Omega'(s;L)\Omega'''(s;L)
\end{align*}
and the right-hand side of the above equality is exactly the quantity appearing in the assumption of Chicone's result.

On the other hand, the Schaaf's criterion \cite[Theorem~1]{Sc-85} provides the monotonicity of $T_{\Omega}$ with respect to $H$ whenever the two conditions 
\begin{itemize}
\item[$(i)$] $5(\Omega'''(s;L))^{2} - 3 \Omega''(s;L) \Omega^{(4)}(s;L) >0$ for all $s>s_{0}(L)$,
\item[$(ii)$] $\Omega'(s_{0}(L);L)  \Omega'''(s_{0}(L);L) <0$,
\end{itemize}
are satisfied. Actually, as shown in the proof of \cite[Lemma~1]{Sc-85}, the above conditions 
imply that
\begin{equation}\label{chicone-cond}
6 \Omega(s;L)(\Omega''(s;L))^2 - 3 (\Omega'(s;L))^2 \Omega''(s;L) 
- 2 \Omega(s;L)\Omega'(s;L)\Omega'''(s;L) >0,
\end{equation}
so that Schaaf's criterion can be considered weaker than Chicone's one. In spite of this, however, in the proof of Lemma~\ref{lem-dH-T}, we have chosen to verify hypotheses $(i)$ and~$(ii)$ since this seems to be easier than proving directly Chicone's assumption~\eqref{chicone-cond}.
\hfill$\lhd$
\end{remark}

We are now ready to state the result on the asymptotic behaviour of $T_{\Omega}(H,L)$ and its derivatives as $H\to -\omega_{0}(L)$. To this aim, let us define
\begin{equation} \label{def-sigma-zero}
\Sigma_{0}(L) = 5 (\Omega_{0}^{(3)}(L))^{2}-3 \Omega_{0}^{(2)}(L) \Omega_{0}^{(4)}(L).
\end{equation}

\begin{proposition} \label{prop-stime-tempo}
Let us assume \eqref{cond-extra-omega02} and \eqref{cond-extra-dL}. Then, for every $L\in J$ the following estimates hold true:
\begin{align*}
\lim_{H\to (-\omega_{0}(L))^{+}} T_{\Omega}(H,L) &= \dfrac{2\pi}{ (\Omega_{0}^{(2)} (L))^{\frac{1}{2}}},
\\
\lim_{H\to (-\omega_{0}(L))^{+}} \partial_{H} T_{\Omega}(H,L) &= \dfrac{\pi}{12 (\Omega_{0}^{(2)} (L))^{\frac{7}{2}}}\, \Sigma_{0}(L),
\\
\lim_{H\to (-\omega_{0}(L))^{+}} \partial_{L} T_{\Omega}(H,L) &= 
-\dfrac{\pi L\, s_{0}(L)^{2k-2}}{12 (\Omega_{0}^{(2)} (L))^{\frac{7}{2}}} \Bigl(s_{0}(L)^{2} \Sigma_{0}(L)-24k \Omega_{0}^{(2)}(L)\Omega_{0}^{(3)}(L) s_{0}(L)
\\ &\hspace{140pt} +24k(2k-1) (\Omega_{0}^{(2)} (L))^{2}\Bigr).
\end{align*}
\end{proposition}

\begin{proof} Recalling \eqref{eq-cambiovar}, let us first observe that 
\begin{equation} \label{eq-dim22}
\lim_{H\to (-\omega_{0}(L))^{+}} s(\theta; H,L)=s_0(L), \quad \text{uniformly in $\theta \in \left[-\dfrac{\pi}{2},\dfrac{\pi}{2}\right]$},
\end{equation}
for every $L\in J$.
Hence, a simple argument shows that for every $L\in J$ the limits of $T_{\Omega}(H,L)$, $\partial_H T_{\Omega}(H,L)$ and $\partial_{L} T_{\Omega}(H,L)$ for $H\to -\omega_{0}(L)$ can be computed passing to the limit under integral sign.

Now, taking into account Lemma~\ref{lem-h}, Lemma~\ref{lem-h1}, Lemma~\ref{lem-chicone}, and \eqref{eq-dim22}, we immediately obtain
\begin{align*}
\lim_{H\to (-\omega_{0}(L))^{+}} T(H,L) 
&= \sqrt{2} \dfrac{\sqrt{2}}{(\Omega_{0}^{(2)}(L) )^{\frac{1}{2}}} \int_{-\frac{\pi}{2}}^{\frac{\pi}{2}} \,\mathrm{d}\vartheta =  \dfrac{2\pi}{(\Omega_{0}^{(2)}(L))^{\frac{1}{2}}}
\end{align*}
and
\begin{align*}
&\lim_{H\to (-\omega_{0}(L))^{+}} \partial_{H} T(H,L) =
\\
&= \dfrac{1}{\sqrt{2}} \dfrac{1}{24 \Omega_{0}^{(2)}(L)} \Bigl(5(\Omega_{0}^{(3)}(L))^{2}-3 \Omega_{0}^{(2)}(L) \Omega_{0}^{(4)}(L) \Bigr) \dfrac{2^{\frac{5}{2}}}{(\Omega_{0}^{(2)}(L))^{\frac{5}{2}}} \int_{-\frac{\pi}{2}}^{\frac{\pi}{2}} \cos^{2} \vartheta \,\mathrm{d}\vartheta
\\
&= \dfrac{\pi}{12 (\Omega_{0}^{(2)} (L))^{\frac{7}{2}}}\, \Sigma_{0}(L).
\end{align*}
As for the estimate of $\partial_{L} T_\Omega$, we first observe that 
\begin{equation*}
\dfrac{\mathrm{d}}{\mathrm{d}L} \Omega_{0}^{(2)}(L) = \dfrac{\mathrm{d}}{\mathrm{d}L} \Omega'' (s_0(L);L)=
\Omega'''(s_0(L);L)s'_0(L)+\partial_{L} \Omega''(s_0(L);L),
\end{equation*}
for every $L\in J$.
Recalling \eqref{eq-derWL}, \eqref{eq-ders0} and \eqref{eq-deromegazero} we then obtain
\begin{align*}
\dfrac{\mathrm{d}}{\mathrm{d}L} \Omega_{0}^{(2)}(L) 
&= - \dfrac{2kLs_0(L)^{2k-1}\Omega_0^{(3)}(L)}{\Omega_0^{(2)}(L)}+\partial_{L} \mathcal{W}''(s_0(L);L)
\\
&= -\dfrac{2kLs_0(L)^{2k-1}\Omega_0^{(3)}(L)}{\Omega_0^{(2)}(L)}+2k(2k-1)Ls_0(L)^{2k-2}
\\
&= -\dfrac{2kLs_0(L)^{2k-2}}{\Omega_0^{(2)}(L)}\,  \Bigl(s_0(L)\Omega_0^{(3)}(L)-(2k-1)\Omega_0^{(2)}(L)\Bigr),
\end{align*}
for every $L\in J$. Taking into account \eqref{eq-derw0}, the first relation in \eqref{eq-hsugamma1}, \eqref{eq-h3} and the asymptotic estimate of $\partial_H T_\Omega$ we deduce that
\begin{align*}
&\lim_{H\to (-\omega_{0}(L))^{+}} \partial_{L} T(H,L) =
\\
&=- \dfrac{\pi Ls_0(L)^{2k}}{12 (\Omega_0^{(2)}(L))^{\frac{7}{2}}} \, \Sigma_0(L)
\\
&\quad + \sqrt{2} \int_{-\frac{\pi}{2}}^{\frac{\pi}{2}} \dfrac{2kLs_0(L)^{2k-2}}{4\Omega_0^{(2)}(L)} \Bigl(s_0(L)\Omega_0^{(3)}(L)-(2k-1)\Omega_0^{(2)}(L)\Bigr) \dfrac{2\sqrt{2}}{(\Omega_0^{(2)}(L))^{\frac{3}{2}}}\, \mathrm{d}\vartheta
\\
&=- \dfrac{\pi Ls_0(L)^{2k}}{12 (\Omega_0^{(2)}(L))^{\frac{7}{2}}} \, \Sigma_0(L)
+ \dfrac{2k\pi Ls_0(L)^{2k-2}}{(\Omega_0^{(2)}(L))^{\frac{5}{2}}} \Bigl(s_0(L)\Omega_0^{(3)}(L)-(2k-1)\Omega_0^{(2)}(L)\Bigr)
\\
&= - \dfrac{\pi Ls_0(L)^{2k-2}}{12(\Omega_0^{(2)}(L))^{\frac{7}{2}}}\Bigl(s_0(L)^{2} \Sigma_0(L) -24k s_0(L)\Omega_0^{(2)}(L) \Omega_0^{(3)}(L)+24k(2k-1)(\Omega_0^{(2)}(L))^2\Bigr).
\end{align*}
The proposition is thus proved.
\end{proof}

\begin{remark}\label{rem-cond-extra-dL}
A careful inspection of the proofs shows that assumption \eqref{cond-extra-dL} is needed only to investigate the dependence of the time-maps on $L$. Thus, in particular, the first and the second formula of Proposition~\ref{prop-stime-tempo} are valid under the sole assumption \eqref{cond-extra-omega02}.
\hfill$\lhd$
\end{remark}

\bibliographystyle{elsart-num-sort}
\bibliography{BoDaFe-biblio}

\end{document}